\documentclass[12pt]{article}
\pdfoutput=1
\usepackage{mathtools}
\usepackage{amsthm}
\usepackage{mathrsfs}
\usepackage{amsmath, amsfonts, amssymb}
\usepackage{graphicx}
\usepackage{psfrag}
\usepackage[usenames,dvipsnames,svgnames,table]{xcolor}
\usepackage{enumerate}
\usepackage{jheppubcustom}
\usepackage{soul}
\usepackage{subfig}
\usepackage[utf8]{inputenc}
\usepackage{amsmath}
\usepackage{amsfonts}
\usepackage{amssymb}
\usepackage{amsthm}
\usepackage{color}

\usepackage{amsmath,amssymb,amsfonts,mathrsfs,hyperref,microtype}

\usepackage{ragged2e}
\usepackage{lipsum}

\usepackage{pgf,pgfarrows,pgfnodes,pgfautomata,pgfheaps,pgfshade}
\usepackage{amsmath,amssymb}
\usepackage[english]{babel}
\usepackage{pdfsync}
\usepackage{textcomp}
\usepackage{eurosym}
\usepackage{ragged2e} 

\usepackage{color}

\usepackage{amsmath}
\usepackage{amssymb}
\usepackage{latexsym}
\usepackage{enumerate}
\usepackage{amsthm}
\usepackage{dsfont}

\newcommand{\teal}{\color{teal}}

\numberwithin{equation}{section}

\newcommand{\R}{\mathbb{R}}
\newcommand{\E}{\mathbb{E}}

\newcommand{\D}{\mathbb{D}}

\newtheorem{lem}{Lemma}[section]
\newtheorem{rmk}{Remark}[section]

\usepackage[utf8]{inputenc}
\usepackage{chngcntr}
\usepackage{apptools}
\AtAppendix{\counterwithin{lemma}{section}}

\newtheorem{lemma}{Lemma}[section]

\newcommand{\be}{\begin{equation}}
\newcommand{\ee}{\end{equation}}

\def\be{\begin{equation}}
\def\ee{\end{equation}}
\def\beq{\begin{equation}}
\def\eeq{\end{equation}}

\def\diam{{\rm diam}}
\def\dist{{\rm dist}}

\def\P{{\mathbb P}}
\newtheorem{theorem}{Theorem}[section]
\newtheorem{corollary}[theorem]{Corollary}

\newtheorem{definition}[theorem]{Definition}

\newtheorem{remark}[theorem]{Remark}

\usepackage[final]{showlabels}

\def\E{\mathbb E }

\newcommand{\IR}{\mathbb{R}}

\newcommand{\IC}{\mathbb{C}}

\newcommand{\BLS}{\mathcal{L}}
\newcommand{\fdd}{\xRightarrow[f.d.d.]{}}
\def\simleq{\; \raise0.3ex\hbox{$<$\kern-0.75em
		\raise-1.1ex\hbox{$\sim$}}\; }
\def\simgeq{\; \raise0.3ex\hbox{$>$\kern-0.75em
		\raise-1.1ex\hbox{$\sim$}}\; }

\title{Brownian Loops, Layering Fields and \\ Imaginary Gaussian Multiplicative Chaos}

\author[\bigstar \spadesuit]{Federico Camia,}
\author[\bigstar]{Alberto Gandolfi,}
\author[\clubsuit]{Giovanni Peccati,}
\author[\bigstar ]{and Tulasi Ram Reddy}

\emailAdd{ federico.camia@nyu.edu, albertogandolfi@nyu.edu, giovanni.peccati@gmail.com, tulasi.math@gmail.com}

\affiliation[\bigstar]{\it New York University Abu Dhabi, United Arab Emirates}
\affiliation[\clubsuit]{\it Luxembourg University}
\affiliation[\spadesuit]{\it VU University, Amsterdam, The Netherlands}

\begin{document}

\abstract{We study fields reminiscent of vertex operators built from the Brownian loop soup in the limit as the loop soup intensity tends to infinity. More precisely,
following Camia, Gandolfi and Kleban (\emph{Nuclear Physics B} {\bf 902}, 2016), we take a (massless or massive) Brownian
loop soup in a planar domain and assign a random sign to each loop. We then consider random fields defined by taking, at every point of the domain, the exponential
of a purely imaginary constant times the sum of the signs associated to the loops that wind around that point. For domains conformally equivalent to a disk, the sum
diverges logarithmically due to the small loops, but we show that a suitable renormalization procedure allows to define the fields in an appropriate Sobolev space. Subsequently,
we let the intensity of the loop soup tend to infinity and prove that these vertex-like fields tend to a conformally covariant random field which can be expressed as an
explicit functional of the imaginary Gaussian multiplicative chaos with covariance kernel given by the Brownian loop measure.
Besides using properties of the Brownian loop soup and the Brownian loop measure, a main tool in our analysis is an explicit Wiener-It\^{o} chaos expansion of linear
functionals of vertex-like fields. Our methods apply to other variants of the model in which, for example, Brownian loops are replaced by disks.
}

%\abstract{ We consider  primary layering vertex fields controlled by measures on unrooted loops, primarily the massless and massive Brownian Loop Soup (BLS), and a disk model. We 
%show the existence of the fields in smooth bounded domains (and in the whole plane for the massive case) when the conformal dimensions satisfy $\Delta=\Delta(\lambda, \beta)<1/2$,
%where $\lambda$ is the intensity parameter  and $\beta$ the inverse temperature.  The  layering vertex fields are not free, which is to say they are not exponentials of Gaussians.
%
%We then use Wiener-It\^o chaos expansion to establish that, whatever the control, if $\Delta <1/4$ the $\lambda$-$\beta^2$-limit of the layering vertex field, i.e. the limit as the intensity $\lambda$  diverges and $\beta$ goes to $0$ so that $\lambda \beta^2$ is constant, is a purely imaginary Complex Gaussian Multiplicative Chaos (CGMC). 
%
%We finally observe that the $\lambda$-$\beta^2$-limit of the layering vertex fields controlled by massive BLS are the massive Gaussian Free Fields CGMC in the whole plane. This, however, fails for finite domains for all controls, including the massless BLS, in spite of this having a Gaussian, fully conformal invariant limit.}

\maketitle

\section{Introduction}

\subsection{Background and motivations}

In this article we are primarily concerned with the (massless and massive) \emph{Brownian loop soup} (BLS) \cite{Lawler,Lawler-Werner,werner,camia_notes} and the \emph{Gaussian multiplicative chaos} (GMC)
\cite{GMC_berestycki,GMC_imaginary,GMC_revisit,GMC_shamov,ComplexGMC,kahane}, two objects that have attracted
intense attention recently, because of their intrinsic interest and their numerous applications and connections to other models and areas of mathematics, including Euclidean
field theory \cite{CFT,henkel}, the Schramm-Loewner Evolution (SLE) \cite{conformal_weldings, sheffield}, the statistical analysis of eigenvalues of large random matrices and of zeros of the Riemann zeta function \cite{GMC_CUE_L1,GMC_CUE_counting,GMC_CUE_char_poly,GMC_GUE,GMC_zeta}, growth models \cite{growth},
disordered systems \cite{disorder}, the study of lattice models of statistical mechanics, of scaling limits of random planar maps and of the geometry of two-dimensional quantum gravity \cite{GMC_LQG,duplantier_sheffield, LQG_Riemann, DOZZ}.

More precisely, we study a two-parameter family of random generalized fields, denoted by $V^{*}_{\beta,\lambda}$, obtained from the BLS in a way that is inspired by the construction
of \emph{vertex operators} in conformal field theory (see, e.g., \cite{CFT}), as we explain below. The conformal properties of such fields were first analyzed in \cite{camia2016conformal},
which in turn was partly motivated by \cite{FK}; here, we show that, when $\beta \to 0$ and $\lambda \to \infty$ in an appropriate way, $V^{*}_{\beta,\lambda}$ converges
to a version of the \emph{imaginary} GMC \cite{ComplexGMC,GMC_imaginary}.

The BLS, introduced in \cite{Lawler-Werner}, is closely related to SLE and to the scaling limit of interfaces of statistical mechanical models at the critical point, such as the
critical Ising model. It is a Poisson process of planar loops with intensity measure given by a multiple $\lambda\mu^{loop}$ of the \emph{Brownian loop measure} $\mu^{loop}$.
The latter, studied by Werner in \cite{werner}, is a measure on simple loops in the plane, and on Riemann surfaces, which is conformally invariant and also invariant under
restriction (i.e., the measure on a Riemann surface $S'$ that is contained in another Riemann surface $S$ is just the measure on $S$ restricted to those loops that stay in $S'$).
Werner showed \cite{werner} that, up to a multiplicative constant, there exists a unique such measure. The Brownian loop measure is related \cite{werner} to the outer
boundaries of critical percolation clusters in the scaling limit and to the outer boundaries of Brownian loops, from which it derives its name.

In \cite{camia2016conformal}, Camia, Gandolfi and Kleban initiated the study of certain conformally covariant Euclidean fields obtained from the BLS.
Roughly speaking, the fields studied in \cite{camia2016conformal} are defined as exponentials, of the form $e^{i \beta N}$, of operators $N$ that compute properties
of the Brownian loop soup, where $\beta$ is a positive constant and $\lambda>0$ is the intensity of the Brownian loop soup. The operator $N(z)$ counts either
(i) the total number of windings of all loops around a point $z$, in which case $e^{i \beta N}$ is called a \emph{winding field}, or
(ii) the algebraic sum of distinct loops that disconnect $z$ from infinity, where each loop contributes with a random sign, in which case $e^{i \beta N}$ is called
a \emph{layering field}. This second type of operator was first introduced in~\cite{FK}, as a model for an eternally inflating universe, using a scale invariant Poisson
Boolean model instead of the BLS. In this paper, we call the model introduced in~\cite{FK} the ``disk model.'' The variants that use the BLS were introduced in
\cite{camia2016conformal} as conformally covariant versions of the original disk model. Whether or not such models will eventually find applications to cosmology,
they show interesting properties of their own, particularly those based on the BLS, and their study reveals analogies with the analysis of more classical Euclidean
field theories, in the tradition of constructive field theory, and in particular of conformal field theories. These analogies provide one of the motivations for our work.
Indeed, the families of fields studied here and in \cite{camia2016conformal} provide examples of nontrivial Euclidean field theories that are amenable to a mathematically
rigorous analysis.
There are not many examples of interesting conformally covariant fields that can be constructed and analyzed rigorously. The few classical examples, such as the Ising spin field \cite{CGN2015,CGN2016}, are difficult to handle and their study often requires sophisticated machinery (e.g., SLE and CLE). From this perspective, the layering and winding fields introduced in \cite{camia2016conformal}, with their Poissonian structure and their connection to Brownian motion, can be seen as a playground to develop techniques that could prove useful in the study of other models. Both the layering and the winding field present interesting and unusual features, such as a periodic spectrum of conformal dimensions. An explicit expression for the four-point function of the layering model was recently obtained in \cite{CFGL2019}.

Another motivation for our work is a calculation in \cite{FK} suggesting that, in infinite volume, the disk model should have a non-interacting (free-field) limit as the intensity
$\lambda$ of the Poisson Boolean model tends to infinity. In this paper, we consider the disk model in the unit disk and the layering models based on the BLS and massive
BLS in any planar domain conformally equivalent to the unit disk. For all such models, we rigorously show the existence of a Gaussian limit by proving convergence of the
corresponding layering fields to versions of the imaginary GMC. However, somewhat surprisingly, the limits we obtain are not free, meaning that, although the limiting fields
can be expressed as exponentials of Gaussian fields (i.e., GMCs), the covariance of those Gaussian fields is not given by the Green's function of the Laplacian.

A third motivation is to provide a new example of the deep and very rich connection between conformal field theories and conformally invariant stochastic models such as
the BLS. In doing so, we also prove some results about the Brownian loop measure $\mu^{loop}$ which may be of independent interest (see Section~\ref{s:plf}).

We also point out that our work provides a new construction of the imaginary GMC, starting with processes that are not Gaussian, and that the techniques introduced in our paper have the potential to find applications in the analysis of asymptotic limits of other random fields. Indeed, we believe that an important contribution of the present paper is the introduction of the Wiener-It\^{o} chaos expansion as a tool to study asymptotic limits of random fields. The chaos expansion turns out to be a very useful tool in this context, effectively reducing the asymptotic analysis of the fields $V^*_{\lambda,\beta}$ to that of their one-point functions, and allowing us to obtain our main results as a consequence of a multivariate central limit theorem.

\subsection{Main results}

The fields $e^{i \beta N}$ studied here and in~\cite{camia2016conformal,FK} are inspired by the \emph{vertex operators} defined as exponentials of free fields
times an imaginary coefficient (see, e.g. \cite{CFT,henkel}), and have similar properties of conformal covariance. However, one difficulty arises immediately in defining the
fields $e^{i \beta N}$, since $N(z)$ itself is not well-defined. Indeed, because the BLS is conformally (and therefore scale-)invariant, with probability one, infinitely many
loops disconnect from infinity any given point of the plane, making $N(z)$ a.\ s. infinite; the same occurs for the other models considered in this paper. It is easy to see that
the number of loops of diameter between $\delta>0$ and $R<\infty$ that disconnect from infinity any given point of the plane diverges \emph{logarithmically} as $\delta \to 0$
or $R \to \infty$. This explains why it is natural to consider fields of the form $e^{i \beta N}$, and why these need to be defined using cutoffs in order to make the number
of loops around a point $z$ finite. In general, two types of cutoff are necessary: an ``ultraviolet'' cutoff that removes small loops, and an ``infrared'' cutoff that removes large loops.

A particularly interesting way to introduce an infrared cutoff is to use the restriction property of the Brownian loop soup (and other Poissonian models) and consider only the loops
that are contained in a finite domain $D$ (or, by conformal invariance, a domain conformally equivalent to a finite domain). An ultraviolet cutoff $\delta>0$ can be introduced by keeping only loops with diameter at least $\delta$. The resulting random fields
$V^{\delta}_{\lambda,\beta}(z) = e^{i \beta N^{\delta}_{\lambda}}(z)$ are then well defined. As $\delta \to 0$, the $n$-point correlation functions of these fields,
$\delta^{-2\Delta^{loop}_{\lambda,\beta}n} \E(V^{\delta}_{\lambda,\beta}(z_1) \ldots V^{\delta}_{\lambda,\beta}(z_n))$, normalized by an appropriate power
$2\Delta^{loop}_{\lambda,\beta}$ (the factor of 2 is a convention, in keeping with the physics literature on conformal field theory) of the cutoff $\delta$ depending on the model,
are shown \cite{camia2016conformal} to converge to real-valued functions that depend on the domain $D$ in a conformally covariant way.
%The appropriate power turns out the be twice the conformal dimension $\Delta_{\lambda, \beta}$ with $\Delta_{\lambda, \beta}=\frac{\lambda}{10} (1-\cos{\beta})$
%for the massive and massless BLS, and $\Delta_{\lambda, \beta}=\pi \frac{\lambda}{2} (1-\cos{\beta})$ for the disk model (see \cite{FK}).

The existence of the limit as $\delta \to 0$ for the normalized fields $\delta^{-2\Delta^{loop}_{\lambda,\beta}} V^{\delta}_{\lambda,\beta}$ themselves,
was obtained first in~\cite{Spin_loops_berg_camia_lis} and later in \cite{lejan} for fields of the winding type, provided that the intensity of the BLS is not too large.
The limit exists in an appropriate Sobolev space of negative index, meaning that the limiting winding fields are random generalized functions, measurable with respect to the BLS
(see Theorem~5.1 of~\cite{Spin_loops_berg_camia_lis}). It is also shown in~\cite{Spin_loops_berg_camia_lis} that the limiting fields are not Gaussian fields
(see Proposition~5.2 of~\cite{Spin_loops_berg_camia_lis}).

In this paper, following \cite{Spin_loops_berg_camia_lis}, we prove the existence of the limiting layering fields,
which we denote $V_{\lambda,\beta,D}^{*}=V_{\lambda,\beta}^{*}$, obtained from $\delta^{-2\Delta^*_{\lambda,\beta}} V^{*,\delta}_{\lambda,\beta}$ as $\delta \to 0$.
In the notation, we often suppress $D$, which is considered fixed, while $*=loop,disk,m$ will distinguish between the three different models considered in this paper,
based respectively on the BLS of Lawler and Werner \cite{Lawler-Werner}, on the disk model mentioned above, and on the \emph{massive} Brownian loop soup, in which a
killing factor is introduced in the BLS (see \cite{camia_notes} for a detailed introduction). Each model has different characteristics: for $*=loop$ the model is scale and
conformally invariant, for $*=disk$ the model is scale but not conformally invariant, for $*=m$ the model is not even scale invariant since the killing rate introduces a scale.
Despite these differences, all models can be analyzed using our methods.
Each \emph{layering field} $V_{\lambda,\beta}^{*}$ is based on a Poisson process of loops (i.e., closed planar curves)
with intensity measure $\lambda\mu^{*}$, where $\mu^{*}$ is either the Brownian loop measure $\mu^{loop}$, the disk measure $d\mu^{disk}=\frac{dz dr}{r^3}$, where
$dz$ denotes the two-dimensional Lebesgue measure and $r$ is the radius of a random disk, or the \emph{massive} Brownian loop measure
\be
d\mu^m(\gamma) = e^{-\int_0^{t_{\gamma}} m^2(\gamma(t))dt} d\mu^{loop}(\gamma),
\ee
where $\gamma$ denotes a planar loop and the ``mass function'' $m: {\mathbb C} \to {\mathbb R}^{+}$ introduces a killing of Brownian loops according to their length.
As shown in Section~\ref{ss:poissononepoint} below, the divergence of the massive Brownian loop measure $\mu^{m}$ is the same as that of $\mu^{loop}$. This follows
from the fact that the divergence of both $\mu^{m}$ and $\mu^{loop}$ is due to small loops, whose measure is not affected much by a bounded mass function $m$.
As a consequence, the measure $\hat\mu$ defined by 
\begin{equation} \label{eq:muhat}
d\hat\mu(\gamma) = \Big( 1 - e^{-\int_0^{t_{\gamma}} m^2(\gamma(t))dt} \Big) d\mu^{loop}(\gamma)
\end{equation}
is finite over bounded domains.

For $\mu^{loop}$, $\mu^{disk}$ and $\mu^m$, we show that, as $\lambda \to \infty$ and $\beta \to 0$ in such a way that $\lambda\beta^2$ tends to some parameter
$\xi^2>0$, $V_{\lambda,\beta}^{*}$ converges to a ``tilted'' imaginary GMC \cite{ComplexGMC,GMC_imaginary}, as explained below.
Assuming that $X : D \to {\mathbb C}$ is a generalized Gaussian process with covariance kernel
\be \label{eq:covkernel}
K(z,w) := \E (X(z)X(w)) = \eta \log^{+}\frac{1}{|z-w|} + g(z,w),
\ee
where $\eta>0$ and $g$ is a bounded continuous function over $D \times D$, in this article, we say that a random generalized function $M_{\xi,D}$ is an
\emph{imaginary Gaussian multiplicative chaos} in $D \subset {\mathbb C}$ with parameter $\xi>0$, written
\be \label{eq:gmc}
M_{\xi,D} = e^{i \xi X},
\ee
if there exists a sequence $\{ X_{\delta} \}_{\delta}$ of Gaussian fields converging to $X$ such that the functions
\be
M_{\xi,D}^{\delta} (z) = e^{i \xi X_{\delta}(z) + \frac{\xi^2}{2} \E (X^2_{\delta}(z))},
\ee
are well defined and converge, as $\delta \to 0$, to $M_{\xi,D}$, where convergence can be understood in probability in some Sobolev space $\mathcal{H}^{-\alpha}(D)$
of negative index.

\begin{theorem}[Existence of the limit as a GMC] \label{MainTheorem}
Take $\lambda>0$ and $\beta \in [0,2\pi)$, let $*=loop, m$ and let $D \subsetneq {\mathbb C}$ be any domain conformally equivalent to the unit disk $\D$.
For each $z \in D$, let $d_z = \text{dist}(z,\partial D)$.
Then, if $D$ is a bounded simply connected domain with a $C^1$ smooth boundary, as $\lambda \to \infty$ and $\beta \to 0$ in such a way that
$\lambda\beta^2 \to \xi^2 < 5$, for every $\alpha>3/2$, $V^{*}_{\lambda, \beta}$ has a limit $W^{*}_{\xi,D}$ in the sense of convergence in distribution in the Sobolev space $\mathcal{H}^{-\alpha}(D)$ with the topology induced by the Sobolev norm. For more general domains conformally equivalent
to $\D$, the limit is in the sense of convergence of finite dimensional distributions.
Moreover,
\be \label{eq:rn-derivative}
\frac{dW^{*}_{\xi,D}}{dM_{\xi,D}^{*}}(z) = e^{-\frac{\xi^2}{2}\Theta^{*}_{D}(z)},
\ee
where
\begin{eqnarray*}
\Theta^{loop}_{D}(z) & = & \frac{1}{5} \log d_{z} + \mu^{loop}(\gamma: \diam(\gamma) \ge d_z, \gamma \text{ stays in $D$ and disconnects } z \text{ from } \partial D), \\
\Theta^{m}_{D}(z) & = & \Theta^{loop}_{D}(z) - \hat\mu(\gamma: \gamma \text{ stays in $D$ and disconnects } z \text{ from } \partial D)
\end{eqnarray*}
%$\varphi^{loop}_{D}(z) = \frac{1}{5} \log d_{z,D} + \alpha^{loop}_{d_{z},D}(z)$,
%$\varphi^{m}_{D}(z) = \varphi^{loop}_{D}(z) - \hat\alpha_{0,D}(z)$,
%$\varphi^{disk}_{D}(z) = \pi \log d_{z,D} + \alpha^{disk}_{d_{z,D},D}(z)$, 
and $M^{*}_{\xi,D}$ is an imaginary Gaussian multiplicative chaos in $D$ with parameter $\xi$ and covariance kernel
$K^{*}_D(z,w) = \mu^{*}(\gamma: \gamma \text{ is in $D$ and disconnects } z,w \text{ from } \partial D)$ for every $z,w \in D$.
The same conclusions hold for $*=disk$ with $D=\D$, $\xi^2<1/\pi$ and
\be \nonumber
\Theta^{disk}_{D}(z) = \pi \log d_{z} + \mu^{disk}(\gamma: \diam(\gamma) \ge d_z, \gamma \text{ stays in $D$ and disconnects } z \text{ from } \partial D).
\ee
\end{theorem}

\begin{proof}
The convergence in the sense of finite dimensional distributions is an immediate consequence of Theorems \ref{ExistenceGaussianLayeringField} and \ref{p:cil} below.
For bounded simply connected domains with a smooth $C^1$ boundary, combining the convergence in the sense of finite dimensional distributions with the results of
Section \ref{sec:tightness-uniqueness} gives convergence in distribution. For general domains, the theorem follows from Theorem \ref{thm:conformal_covariance}.
\end{proof}

\bigskip

\begin{remark}
{\rm \begin{enumerate}
\item[]{\empty}
\item The action of $W^{*}_{\xi,D}$ on $\varphi \in C^{\infty}_0(D)$ will be denoted by
\be
W^{*}_{\xi,D}(\varphi) = \int_D W^{*}_{\xi,D}(z) \, \varphi(z) \, dz,
\ee
and similarly for other (generalized) fields.
With this notation, Equation \eqref{eq:rn-derivative} means that, for any $\varphi \in C^{\infty}_0(D)$,
\be \label{eq:rn-derivative-expanded}
W^{*}_{\xi,D}(\varphi) = \int_D M_{\xi,D}^{*}(z) \, e^{-\frac{\xi^2}{2}\Theta^{*}_{D}(z)} \, \varphi(z) \, dz. %= M_{\xi,D}^{*} \big(e^{-\frac{\xi^2}{2}\Theta^{*}_{D}}\varphi \big).
\ee
\item We note that \eqref{eq:rn-derivative} is deterministic, so all the randomness of the field $W^{*}_{\xi,D}$ is captured by the imaginary GMC $M^{*}_{\xi,D}$.
\item Because of the definition of imaginary GMC, the theorem implicitly implies that
$\mu^{*}(\gamma: \gamma \text{ is in $D$ and disconnects } z,w \text{ from } \partial D)$ can be
written as $\eta \log^{+}\frac{1}{|z-w|} + g(z,w)$ for some $\eta$ and some bounded continuous function $g$. This
is proved in Lemma \ref{l:cont} below.
\item For $*=disk$, Theorem \ref{MainTheorem} is stated for the unit disk $\D$ for simplicity, but its conclusions can be easily extended to general bounded simply connected domains
with a $C^1$ boundary. The same considerations apply to Theorems \ref{ExistenceLayeringField} and \ref{ExistenceGaussianLayeringField} below, as one can easily deduce from
Theorems \ref{ExistenceLayeringFieldBoundedDomains} and \ref{ExistenceGaussianLayeringFieldBoundedDomains} in the appendix.
\item Although Theorem~\ref{ExistenceGaussianLayeringField} below guarantees the existence of $W^*_{\xi,D}$
for $\xi^2 \in (0,10)$ for $*=loop,m$ (respectively, $\xi^2 \in (0,2/\pi)$ for $*=disk$), we can prove Theorem \ref{MainTheorem} only for $\xi^2 \in (0,5)$ for $*=loop,m$ (respectively,
$\xi^2 \in (0,\pi)$ for $*=disk$). We do not know if this is an artifact of our methods or if the convergence result does not hold for $\xi^2 \geq 5$ (respectively, $\xi^2 \geq 1/\pi$).
\end{enumerate}
}
\end{remark}

For $*=loop$, and taking $D$ to be the upper half plane $\mathbb H$, an explicit formula for
$\mu^{loop}(\gamma: \gamma \text{ is in $\mathbb H$ and disconnects } z,w \text{ from } \partial{\mathbb H})$
%the convariance kernel $K^{loop}_D(z,w)$ of the Gaussian field $W^{loop}_{\xi}$
was predicted by Gamsa and Cardy \cite{GaCa06} using conformal field theory methods. The correct formula, which coincides with the Gamsa-Cardy prediction up to an
additive term, was obtained rigorously in \cite{HaWaZi19}. Given the conformal invariance of $\mu^{loop}$, from the formula for $D={\mathbb H}$, one can in principle obtain
a formula for $K^{loop}_D$ for any domain conformally equivalent to $\mathbb H$. In the case of the unit disk $\D$, one has the following formula (see \cite{HaWaZi19}):
\begin{eqnarray}
K^{loop}_{\D}(z,w) & = & \mu^{loop}(\gamma: \gamma \text{ is in $\D$ and disconnects } z,w \text{ from } \partial \D) \notag \\
& = & -\frac{1}{10} \Big[ \log\tilde\sigma + (1-\tilde\sigma) {_3}F_2\Big( 1, \frac{4}{3}, 1; \frac{5}{3}, 2; 1-\tilde\sigma \Big) \Big],
\end{eqnarray}
where ${_3}F_2$ is the hypergeometric function and $\tilde\sigma=\tilde\sigma(z,w)=\frac{|z-w|^2}{|1-zw|^2}$.

The limiting fields $W_{\xi,D}^*$ are conformally covariant in the sense of the result below.
\begin{theorem}[Conformal covariance] \label{TheoremConfomalCovariance}
Let $f: {\mathbb D} \to D$ be a conformal map from the unit disc $\mathbb D$ to a domain $D$ conformally equivalent to $\D$.
Let $m: \D \to {\mathbb R}^+$ be a mass function on $\D$ and define a mass function $\tilde m$ on $D$ by setting $\tilde m(f(w)) = |f'(w)|^{-1} m (w)$ for each $z=f(w) \in D$.
Then, with the notation of the previous theorem, we have that, for every $\varphi \in C^{\infty}_{0}(D)$,
\be
W^{loop}_{\xi,D}(\varphi) = \int_D W^{loop}_{\xi,D}(z) \varphi(w) dz \stackrel{d}{=} \int_{\D} |f'(w)|^{2 - \xi^2/10} W^{loop}_{\xi,{\D}}(w) \varphi \circ f(w) dw,
\ee
and
\be
W^{\tilde m}_{\xi,D}(\varphi) = \int_D W^{\tilde m}_{\xi,D}(z) \varphi(w) dz \stackrel{d}{=} \int_{\D} |f'(w)|^{2 - \xi^2/10} W^{m}_{\xi,{\D}}(w) \varphi \circ f(w) dw,
\ee
where $\stackrel{d}{=}$ stands for equality in distribution. Moreover, if $f$ is a M\"obius transformation of the unit disk into itself, then
\be
W^{disk}_{\xi,D}(\varphi) = \int_{D} W^{disk}_{\xi,D}(z) \varphi(w) dz \stackrel{d}{=} \int_{\D} |f'(w)|^{2 - \pi\xi^2/2} W^{ disk }_{\xi,{\D}}(w) \varphi \circ f(w) dw.
\ee
\end{theorem}

\begin{proof}
The theorem follows immediately from Theorem \ref{thm:conformal_covariance} and Remark \ref{remark:conf_inv} in the appendix.
\end{proof}

The result above can be expressed in the following way:
\be \label{eq:conformal_covariance_W}
W^{loop}_{\xi,D}(z) dz \stackrel{d}{=} |f'(w)|^{2 - \xi^2/10} W^{loop}_{\xi,\D}(w) dw,
\ee
\be \label{eq:conformal_covariance_Wm}
W^{\tilde m}_{\xi,D}(z) dz \stackrel{d}{=} |f'(w)|^{2 - \xi^2/10} W^{m}_{\xi,\D}(w) dw
\ee
and
\be \label{eq:conformal_covariance_Wdisk}
W^{disk}_{\xi,\D}(z) dz \stackrel{d}{=} |f'(w)|^{2 - \pi\xi^2/2} W^{disk}_{\xi,\D}(w) dw.
\ee
In keeping with the physics literature on conformal field theory, the powers to which $|f'(w)|$ is raised can be written as $2 - 2\Delta^{loop}_{\xi}$
and $2 - 2\Delta^{disk}_{\xi}$, where $\Delta^{loop}_{\xi}=\xi^2/20$ and $\Delta^{disk}_{\xi}=\pi\xi^2/4$ can be interpreted as the
\emph{conformal/scaling dimension} of the fields $W^*_{\xi,D}$.

\subsection{Relations to the existing literature}

\noindent {\it Brownian loop soup.} The use of random paths in the analysis of Euclidean field theories and statistical mechanical models was initiated by Symanzik,
who introduced a representation of the $\phi^4$ Euclidean field as a ``gas'' of weakly interacting random paths in his seminal work on Euclidean quantum field
theory~\cite{symanzik}. The approach was later developed by various authors, most notably %David
Brydges, %J\"urg
Fr\"ohlich, %Thomas
Spencer and %Alan
Sokal~\cite{bfsp,bfso}, and %Michael
Aizenman~\cite{aizenman} (see~\cite{ffs-book} for a comprehensive account).
The probabilistic analysis of Brownian and random walk paths and associated local times was carried out by Dynkin~\cite{dynkin1,dynkin2}.
More recently, ``gases'' or ``soups'' (i.e., Poissonian ensembles) of Brownian and random walk loops have been extensively studied in connection with SLE and the Gaussian free field
(see, e.g., \cite{werner1,Lawler-Werner,werner2,lejan1,lejan2,sznitman-notes,camia_notes}).

The Brownian loop soup \cite{Lawler-Werner} is, roughly speaking, a Poisson point process with intensity measure $\lambda\mu$,
where $\lambda$ is a positive constant and the Brownian loop measure $\mu$ \cite{werner} is uniquely determined (up to a multiplicative constant) by its \emph{conformal restriction property},
a combination of conformal invariance and the property that the measure in a subdomain is the original measure restricted to loops that stay in that subdomain). A realization of the loop soup
consists of a countable collection of loops. Given a bounded domain $D$, there is an infinite number of loops that stay in $D$; however, the number of loops in $D$ of diameter at least
$\varepsilon>0$ is finite. A consequence of conformal invariance is scale invariance: if $\cal A$ is a realization of the Brownian loop soup and each loop is scaled in space by
$1/N$ and in time by $1/N^2$, the resulting configuration also has the distribution of the Brownian loop soup.

%with intensity measure $\lambda\mu$, where $\lambda>0$ is a parameter (the \emph{intensity} of the soup)
%and $\mu$ is a measure on loops given by
%\begin{equation} \nonumber
%\mu = \int_{\mathbb C} \int_0^{\infty} \frac{1}{2 \pi t^2} \, \mu^{br}_{z,t} \, dt \, d{\bf A}(z) \, ,
%\end{equation}
%where $\bf A$ denotes area and $\mu^{br}_{z,t}$ is the Brownian bridge measure with starting point
%$z$ and duration $t$. It has attracted considerable attention because of its conformal invariance and
%relation with the Schramm-Loewner Evolution (SLE) \cite{schramm} and, in particular, with the Conformal
%Loop Ensembles (CLEs)|see \cite{sw}.

The Brownian loop soup exhibits a connectivity phase transition in the parameter $\lambda>0$ that
multiplies the intensity measure. When $\lambda \leq 1/2$, the loop soup in $D$ is composed of disjoint
clusters of loops \cite{werner1,werner2,sw} (where a cluster is a maximal collection of loops that intersect
each other). When $\lambda>1/2$, there is a unique cluster~\cite{werner1,werner2,sw} and the set of
points not surrounded by a loop is totally disconnected (see~\cite{erik_federico}). Furthermore, when
$\lambda \leq 1/2$, the outer boundaries of the loop soup clusters are distributed like Conformal Loop
Ensembles (CLE$_{\kappa}$) \cite{werner1,sheffield-cle,sw} with $8/3 < \kappa \leq 4$. The latter are
conjectured to describe the scaling limit of cluster boundaries in various critical models of statistical mechanics,
such as the critical Potts models for $q \in (1,4]$.

More precisely, if $8/3 < \kappa \leq 4$, then $0 < \frac{(3 \kappa -8)(6 - \kappa)}{4 \kappa} \leq 1/2$
and the collection of all outer boundaries of the clusters of the Brownian loop soup with
intensity parameter $\lambda = \frac{(3 \kappa -8)(6 - \kappa)}{4 \kappa}$ is distributed
like $\text{CLE}_{\kappa}$ \cite{sw}. For example, the continuum scaling limit
of the collection of all macroscopic boundaries of critical Ising spin clusters is conjectured
to correspond to $\text{CLE}_3$ and to a Brownian loop soup with $\lambda=1/4$.

\medskip

\noindent {\it Gaussian Multiplicative Chaos.} The theory of Gaussian multiplicative chaos was initiated by Kahane \cite{kahane} in an attempt to define rigorously
a random measure of the form $e^{\alpha\phi(x)}\sigma(dx)$ where $\alpha>0$ is a real parameter, $\phi$ is a log-correlated, centered Gaussian field on a
domain $D$ and $\sigma$ is an independent measure on $D$. The theory has been developed and expanded extensively in recent years by various authors
(see, e.g. \cite{GMC_revisit, GMC_LQG, GMC_shamov, ComplexGMC, GMC_berestycki}) and has applications that range from mathematical finance to the study
of turbulent flows (Kahane's work was motivated by a desire to make rigorous a model of energy dissipation in turbulent flows developed by Mandelbrot \cite{mandelbrot}).
A case that has recently attracted particular attention is that of dimension two when $\phi$ is the massless Gaussian free field (GFF) (i.e., the Gaussian field whose
covariance is given by the Green's function associated with the Laplace operator), called the Euclidean (bosonic) massless free field in quantum field theory. In this case,
the random measure is called Liouville Quantum Gravity and was first discussed by Polyakov \cite{Polyakov} and others \cite{BPZ_JSP, BPZ_NP}. The Complex
Gaussian Multiplicative Chaos (CGMC) \cite{ComplexGMC} corresponds to taking a complex parameter $\alpha$. In this case, $e^{\alpha\phi}\sigma(dx)$ is no
longer a measure but can still be defined as a random field. The special case in which $\alpha$ is purely imaginary has been studied in \cite{GMC_imaginary},
where it is shown to be related to the scaling limit of the spin field of the XOR-Ising model and to the sine-Gordon field. Moreover, as pointed out in \cite{GMC_imaginary}, 
at least at a formal level, the imaginary GCM appears to play a central role in the study of so-called imaginary geometry (see \cite{MS16}).

\bigskip

\noindent{\bf Notation.} In what follows, and for the rest of the paper, every random object is defined on a common probability space $(\Omega, \mathscr{G}, \mathbb{P})$. The expectation with respect to $\mathbb{P}$ is denoted by $\mathbb{E}(\cdot)$ or by the brackets $\langle \cdot \rangle$, according to notational convenience. Whenever required by the context, we will write $\langle \cdot \rangle_D$ instead of $\langle \cdot \rangle$, in order to indicate that some underlying random field is indexed by a given domain $D$. In several parts of the paper, we will deal with three specific random models based on the {\it Brownian loop soup}, the {\it disk model} and the {\it massive Brownian loop soup} on a given domain $D$;
to emphasize the dependence on the law of the underlying model, we will sometimes denote expectation with the symbols $\langle \cdot \rangle^{loop}_D$, $\langle \cdot \rangle^{disk}_D$ and $\langle \cdot \rangle^{m}_D$ ($m$ stands for `massive'), respectively, or with the symbol $\langle \cdot \rangle^*_D$ where $*$ stands for either $loop$, $disk$ or $m$. In the notation for the random fields we discuss in this paper, we will also sometimes include the domain $D$, but only when we want to emphasize the dependence of the fields on the domain; most of the time the domain $D$ can be considered fixed and will be suppressed from the notation. By an abuse of notation, we will sometimes treat objects that are not defined pointwise as if they were. For example, we will write
$\langle V^{*}_{\lambda,\beta}(z) \rangle$ for $\lim_{\delta \to 0} \langle \delta^{-2\Delta^{loop}_{\lambda,\beta}} V^{\delta}_{\lambda,\beta}(z) \rangle$ even though
$V^{*}_{\lambda,\beta}$ is not defined pointwise.
The lowercase letter $\omega$ always denotes generic elements of the set $\Omega$.
We will use $| \cdot |$ to denote Euclidean distance.

\bigskip

\noindent {\bf Organization of the paper.}
In the next section we give some preliminaries on Poisson layering fields and the Brownian loop soup. In Section \ref{construction}, following \cite{Spin_loops_berg_camia_lis},
we prove a convergence result for a general class of fields, which contains the fields studied in this paper. In Section \ref{s:plf}, we apply the result of the previous section to the
Poisson and Gaussian layering fields that are the main focus of this paper. We also characterize the Gaussian layering fields as ``tilted'' imaginary GMC. In Section \ref{sec:chaos},
we show that the fields constructed in the previous section admit a Wiener-It\^o chaos expansion. In Section \ref{s:convtoGMC}, using the Wiener-It\^o chaos expansion presented
in the previous section, we prove that the Poisson layering fields converge to the Gaussian layering fields in the appropriate limit of the parameters $\beta \to 0$ and $\lambda \to \infty$.
The Appendix contains various technical results, including some background on Sobolev spaces (Appendix \ref{a:sobolev}) and several of the proofs of the results of
Section \ref{ss:poissononepoint} (see Appendix \ref{ss:proofponepoint}). Appendix \ref{sec:one-point-functions} contains a useful representation for the one-point functions of Poisson
and Gaussian layering fields, while in Appendix \ref{sec:conf-inv}, we prove the existence and conformal covariance of Poisson and Gaussian layering fields for general domains.

\bigskip

\noindent {\bf Acknowledgments.} Part of this paper was written while Giovanni Peccati was visiting the Division of Science of New York University at Abu Dhabi, in February 2018. This author wishes to heartily thank Federico Camia and Alberto Gandolfi for their kind hospitality and support. Giovanni Peccati is also supported by the FNR grant FoRGES (R-AGR-3376-10) at Luxembourg University. Federico Camia thanks Wei Wu for interesting discussions on the GMC and related topics. The authors thank the associate editor and four anonymous referees for their constructive remarks and useful comments and suggestions, in particular for a question that led to an improvement of Theorem \ref{MainTheorem} and for providing a shorter proof of Lemma \ref{finitenessunitdisk}.

\section{Preliminaries on Poisson layering fields and the Brownian loop soup} \label{ss:layering_field}

In order to motivate the reader -- and to fix ideas -- we will now provide a succinct presentation of a class of {\it Poisson layering fields} associated with the Brownian loop soup. Understanding the limiting behavior of such fields -- that are special elements of the general class introduced in Section \ref{construction} -- is one of the main motivations of the present paper. The layering model underlying the definition of such fields first appeared in~\cite{camia2016conformal}, where properties of its $n$-point correlation functions were studied. As explained in the introduction, such layering model  is defined by first imposing a cutoff and then removing it via a limiting procedure, but only the existence of the limit for the $n$-point functions is proved in~\cite{camia2016conformal}. The existence of a layering \emph{field}, i.e., a random generalized function whose $n$-point functions are those derived in~\cite{camia2016conformal}, will follow from Theorem~\ref{convergence} below. We observe that the Poisson layering fields discussed in this section have to be distinguished from the {\it Gaussian layering fields} introduced in Section~\ref{ss:GaussianLayering}.
	
\subsection{Loops} \label{sec:loops}
	
%	 For measures on 
% loops we adopt the definitions and notation of \cite{erik_federico}. In particular, we fix  a
% bounded, simply connected domain  $D$ in $\mathbb R^2$.
%Then we consider the  set $M=M_D$ of connected compact
%subsets of $\mathbb R^2$ with nonempty interior contained in $D$, and
%the $\sigma$-algebra $\mathcal M=\mathcal M_D$ on $M$
%{\teal generated by} all subsets of $M$ of the form
%\begin{eqnarray} \label{measurability}
%\{ \gamma : \gamma \in M, \delta < \diam(\gamma) \leq R, \gamma \subseteq C\}
%\end{eqnarray}
%for all $\delta, R \in \mathbb R^+$ and every Borel subset $C$ of $\mathbb R^2$. 
%Notice that the boundary of each subset $\gamma $ in $M$ determines a 
%closed, unrooted 
%loop contained in {\teal $D$} {\teal and viceversa}. {It follows that we can identify $M$ with the set of unrooted loops contained in $D$, and we will refer to a} measure $\mu^*=\mu^*_D$ 
%on $(M_D, \mathcal M_D)$ as a  {\bf measure on unrooted loops}. 

In this section (and for the rest of the paper) we identify
	$\mathbb C$ and ${\mathbb R}^2$, and we call any connected open subset of $\mathbb C$ a \emph{domain}.
	A \emph{(rooted) loop} $\gamma$ of time length $t_{\gamma}$ is a continuous map $\gamma : [0,t_{\gamma}] \to \IC$ with $\gamma(0)=\gamma(t_\gamma)$.
	Given a domain $D \subset \IC$, a conformal map $f : D \to \IC$, and a loop $\gamma \subset D$ (that is, such that the image of $\gamma$ is contained in $D$),
	we define $f \circ \gamma$ to be the loop $f (\gamma)$ with time parametrization given by the Brownian scaling $f \circ \gamma (s) = f(\gamma(t))$, where $s$ is defined by the relation
	\begin{equation}
	\quad s= s(t) = \int_0^{t}  | f'(\gamma(u))|^2 du,
	\end{equation}
	and $t_{f\circ \gamma} = s(t_{\gamma})$.
	As a special case, consider the mapping $\Phi_{a,b}(w)=aw+b$, $a \neq 0$: then, $\Phi_{a,b} \circ \gamma$ is the loop $\gamma$ scaled by $|a|$, rotated around the origin by $\arg a$ and
	shifted by~$b$, with time parametrization $s(t)= |a|^2t$, and  length $t_{\Phi_{a,b}\circ \gamma} = |a|^2 t_{\gamma} $. 
	An \emph{unrooted loop} is an equivalence class of loops under the equivalence relation $\gamma \sim \theta_r \gamma $ for every $r \in \IR$, where $\theta_r \gamma(s) = \gamma(s+r \ \text{mod} \ t_{\gamma})$.
We denote by $M$ the class of unrooted loops in $\R^2$ and, for a given domain $D$, we write $M_D$ to indicate the class of unrooted loops in $D$. As explained in \cite[Section 4.1]{Lawler-Werner}, the class $M$ can be made into a metric space, and the corresponding Borel $\sigma$-field is denoted by $\mathcal{M}$. For a given domain $D$, the trace $\sigma$-field of $\mathcal{M}$ associated with $M_D$ is written $\mathcal{M}_D$ (note that $\mathcal{M}_D$ is a $\sigma$-field of subsets of $M_D$). %We will often use the fact that
For our purposes, we can work with any $\sigma$-field $\mathcal{M}$ (resp. $\mathcal{M}_D$) that contains all sets of the form
\begin{eqnarray} \label{measurability}
\{ \gamma \in M: \delta \leq \diam(\gamma) \leq R, \gamma \text{ disconnects } z \text{ from } \infty \},
\end{eqnarray}	
where $\delta, R\geq 0$ and $z \in \IC$, and
%$C$ is a Borel subset of $\R^2$ (resp. a Borel subset of $D$). 
\begin{eqnarray} \label{measurability-bis}
\{ \gamma \in M: \gamma \not \subseteq U, \gamma \subseteq V, \gamma \text{ disconnects } z \text{ from } \infty),
\end{eqnarray}	
where $U \subseteq V$ are two Borel subsets of $\R^2$ (resp. a Borel subset of $D$) and $z \in U$.
Any measure $\mu$ on $(M, \mathcal{M})$ or $(M_D, \mathcal{M}_D)$ is called a {\it measure on unrooted loops}.

\subsection{The Brownian loop measure and the Brownian loop soup}\label{ss:loops}
We use the symbol $\mu^{br}$ to denote the \emph{complex Brownian bridge} measure, i.e., the probability measure on loops rooted at~$0$ and of time length $1$, induced by the process $B_t=W_t-tW_1$, ${t\in[0,1]}$, where $W_t$ is a standard complex Brownian motion initialized at $0$. For $z\in \IC$ and $t>0$, we write $\mu^{br}_{z,t}$ to indicate the complex Brownian bridge measure on loops rooted at $z$ of time length $t$, i.e., the measure
	\begin{equation}
	\mu^{br}_{z,t} = \mu^{br} \circ \Phi_{\sqrt{t},z}^{-1}.
	\end{equation}
	The \emph{Brownian loop measure} is a $\sigma$-finite measure on loops given by
	\begin{equation} \label{eq:BLMeasure}
	\mu^{loop} = \int_{\IC} \int_0^{\infty} \frac{1}{2\pi t^2} \, \mu^{br}_{z,t} \, dt \, dz,
	\end{equation}
where $dz$ denotes the Lebesgue measure on $\IC$ which, for our purposes and as already mentioned, can be identified with $\R^2$.
This measure is clearly translation invariant and it is easy to check that it is scale invariant, meaning that $\mu^{loop} = \mu^{loop} \circ \Phi_{a,b}$ for any $a>0$ and $b\in\IC$.
Since $\mu^{loop}$ inherits rotational invariance from the complex Brownian motion, we actually have that $\mu^{loop} = \mu^{loop} \circ \Phi_{a,b}$ for any $a,b\in\IC$, $a\neq 0$.
To recover the full \emph{conformal invariance} of Brownian motion one has to regard the push-forward measure of $\mu^{loop}$ on the space $(M, \mathcal{M})$ of \emph{unrooted loops} defined above (as it is customary, and by a slight abuse of notation, we continue to use the symbol $\mu^{loop}$ for such a push-forward measure). If $D$ is a domain, then by $\mu^{loop}_D$ we denote the measure $\mu^{loop}$ restricted to $(M_D, \mathcal{M}_D)$. Let $D,D'$ be two simply connected domains, and let $f: D\to D'$ be a conformal equivalence. The full conformal invariance of $\mu^{loop}$ is expressed by the fact that $\mu^{loop}_{D'} \circ f =\mu^{loop}_D$. A proof of this fact can be found in~\cite{Lawler-Werner,Lawler}.
	
The \emph{Brownian loop soup} $\BLS_D=\BLS_{D,\lambda}$ with intensity parameter $\lambda >0$ is a Poissonian collection of unrooted loops with intensity measure $\lambda \mu^{loop}_D$. We write $\BLS=\BLS_{\IC}$. The Brownian loop soup inherits all invariance properties of the Brownian loop measure. In particular, for $D, D'$ and $f$ as above, one has that $\BLS_{D'}$ has the same distribution as the image of $\BLS_{D}$ under $f$.

\subsection{Poisson layering fields}	
	
Following \cite{camia2016conformal}, we now define the layering model arising from the Brownian loop soup.
Let $\bar{\gamma}$ be the \emph{hull} of the loop $\gamma$, i.e., the complement of the unique unbounded connected component of the complement of $\gamma$.
Regarding $\gamma$ as a subset of $\IC$, we say that $\gamma$ \emph{covers} $z$ if $z \in \bar{\gamma}$. We restrict our attention to bounded domains $D$ and declare
each loop in $\BLS_{D}$ to be of type $1$ or type $2$ with equal probability, independently of all other loops. For any $z \in D$, the scale invariance of the BLS implies that
the number of loops of each type covering $z \in D$ is almost surely infinite. Because of this, we introduce an ultraviolet cutoff $\delta>0$
on the size of the loops and let
\begin{equation}
\BLS^{\delta}_D =\{ \gamma \in \BLS_D:  \diam(\gamma) > \delta \}.
\end{equation}
	
Similarly, we let $\mu^{\delta}_D$ denote the measure $\mu_D$ restricted to loops of diameter larger than $\delta$. We note that $\mu^{\delta}_D$ is a \emph{finite} measure
when $D$ is bounded and $\delta>0$. This implies that, almost surely, each point $z$ is covered by only finitely many loops from $\BLS^{\delta}_D$, so that the number
$N^{\delta}_j(z)$ of loops of type $j=1,2$ covering $z \in D$ is almost surely finite. We can now define
\begin{align} \label{eq:winddef}
{N_{D}^{\delta}(z) } := N^{\delta}_1(z) - N^{\delta}_2(z), \qquad \forall z \in D.  
\end{align}
Note that the loops which do not cover $z$ do not contribute to $N^{\delta}_D(z)$, and therefore $N^{\delta}_D(z)$ is finite almost surely.

The \emph{layering field with cutoff} $\delta>0$ and parameters $\lambda, \beta >0$ is defined by
\begin{equation}
V^{\delta}(z) =V^{\delta}_{\lambda, \beta}(z) ={ e^{i \beta N_{D}^{\delta}(z)}}.
\end{equation}
The correlation functions of such random fields were explicitly studied in \cite[Theorem 4.1]{camia2016conformal} in the limit $\delta \to 0$.
%	Whenever the young woman swooned, she always seemed to manage falling into the arms of a good-looking man.{\teal\it (Giovanni says: ``Who is this woman? This is a test to check that I read everything, I am sure ! :) '')}
There, it was proved that the one-point function $\langle V^{\delta}(z) \rangle_D = \E(V^{\delta}(z))$ decays like $\delta^{2\Delta}$, where
\begin{equation}
\Delta = \Delta^{loop}_{\lambda,\beta} =\frac{\lambda}{10}(1-\cos\beta).
\end{equation}

%(Following \cite{camia2016conformal}, we will use $\langle \cdot \rangle$ to denote expectation. This is common practice in the physics literature.)
It is therefore natural to study the ``renormalized'' field $ \delta^{-2\Delta}V_{\lambda,\beta}^{\delta}$ as $\delta \to 0$.
Observe however that, as $\delta \to 0$ and in view of the equality $|\delta^{-2\Delta} V_{\lambda,\beta}^{\delta}(z) | = \delta^{-2\Delta}$, the field $ \delta^{-2\Delta}V_{\lambda,\beta}^{\delta}$ does not converge as a function on $D$.  As a consequence, in order to deduce well-defined convergence results, one has to regard $ \delta^{-2\Delta}V_{\lambda,\beta}^{\delta}$ as an element of a topological space larger than any classical function space. This is usually achieved by regarding $ \delta^{-2\Delta}V_{\lambda,\beta}^{\delta}$ as a random generalized function, i.e., a random continuous functional on some appropriately chosen space of test functions where the action of $ \delta^{-2\Delta}V_{\lambda,\beta}^{\delta}$ on a test function $\varphi$ is given by
\begin{equation}
\delta^{-2\Delta}V_{\lambda,\beta}^{\delta}(\varphi) =\delta^{-2\Delta} \int_D  V_{\lambda,\beta}^{\delta}(z) \varphi(z) dz.
\end{equation}
Such a strategy was successfully implemented in \cite[Section 5]{Spin_loops_berg_camia_lis} for winding fields. The aim of the next section is to establish a class of general convergence results, extending the conclusions \cite{Spin_loops_berg_camia_lis} to a large class of layering models.

\section{A general class of fields} \label{construction}

In this chapter we set up a general framework for dealing simultaneously with random fields which are exponentials of Poisson or Gaussian random measures
having a $\sigma$-finite control with infinite mass.

\subsection{Random measures}\label{ss:concepts}

Let $(A, \mathscr{A})$ be a measurable space, and let $\nu$ be a $\sigma$-finite positive measure on $(A, \mathscr{A})$. We will assume that $\nu(A) = +\infty$, and also that $\nu$ has no atoms (at the cost of some technicalities, one can remove such an assumption with essentially no bearings on our results).

%\begin{remark}{\rm
%In models based on a Poisson process, including the layering and the disc models described below, it is convenient to consider an additional parameter $\lambda>0$ {regulating the `intensity' of the Poisson process, so that the intensity measure of the Poisson process will be often written as $\lambda\nu$. This will allow us to work with a fixed measure $\nu$ while letting the intensity $\lambda$ diverge. In the Gaussian case related to CGMC, the intensity will be absorbed into other parameters of the model, so $\lambda$ will be set equal to $1$.
%}}
%\end{remark}

\medskip

For every $\lambda>0$, we write 
\begin{eqnarray} \label{Poisson}
N_\lambda = \{ N_\lambda(B) : B\in \mathscr{A}\}
\end{eqnarray}
to indicate a {\it Poisson measure} on $(A, \mathscr{A})$ with intensity $\lambda \nu(\cdot)$. Remember that the distribution of $N_\lambda$ is completely characterized by the following two properties: (a) for every collection $B_1,...,B_n\in \mathscr{A}$ of pairwise disjoint sets, one has that the random variables $N_\lambda(B_1),..., N_\lambda(B_n)$ are stochastically independent, and (b) for every $B\in \mathscr{A}$ the random variable $N_\lambda(B)$ is distributed according to a Poisson law with parameter $\lambda\nu(B)$; {here, we adopt the standard convention that a Poisson random variable with parameter 0 (resp. $\infty$) equals 0 (resp. $\infty$) with probability one}. It is easily checked that, for $\P$-almost every $\omega$, the mapping $B\mapsto N_\lambda(B)(\omega)$ defines an integer-valued measure on $(A, \mathscr{A})$ such that, for every $a\in A$, the value of $N_\lambda(\{a\})(\omega)$ is either 0 or 1 (this last property comes from the fact that $\nu$ has no atoms). It is also convenient to introduce the quantity
	$$
	\widehat{N}_\lambda (B):= N_\lambda(B) - { \lambda} \nu(B), \quad \mbox{for every $B$ such that $\nu(B)<\infty$}.
	$$
	The mapping $B\mapsto \widehat{N}_\lambda (B)$ is called a {\it compensated Poisson measure}.

\medskip

We will also denote by 
\begin{eqnarray} \label{Gauss}
G = \{G(B) : \nu(B)<\infty\}
\end{eqnarray}
{a {\it Gaussian measure with control} $\nu$ on $(A, \mathscr{A})$. This means that $G$ is a centered Gaussian family, indexed by the elements $B$ of $\mathscr{A}$
	such that $\nu(B)<\infty$, with covariance $\E[G(B)G(C)] = \nu(B\cap C)$.}

\medskip 

 Background material on Poisson and Gaussian measures can be found in \cite{Last, Last_Penrose_PTRF, Last_Penrose_lectures, Gio_ivan, Gio_murad_taqqu}. We record the following facts (see e.g. \cite[Chapter 5]{Gio_murad_taqqu}): 
	
	\begin{itemize}
		
		\item[--] if $h\geq 0$ is measurable, then the integral $\int_A h\, dN_\lambda := N_\lambda(h)$ is $\P$-a.s. well-defined (since $N_\lambda$ is $\P$-a.s. a measure on $(A, \mathscr{A})$);
		
		\item[--] if $h\in L^1(\nu)$, then the integral $\int_A h\, dN_\lambda= N_\lambda(h) := N_\lambda(h_+) -N_\lambda(h_-) $ is $\P$-a.s. well-defined, and verifies the fundamental relation $\E[N_\lambda(h)] =\lambda \int_A h\, d\nu$;
		
		\item[--] If $h\in L^2(\nu)$, then the two stochastic integrals 
		\be
		\int_A h(x) G(dx) =: G(h)\mbox{  and  }  \int_A h(x) \widehat{N}_\lambda (dx) =: \widehat{N}_\lambda(h)
		\ee
		are well defined as limits in $L^2(\P)$ of finite sums of the type $\sum_k b_k  Z(B_k)$, where $b_k\in\R$, $B_k\in \mathscr{A}$ and $Z = G, \widehat{N}_\lambda$. In this case, one has also the crucial isometric relations
		\be
		\E[G(h)G(f)] = \frac1\lambda \E[ \widehat{N}_\lambda(h) \widehat{N}_\lambda(f)] = \int_A f(x)h(x) \nu(dx),
		\ee
		for every $f,h\in L^2(\nu)$.
	\end{itemize}

%Any measure $X$ as above, where $X$ can be $N_\lambda$ or $G$,
% can be extended to real valued measurable functions
%$h$ defined on $A$ such that $\mu(h) = \int_A |h(x)| d\mu(x) < \infty$
%by taking $X(\sum_{I=1}^n b_i B_i):=
%\sum_{I=1}^n b_i X(B_i)$ and then taking limits; the extensions 
%are indicated by $N_{\lambda}(h)$ and $G(h)$ {\red [is this correct? ref??]}.

%We consider  a simply connected domain  $D$ of $\R^2$, with a $C^1$ boundary. Our objects of interest are the vanishing cut-off limits of
% the layering fields introduced in  [Kleban, Vreivogel] and [Camia Gand Kleban] and of the  related  gaussian fields; with cut-off in place, these are defined by means of 

\subsection{Kernels}\label{ss:generalkernels}

\medskip

{\it We keep the notation from the previous subsection and, for the rest of Section \ref{construction}, we let $D$ denote a domain of $\IC$, which we identify with ${\mathbb R}^2$. }

\bigskip

\noindent We write $B_{w,r}$ to indicate the open disk of radius $r>0$
centered at $w\in \R^2$. Two {distinguished} points in $\R^2$ are the origin, denoted by ${\bf 0}$,
and the point $(0,1)$ denoted by ${\bf 1}$. For $z,w \in D$, we define
\begin{eqnarray} 
d_z&=& d(z, {\partial} D) \label{distance1}\\
d_{z,w}&=& \min(d(z, { \partial} D),d(z, w)) \label{distance2},
\end{eqnarray}
where $\partial D$ stands for the boundary of $D$, and $d$ is the Euclidean distance. We also indicate the diameter of a subset $U$ of $D$ by $\diam(U)$.
%\begin{eqnarray} \label{diameter}
%diam(U) \quad \mbox{or}\quad d(U).
%\end{eqnarray}

\medskip

In what follows, we consider collections of jointly measurable kernels of the type 
	\begin{equation}\label{e:hkernels}
	{\bf h} :=  \big\{(z,x) \mapsto h^\delta(z,x) = h_z^\delta(x) : \delta \in (0, \delta_0], \,\, (z,x) \in D\times A \big\},
	\end{equation}
	where  $\delta_0>0$ is a fixed parameter (possibly infinite, in which case $\delta \in (0, \delta_0]$ has to be read as $\delta>0$). The objects appearing in \eqref{e:hkernels} have the property of {\it  locally exploding in the $\delta\to 0$ limit}, that is: for every $z\in D$,
$x \mapsto h_z^\delta(x) \in L^1(\nu)\cap  L^2(\nu)$ for $\delta >0$, and, as $\delta \to 0$, the kernel $h_z^\delta$ converges $\nu$-almost everywhere to some $h_z$ such that  $\nu(| h_z | ) = \nu(| h_z |^2 ) = \infty$, where we adopted the standard notation $\nu(g) := \int_A  g(x) \nu(dx)$.

\medskip

\begin{definition}[Small parameter families] \label{def:spf}
{\rm Given an exploding collection of kernels ${\bf h} $ as in \eqref{e:hkernels}, a {\it small parameter family of $\nu$ controlled complex random fields} associated with ${\bf h}$
	is defined to be a collection of random objects of the type 
	\begin{equation}\label{e:v} 
	\mathbb{U} = \mathbb{U}({\bf h}) = \{U^\delta : \delta\in (0, \delta_0]\}, 
	\end{equation}
	where
\begin{equation}\label{e:spf}
 U^\delta = \{ U^{\delta}(z)=e^{i \zeta  X(h_z^\delta)} : z \in D\},
\end{equation} 
$\zeta \in \mathbb R$, and $X$ is either a Gaussian measure $G$ with control $\nu$, or a Poisson measure $N_\lambda$ with intensity $\lambda\nu$, for some $\lambda>0$.
It will be always implicitly assumed that the kernels in ${\bf h}$ are such that the mapping $(\omega, z) \mapsto U^\delta(z)(\omega)$ is jointly measurable (such an assumption is needed e.g. to apply Fubini-type arguments).
}
\end{definition}

\begin{remark}{\rm In the case of the layering fields introduced in Section~\ref{ss:layering_field}, $A$ is the space of {\it marked} loops, where each loop $\gamma$ appears in
	two copies, each with a {\it mark} $\epsilon=1$ or $-1$. In this context, letting
	$x = (\gamma, \epsilon) \in A$ and $A_{\delta}(z) = \{ \gamma : z \in  \bar\gamma, \diam(\gamma)>\delta \}$,
	$\nu$ is $\frac12 \mu^{loop}$ times the uniform measure on $\{-1,1\}$ and $h_z^{\delta}(x) = \epsilon {\bf 1}_{A_{\delta}(z)}(\gamma)$.
	As a consequence, $X(h_z^\delta)=N_{\lambda}(h_z^\delta)$ is the difference of two independent Poisson r.v.'s with intensity equal to half of the $\mu^{loop}$-mass
	of the loops surrounding $z$ that are contained in $D$ and have radius larger than $\delta$. See Section \ref{s:plf} for more details.}
	\end{remark}

The ``explosive'' nature of the kernels ${\bf h}$ of interest implies that $X(h_z^\delta)$ is not well defined for $\delta=0$ and that some care is needed when removing the cutoff $\delta$. As $\delta \to 0$, the existence of the limit of the fields $U^{\delta}$, when normalized by an appropriate power of the cutoff $\delta$, is discussed in the next section.

\subsection{A general existence result} \label{ss:existence}

In this section we identify general conditions for the existence of the limiting fields of a small parameter family of $\nu$ controlled complex random fields,
 as the {cutoff} fields are renormalized and the {cutoff} is removed. 

The next theorem is a generalization of Theorem~5.1 of~\cite{Spin_loops_berg_camia_lis}. Our formulation of Theorem \ref{convergence} has been devised in order to facilitate the connection with the formalism adopted in \cite{Spin_loops_berg_camia_lis}. For the convenience of the reader, definitions and basic facts about negative Sobolev spaces are gathered together in Section \ref{a:sobolev}.

\begin{theorem} \label{convergence}
Let $D$ be a bounded simply connected domain with a $C^1$ boundary and
let $\mathbb{U} = \{U^\delta(z) :\delta\in (0,\delta_0], z\in D\}$ be a small parameter family of $\nu$ controlled complex
random fields, as in \eqref{e:hkernels}--\eqref{e:spf}, such that
\begin{equation} \label{c:1}
U^\delta \in L^2(D,dz) \text{ for all } \delta \in (0,\delta_0] \, \mbox{ $\mathbb{P}$-a.s.}
\end{equation}
Assume that there exist $\Delta \in (0,1/2)$ and a real-valued nonnegative function $\phi$ with domain $D$ such that, for every $z, w \in D$ with $z \neq w$,
\begin{equation}
\lim_{\delta,\delta' \to 0} (\delta \delta')^{-2\Delta} \langle U^\delta(z)\overline{U^{\delta'}(w)}\rangle = \phi(z,w). \label{two_point_function_exists}
\end{equation}
Moreover, suppose that there exist a constant $c(D)<\infty$ and a real-valued nonnegative function
\begin{align*}
\mathbb{R}^+\times D\times D &\ni (\delta, z, w) \mapsto \tau_{z,w}^\delta %\quad\mbox{and}  \\
%\mathbb{R}^+\times\mathbb{R}^+\times D\times D &\ni (\delta,\delta', z, w) \mapsto \tau_{z,w}^{\delta',\delta} 
\end{align*}
such that, for $0 < \delta' \leq \delta$,
\begin{equation}
0 \leq \langle U^\delta(z)\overline{U^{\delta'}(w)}\rangle \leq e^{-(\tau_{z,w}^\delta+\tau_{w,z}^{\delta'})} \label{two_point_function}
\end{equation}
with
\be
%    \lim\limits_{\delta \to 0} \frac{e^{-\tau^{\delta}_{z,w}} }{\delta^{2 \Delta}} &=&:\phi(z,w)
% \text{ exists  for every } z, w \in D, \label{2.3bis} \\
e^{-\tau^{\delta}_{z,w}} \leq c(D) \left(\frac{ d_{z,w}}{\delta}\right)^{-2 \Delta} \text{ for all }  \delta<d_{z,w}. \label{2.3}
%	\tau^{\delta', \delta}_{w,z}&=& 0 \text{ for } \delta' < \delta < |z-w| \label{2.4}
\ee
with $d_{z,w}$ as in \eqref{distance2}.
Then, for every $\alpha >\frac{3}{2}$, the cutoff field $z\mapsto \delta^{-2\Delta} U^\delta(z)$ converges as $\delta \rightarrow 0$ in
second mean in the Sobolev space $\mathcal{H}^{-\alpha}$, in the sense that there exists a $\mathcal{H}^{-\alpha}$-valued random distribution $U$,
measurable with respect to {the $\sigma$-field generated by $X$ (as appearing in \eqref{e:spf})} and such that
\begin{equation} \label{2.4bis}
\lim\limits_{\delta \rightarrow 0}\langle \|\delta^{-2\Delta}U^\delta-U\|^2_{\mathcal{H}^{-\alpha}}\rangle = 0.
\end{equation}
\end{theorem}

\begin{proof}
The proof is obtained  by extending the techniques used in the proofs of Theorem 5.1 and Proposition 5.3 of \cite{Spin_loops_berg_camia_lis}.
For any $z\neq w$, the limit $ \lim\limits_{\delta,\delta' \rightarrow 0}(\delta\delta')^{-2\Delta}\langle U^\delta(z)\overline{U^{\delta'}(w)}\rangle$ exists pointwise
by \eqref{two_point_function_exists}.
We will now show that convergence in $L^1(D\times D)$ also holds, by virtue of the Bounded Convergence Theorem. For this purpose, note that if $\delta \geq d_{z,w}$ we have
$
\delta^{-2\Delta}e^{-\tau^\delta_{z,w}} \leq d_{z,w}^{-2\Delta} \label{large_delta}
$  
as $\tau_{z,w}^\delta \geq 0$; then,
for any $z,w \in D$ {and any $\delta, \delta'$}, from \eqref{two_point_function}, \eqref{2.3}, \eqref{large_delta}, we infer that
\begin{eqnarray}
(\delta\delta')^{-2\Delta}\langle U^\delta(z)\overline{U^{\delta'}(w)}\rangle & \leq & (\delta\delta')^{-2\Delta}e^{-(\tau_{z,w}^\delta+\tau_{w,z}^{\delta'})}\\
%& \leq c(D)^2(d_{z,w}d_{w,z})^{-2\Delta}e^{-(\tau_{z,w}^{d_{z,w}}+\tau_{w,z}^{d_{w,z}})},\\
& \leq & (c(D)\vee 1)^2(d_{z,w}d_{w,z})^{-2\Delta}.
\end{eqnarray}
Hence, in order to prove $L^{1}(D \times D)$ convergence, it suffices to show that
\begin{equation} \label{convergent_integral}
\int_{D}\int_{D}(d_{z,w}d_{w,z})^{-2\Delta} dz dw < \infty.
\end{equation}
Denote by $\mathbb{D}$ the open unit disk. Let $f:\mathbb{D} \rightarrow D$ be a conformal equivalence map and $d_{z,w}^\mathbb{D}=d(z,\partial \mathbb{D}\cup\{w\})=|z-w|\wedge(1-|z|)$. From the  Koebe quarter theorem  \cite[Theorem 1.3]{pommerenke}, $d_{z,w}^{\mathbb{D}}|f'(z)|\leq4d_{f(z),f(w)}$. Note that, since $D$ has a $C^1$ smooth boundary, $f'$ has a continuous extension on to $\overline{\mathbb{D}}$ and hence $\|f'\|_{L^\infty(\mathbb{D})}<\infty$, see  \cite[Theorem 3.5]{pommerenke}. By the corresponding change of variables, the integral in \eqref{convergent_integral} is equal to
\begin{align}
\int_{\mathbb{D}}\int_{\mathbb{D}}(d_{f(z),f(w)}d_{f(w),f(z)})^{-2\Delta}|f'(z)f'(w)|^2 dz dw \\
\leq 4^{4\Delta}\|f'\|^{4-4\Delta}_{L^\infty(\mathbb{D})}\int_{\mathbb{D}}\int_{\mathbb{D}}(d_{z,w}^{\mathbb{D}}d_{w,z}^{\mathbb{D}})^{-2\Delta}dzdw
\end{align} 
with the inequality holding as $\Delta < 1/2$.
Note that $(d^\mathbb{D}_{z,w})^{-2\Delta}\leq |z-w|^{-2\Delta}+(1-|z|)^{-2\Delta}$ and  whenever $\Delta<\frac{1}{2}$ we have
\begin{align}
\int_{\mathbb{D}}|z-w|^{-4\Delta} dz \leq \int_{|z-w|\leq 2} |z-w|^{-4\Delta} dz = \int_{0}^{2\pi}\int_{0}^{2}r^{1-4\Delta} dr d\theta <\infty.
\end{align}  
Then, a straightforward calculation (or Lemma \ref{finitenessunitdisk} in the appendix) shows that
\begin{equation}
\int_{\mathbb{D}}\int_{\mathbb{D}}(d_{z,w}^{\mathbb{D}}d_{w,z}^{\mathbb{D}})^{-2\Delta} dz dw<\infty
\end{equation}
whenever $\Delta<\frac{1}{2}$, from which \eqref{convergent_integral} follows.
 
We now show that $\delta^{-2\Delta}U^\delta$ is a Cauchy sequence {in the Banach space $\mathbb{S} := \mathbb{S}(\alpha, D)$ of $\mathcal{H}^{-\alpha}$-valued square-integrable random elements that are measurable with respect to the $\sigma$-field generated by $X$}. Notice that, {for every $z$, the random variable} $\delta^{-2\Delta}U^\delta{ (z)}$ is {square-integrable}, as it is bounded. Let $\{f_i\}$ denote an eigenbasis of the Laplacian on $D$ {with Dirichlet boundary conditions}, and let $\lambda_1\leq \lambda_2\leq \cdots \to \infty$ be the corresponding sequence of eigenvalues. Then, writing $\delta^{-2\Delta}U^\delta =: \tilde{U}^\delta$ to simplify the notation,

\begin{align} \label{Caychy_seq}
&	\langle\|\tilde{U}^\delta- \tilde{U}^{\delta'}\|_{\mathcal{H}^{-\alpha}}^2\rangle_D \nonumber \\
	&= \sum\limits_{i}\frac{1}{\lambda_i^\alpha}\Big\langle\Big|\int_{D}\big(\tilde{U}^\delta(z)-\tilde{U}^{\delta'}(z)\big)\overline{f_i(z)} dz \Big|^2\Big\rangle_D,\\
	&= \sum\limits_{i}\frac{1}{\lambda_i^\alpha} \Big\langle\int_{D}\int_{D}(\tilde{U}^\delta(z)-\tilde{U}^{\delta'}(z))\overline{f_i(z)(\tilde{U}^\delta(w)-\tilde{U}^{\delta'}(w))} f_i(w) dz dw \Big\rangle_D \\
&	\leq \left(\int_{D}\int_{D}\big|\big\langle(\tilde{U}^\delta(z)-\tilde{U}^{\delta'}(z))\overline{(\tilde{U}^\delta(w)-\tilde{U}^{\delta'}(w))} \big\rangle_D \big| dz dw \right)\sum_{i}\frac{c^2}{\lambda_i^{\alpha-\frac{1}{2}}}. 
\end{align}
Note that the last inequality follows from the uniform bound on the norm $\|f_i\|_{L^\infty(D)}\leq c\lambda_i^{\frac{1}{4}}$, see Theorem~1 of \cite{grieser}.
The sum in the last expression is finite whenever $\alpha>\frac{3}{2}$. This is seen by using Weyl's law \cite{weyl}, which says that the number of eigenvalues $\lambda_i$,
that are less than $\ell$ is proportional to $\ell$ with an error of $o(\ell)$. From the $L^1(D\times D)$ convergence of
$(\delta\delta')^{-2\Delta}\langle U^\delta(z)\overline{U^{\delta'}(w)}\rangle$ and, by inspection of the form of the integrand in \eqref{Caychy_seq}, one deduces that
\be
\langle\|\tilde{U}^\delta- \tilde{U}^{\delta'}\|_{\mathcal{H}^{-\alpha}}^2\rangle_D \to 0
\ee
as $\delta,\delta' \to 0$. This implies that $\delta^{-2\Delta}U^\delta$ is a Cauchy sequence in $\mathbb{S}(\alpha, D)$.
From the completeness of {$\mathbb{S}$}, there is $U \in \mathbb{S}(\alpha, D)$ such that
$\langle\| \delta^{-2\Delta}U^\delta - U \|^2_{\mathcal{H}^{-\alpha}}\rangle_D\rightarrow 0$ as $\delta \rightarrow 0$.
\end{proof}

%\medskip

\begin{remark}
{\rm \begin{enumerate}

\item[]{\empty}

\item Theorem \ref{convergence} can be used in order to study the convergence of Poisson winding models, similarly to \cite[Section 5]{Spin_loops_berg_camia_lis}. In the next section, we will use such a result to extend the analysis of \cite{Spin_loops_berg_camia_lis} to the case of Poisson layering models. 

\item As shown in \cite{Spin_loops_berg_camia_lis}, the conclusion of Theorem \ref{convergence} does not hold for the winding model when $\Delta \geq 1/2$, so the condition
$\Delta \in (0,1/2)$ cannot be removed from the theorem without further assumptions.

\end{enumerate}
}
\end{remark}

\section{Layering fields} \label{s:plf}

{\it Throughout Section \ref{s:plf}, we let $D \subset \IC$ denote a domain conformally equivalent to the unit disk $\D$. We also recall the definition of the measurable spaces of loops $(M, \mathcal{M})$ and $(M_D, \mathcal{M}_D)$ given in Section \ref{sec:loops}, and use the convention that, if $\mu^*$ is a measure on $(M, \mathcal{M})$, then $\mu^*_D$ indicates the restriction of $\mu^*$ to $(M_D, \mathcal{M}_D)$.}

\medskip

The aim of the present section is to define a class of layering fields -- containing the objects discussed in Section \ref{ss:layering_field} as special cases -- to which the content of Theorem \ref{convergence} directly applies. In order not to disrupt the flow of our discussion, proofs are deferred to Section \ref{ss:proofponepoint}.

\subsection{Three measures on unrooted loops}\label{ss:curvemeasures}

%We now introduce {\teal the specific} 
%measures and kernels {\teal associated with} the families of loops of concern.
%
%
%% For measures on 
%% loops we adopt the definitions and notation of \cite{erik_federico}. In particular, we fix  a
%% bounded, simply connected domain  $D$ in $\mathbb R^2$.
%%Then we consider the  set $M=M_D$ of connected compact
%%subsets of $\mathbb R^2$ with nonempty interior contained in $D$, and
%%the $\sigma$-algebra $\mathcal M=\mathcal M_D$ on $M$
%%{\teal generated by} all subsets of $M$ of the form
%%\begin{eqnarray} \label{measurability}
%%\{ \gamma : \gamma \in M, \delta < \diam(\gamma) \leq R, \gamma \subseteq C\}
%%\end{eqnarray}
%%for all $\delta, R \in \mathbb R^+$ and every Borel subset $C$ of $\mathbb R^2$. 
%%Notice that the boundary of each subset $\gamma $ in $M$ determines a 
%%closed, unrooted 
%%loop contained in {\teal $D$} {\teal and viceversa}. {It follows that we can identify $M$ with the set of unrooted loops contained in $D$, and we will refer to a} measure $\mu^*=\mu^*_D$ 
%%on $(M_D, \mathcal M_D)$ as a  {\bf measure on unrooted loops}. 
%
%\smallskip

The three examples of a measure $\mu^*$ on the space $(M, \mathcal{M})$ (of unrooted loops) that are more relevant for the present paper are:
\begin{itemize}
\item[(I)] the {\bf Brownian loop measure}, in which $\mu^*=\mu^{loop}$, as defined
in \cite{Lawler-Werner} and summarized in Section \ref{ss:layering_field} (in this case, $\mu^{loop}$ has full conformal invariance); 
\item[(II)] the {\bf scale invariant disk distribution}, see \cite{FK,erik_federico},
in which $\mu^*=\mu^{disk}$ is defined as follows: consider the set
\be
\mathscr D=\{B: B = B_{y,r}, 
\text{ for some } y \in \mathbb R^2 \text{ and } r \in \mathbb R^+ \};
\ee
then $\mu^{disk}$ is the image under the embedding $B \mapsto \partial B$
(one can think of this embedding as a mapping from $\mathscr D$ to the space obtained from $M$ by identifying two loops if one can be obtained from the other by changing its orientation),
of the measure $ dy \times \frac{dr}{r^3}$ on
(the Borel $\sigma$-algebra of) $\mathbb R^2 \times \mathbb R^+$;
$\mu^{disk}$ is invariant under scaling and rigid motions, but not 
under the full conformal group of transformations;
\item[(III)] the {\bf massive Brownian loop measure}, see \cite[Section 2.3]{camia_notes}: here, a 
bounded mass function $m: \mathbb R^2 \rightarrow
\mathbb R^+, m \leq \overline m,$ is considered, and the corresponding measure $\mu^*=  \mu^m$ is such that
\be
d\mu^m(\gamma) = e^{-R_m(\gamma)} d\mu^{loop}(\gamma)
\ee
where $R_m(\gamma)=\int_0^{t_{\gamma}} m^2(\gamma(t))dt$  (where $\gamma$ is any element of the equivalence class corresponding to a given unrooted loop and $t_{\gamma}$ denotes the time length of $\gamma$).
Observe that there is another representation for the massive Brownian loop measure,
obtained by taking one realization of the BLS, 
 then a mean one exponential random variable  $T_{\gamma}$
for each loop, and, finally, removing the loops for which 
$\int_0^{t_{\gamma}} m^2(\gamma(t))dt > T_{\gamma}$
(see \cite[Proposition 2.9]{camia_notes}).
The massive Brownian loop measure $\mu^m$
is not scale invariant; on the other hand, the mass of the loops
with diameter bounded from below and which wind
around a point is finite, even when one does consider the full measure $\mu^m$ (without restricting it to $M_D$).

%: after renormalization, it defines a field over the whole of $\mathbb R^2$,
%and it allows for the explicit calculation of the two point function.

\end{itemize}

%{\teal \noindent{\bf Remark on notation}. In what follows we shall often use symbols decorated by a star, like $\mu^*, \alpha^*,...$ and so on. In all instances, the star $*$ has to be regarded as a placeholder for either 
%
%
%}

We will typically deal with some particular {sets} of loops, 
so we introduce some ad hoc notation, adapted  from \cite{camia2016conformal}.
Each {set of interest is} identified by the symbol $A$, {decorated with appropriate} indices; {such sets
 are all measurable since our reference $\sigma$-field contains sets of the type \eqref{measurability}.} {The (possibly decorated) symbol $\alpha^*$ is used below to build expressions useful for denoting the measure $\mu^*$ of a given set; we stress that, even if the measure $\mu^*$ appearing below is a generic measure on loops, in this paper the star $*$ has to be mainly regarded as a placeholder for either one of the three symbols, $loop$, $disk$ and $m$. Here is our notation:} 
\begin{itemize}
\item[-] for two real values $\delta$ and $R$ with $0\leq  \delta<R$, and for $z \in \R^2$, we let
	\begin{eqnarray} \label{alpha1}
	\alpha^*_{\delta , R}(z):= \mu^{*}(A_{\delta , R}(z)) := \mu^{*}(\gamma: z \in \overline \gamma, \delta \leq \diam(\gamma) \leq R) 
	\end{eqnarray}
	(note that, in the previous display and similarly for the forthcoming definitions, the last equality implicitly defines the set $A_{\delta , R}(z)$);
\item [-]
for a real number $\delta \geq 0 $, $z\in D$ and a set $V \subseteq D$, we let
\begin{eqnarray} \label{alpha2}
\alpha^*_{\delta , V}(z):= \mu^{*}(A_{\delta , V}(z) )&:=&  \mu^{*}(\gamma: z \in \overline \gamma, \delta \leq \diam(\gamma), \gamma \subseteq V)\\&=&  \mu_D^{*}(\gamma: z \in \overline \gamma, \delta \leq \diam(\gamma), \gamma \subseteq V)\notag
\end{eqnarray}
(it will be always be clear from the context whether an expression of the type $\alpha^*_{\delta , Q}$ refers to a real number $Q>0$ or to a set $Q\subseteq D$);

\item[-] for real numbers $0 \leq \delta' < \delta $, points $z,w\in D$ and a set $V \subseteq D$, we let
\begin{eqnarray} \label{alpha3}
\alpha^*_{\delta' , \delta, V}(z,w)&:=&\mu^{*}(A_{\delta' , \delta, V}(z,w))
= \mu_D^{*}(A_{\delta' , \delta, V}(z,w)) \\
&:=&\mu^{*}(\gamma: \delta' \leq \diam(\gamma) \leq \delta, z, w \in \overline \gamma, \gamma \subseteq V);
\nonumber
\end{eqnarray}
\item[-] for two sets $U, V$ such that $U \subseteq V \subseteq D$ and for $z\in D$ we let 
\begin{eqnarray} \label{alpha4}
\alpha^*_{U,V}(z):= \mu^{*}(A_{U,V}(z)) =  \mu_D^{*}(A_{U,V}(z)):= \mu^{*}(\gamma: z \in \overline \gamma, \gamma \not \subseteq U, \gamma \subseteq V);
\end{eqnarray}
\item[-]  for two points $z, t \in D$ and $\delta>0$
\begin{eqnarray} \label{alpha6}
\alpha^*_{D}(z,t) &:=&\mu^{*}(A_{D}(z,t)):= \mu^{*}(\gamma: z,t \in \overline \gamma, \gamma \subseteq D),\\
\alpha^*_{\delta, D}(z,t) &:=&\mu^{*}(A_{\delta, D}(z,t)):= \mu^{*}(\gamma: z,t \in \overline \gamma, \delta \leq \diam(\gamma),  \gamma \subseteq D),\label{e:floris}
\end{eqnarray}
and
\begin{eqnarray} \label{alpha7}
	\alpha^*_{\delta, D}(z|t):=\mu^{*}(A_{\delta, D}(z|t)) := \mu^{*}(\gamma: z \in \overline \gamma, t \not \in  \overline \gamma, \delta \leq \diam(\gamma), \gamma \subseteq D).
\end{eqnarray}
\item[-] for a set $V$  we write
\begin{eqnarray} \label{alpha9}
\alpha^*_{\neg V}(z) := \mu^{*}(A_{\neg V}(z))=
 \mu^{*}(\gamma: z \in \overline \gamma, \gamma \not \subseteq  V)
\end{eqnarray}
 and
\begin{eqnarray} \label{alpha10}
\alpha^*_{\neg V}(z|w) =  \mu^{*}(A_{\neg V}(z|w) ):=
 \mu^{*}(\gamma: z \in \overline \gamma, w \notin \overline \gamma, \gamma \not \subseteq  V)
\end{eqnarray}
so that, for instance,
$\alpha^*_{\neg B_{{\bf 0},1}}({\bf 0}|{\bf 1}) $
is the measure of the set of loops not restricted to the disk of radius $1$ centered at the origin,
containing the origin but not the point $(0,1)$.
\end{itemize}

\medskip

\begin{remark}{\rm \begin{enumerate}

\item[]{\empty}

\item As already observed, in the three cases $\mu^{*} = \mu^{loop}, \mu^{disk}, \mu^{m}$, using the notation \eqref{alpha2}, we have
\begin{eqnarray} \label{alpha11}
\alpha_{0,D}^*(z)= \mu_D^{*}(\gamma: z \in \overline \gamma,  \gamma \subseteq D) = \infty,  \text{ for all } z \in D,
\end{eqnarray}
that is: for the Brownian loop measure, the disk distribution and the massive Brownian loop measure, $\mu^*$ charges with infinite mass the set of loops in $D$ covering any fixed point $z\in D$. In general, the behaviour of the function $\alpha^{*}_{\delta,R}(z)$ for small $\delta$'s is needed to identify the conformal dimensions of the fields of interest.
%and renormalize the exponential operators, as we will do starting from the next subsection.

\item To further illustrate the use of the notation, let us recall that the following equalities are shown in \cite[Lemma A.1]{camia2016conformal}:
\begin{eqnarray} \label{alpha5}
	\alpha^{loop}_{B_{z,\delta},B_{z,R}}(z) = \alpha^{loop}_{\delta,R}(z) = \frac{1}{5} \log \frac{R}{\delta}.
\end{eqnarray}
Both equalities are used in the present paper. A straightforward adaptation of \cite[Lemma A.1]{camia2016conformal} yields the {next} equality (see also \cite[Section 3.1]{FK}):
\begin{eqnarray} \label{alpha8}
	\alpha^{disk}_{B_{z,\delta},B_{z,R}}(z) = \alpha^{disk}_{\delta,R}(z) = \pi \log \frac{R}{\delta}.
\end{eqnarray}

\end{enumerate}
}
\end{remark}

\medskip

\subsection{Some useful one- and two-point computations}\label{ss:poissononepoint}

In this section we discuss the divergence of the measures we are interested in. In particular, we show that the divergence of the massive Brownian loop measure
$\mu^{m}$ is the same as that of $\mu^{loop}$. This is a consequence of the fact that the divergence of both $\mu^{m}$ and $\mu^{loop}$ is due to small loops,
whose measure is not much affected by a bounded mass function $m$.
Some of the results in this section can be expressed in terms of the measure $\hat\mu$ defined by 
\begin{equation} \label{e:muhat}
d\hat\mu(\gamma) = \big( 1 - e^{-R_m(\gamma)} \big) d\mu^{loop}(\gamma).
\end{equation}
Note that, if $A$ is such that $\mu^{loop}(A)<\infty$, then $\hat\mu(A)=\mu^{loop}(A)-\mu^m(A)$.
In order to make our discussion more streamlined, the proofs of the results of this section are gathered together in Section \ref{ss:proofponepoint}.

\medskip

%We start by observing that, for the massive case, the quantity $\alpha^{m}_{\delta,R}(z)$ can be estimated as follows.

\begin{lemma}\label{MassiveOnePointFunctionEstimate}
Given a bounded mass function $m \leq \overline m \in \mathbb R$, we have that 
for $z \in D$, $0<\delta<R$, and some $c=c(z, \overline m, D)>0$,
\begin{eqnarray} 
\alpha^{loop}_{\delta, R}(z) \geq \alpha^{m}_{\delta, R}(z)
& \geq & \alpha^{loop}_{\delta, R}(z) - 2R\Big( \frac{\overline m^2}{5} \log{2} + 1 \Big) \label{MassiveAlphaEstimate1} \\
\lim_{\delta \to 0} \Big( \alpha^{loop}_{\delta,R}(z) -\alpha^{m}_{\delta,R}(z) \Big) & = & \hat\mu(\gamma: z \in \bar\gamma, \diam(\gamma) \leq R)
\nonumber \\
& \leq & 2R \Big( \frac{\overline m^2}{5} \log{2} + 1 \Big) < \infty
\label{MassiveAlphaEstimate2} \\
\lim_{\delta \to 0} \Big( \alpha^{loop}_{\delta,D}(z) -\alpha^{m}_{\delta,D}(z) \Big) & = & \hat\mu(\gamma: z \in \bar\gamma, \gamma \subset D) \leq c < \infty. \label{MassiveAlphaEstimate3}
\end{eqnarray}
\end{lemma}

For later purposes, it is useful to introduce the following notation:
$\hat\alpha_{0,R}(z) := \hat\mu(\gamma: z \in \bar\gamma, \diam(\gamma) \leq R)$
and $\hat\alpha_{0,D}(z) := \hat\mu(\gamma: z \in \bar\gamma, \gamma \subset D)$, where the measure $\hat\mu$ has been defined in \eqref{e:muhat}.

We also introduce the concept of \emph{thinness} of a loop soup, taken from \cite{NW11} (see Lemma 2 there). We say that $\mu^*$ is \emph{thin}
if there exists $0<R_0<\infty$ such that $\mu^*(\gamma: \gamma \cap \D \neq \emptyset, \diam(\gamma) \geq R_0) < \infty$. We say that a loop soup with
intensity measure $\lambda\mu^*$ is \emph{thin} if $\mu^*$ is thin.
An equivalent formulation, the one we will use in this paper, is that $\lim_{R \to \infty} \mu^*(\gamma: \gamma \cap \D \neq \emptyset, \diam(\gamma) \geq R) = 0$.
The thinness of the BLS is proved in \cite{NW11}. The thinness of the disk model and the massive BLSs can be proved easily, as shown below.
\begin{lemma}\label{lemma:disk-thin}
The measures $\mu^{disk}$ and $\mu^m$ are thin.
\end{lemma}
\begin{proof}
For $\mu^{disk}$, using elementary geometric and trigonometric considerations, we have that
\begin{eqnarray}
&& \lim_{R \to \infty} \mu^{disk}_{\mathbb R^2} \Big( {\bf 0} \in \bar\gamma, {\bf 1} \notin \bar\gamma, \diam(\gamma)>R \Big) \\
&& \quad = \lim_{R \to \infty} \int_{R}^{\infty} \Big( 2 r^2 \sin^{-1}(1/2r) + \sqrt{r^2-1/4} \Big) \frac{dr}{r^3} = 0.
\end{eqnarray}
By translation and rotation invariance, this implies that the disk model is thin.
For the massive BLS, we trivially have that
\be \notag
\lim_{R \to \infty} \mu^m(\gamma: \gamma \cap \D \neq \emptyset, \diam(\gamma) \geq R) \leq
\lim_{R \to \infty} \mu^{loop}(\gamma: \gamma \cap \D \neq \emptyset, \diam(\gamma) \geq R) = 0,
\ee
which concludes the proof.
\end{proof}

Next, we focus on two-point estimates. We will show that in the case of a thin loop soup, if $z_1$ and $z_2$ are close to each other, then one can relate
the divergence of $\alpha^*_D(z_1, z_2)$, as $z_2 \to z_1$, to that of $\alpha^*_{\delta,R}(z_1)$ as $\delta \to 0$, as demonstrated by the next theorem.

\begin{theorem}\label{t:2p1}
If $*=loop$ or $disk$, then for each $z_1 \in D$ the following limit exists and is finite, and it defines a continuous function of $z_1 \in D$:
\begin{eqnarray} \label{ThScaleInvTwoPoint}
\lim\limits_{z_2 \to z_1} ( \alpha^*_D(z_1, z_2) - \alpha^*_{|z_1-z_2|,d_{z_1}}(z_1) )
=:\Psi^*(z_1,D).
\end{eqnarray}
\end{theorem}

\begin{lemma}\label{l:brexit}
For $z \in D$, let $\Psi^m(z,D) := \Psi^{loop}(z,D) - \hat\mu(\gamma: z \in \bar\gamma, \gamma \subset D)$. Then, for each $z_1 \in D$,
the following limit exists:
\begin{eqnarray} \label{ThScaleInvTwoPointMassive}
\lim\limits_{z_2 \to z_1} ( \alpha^m_D(z_1, z_2) - \alpha^{loop}_{|z_1-z_2|,d_{z_1}}(z_1) )
=:\Psi^m(z_1,D).
\end{eqnarray}
Moreover, $\Psi^m(z,D)$ is continuous in $z \in D$.
\end{lemma}

We also record the following continuity property, holding in the special case of the (massive) Brownian loop and disk measures; it will be exploited in Section \ref{ss:GaussianLayering}.

\begin{lemma}\label{l:cont}
Let the framework of the present section prevail, and define in particular $\Psi^{loop}, \Psi^{disk}$ and $\Psi^m$ according to \eqref{ThScaleInvTwoPoint} and \eqref{ThScaleInvTwoPointMassive}. For $* = loop, m$ and $z_1,z_2 \in D$ such that $z_1 \neq z_2$, let 
\begin{eqnarray}
g^{*}(z_1,z_2)&:=& \alpha_D^*(z_1,z_2)- \frac15 \log^+{\frac{1}{|z_1-z_2|}}, \\
g^{*}(z_1,z_1) &:=&  \frac{1}{5} \log d_{z_1} +\Psi^{*}(z_1, D), \\
g^{disk}(z_1,z_2) &:=& \alpha_D^{disk}(z_1,z_2)- \pi \log^+{\frac{1}{|z_1-z_2|}}, \\
g^{disk}(z_1,z_1) &:=&  { \pi} \log d_{z_1} +\Psi^{disk}(z_1, D).
 \end{eqnarray}
Then, the three mappings $g^{loop},\, g^{disk},\,  g^m : D\times D \rightarrow \R$ are continuous.
\end{lemma}

We now focus on kernels defined on loop spaces.

\subsection{Kernels on unrooted loops and Poisson layering models }\label{ss:PoissonLayering}
We will specify the exact form of the kernels $ {\bf h}$ appearing in \eqref{e:hkernels} that is needed in order to obtain a collection of general Poisson layering fields.
We shall consider exclusively kernels yielding a mechanism for assigning a random sign to each loop; although centered random variables more general than a random
symmetric sign could be used, we stick to this case for simplicity and in view of \cite{camia2016conformal,FK}. This completes the definition of the truncated fields
and brings us back to the abstract setting of Section \ref{construction}. We consider the following elements:
\begin{itemize}
	\item[(i)]  the set $A = M_D \times {\mathbb R}$;
	the elements of $A$ are generally denoted by $x = (\gamma, \epsilon)$;
	\item[(ii)]  $\mathscr{A} = \mathcal{ M}_D \otimes \mathscr B$, that is, $\mathscr{A}$ is the product
	$\sigma$-algebra of $\mathcal{ M}_D$ and  $ \mathscr B$,
	where 
	$\mathscr B$ is the Borel $\sigma$-field of $\mathbb R$;
	\item[(iii)] the measure $\nu(d\gamma, d\epsilon)=\mu_D^{*}(d\gamma) \times \frac12 (\delta_1+\delta_{-1})(d\epsilon)$, where $\delta_a$ denotes the Dirac mass at $a$ (see also \cite{camia2016conformal}), {and $\mu_D^*$ is a measure on unrooted loops contained in $D$}, { verifying $\mu_D^{*}(\gamma\subset D: z \in \overline \gamma) = \infty$};
	\item[(iv)] a {\it locally exploding} collection of kernels $ {\bf h}$ of the type  \eqref{e:hkernels}, where
	 $h_z^{\delta}(x) = \epsilon {\bf 1}_{A_{\delta, D}(z)}(\gamma)$, with $x = (\gamma,\epsilon) \in A$
		and $z \in D$, and the definition of the set $A_{\delta, D}(z)$ is given in \eqref{alpha2}. We observe in particular that, for every $z\in D$ and for $\nu$-almost every $x = (\gamma, \epsilon) \in A$, 
		\begin{equation}\label{e:accadelta}
		h^\delta_z(x) \longrightarrow h_z(x) := \epsilon {\bf 1}_{A_{0, D}(z)}(\gamma), \quad \delta\to 0. 
		\end{equation}
		
		\end{itemize}
\vspace{-0.5cm}
For notational convenience, we will sometimes write
\begin{equation}\label{e:rsign}
S(d\epsilon) := \frac12 (\delta_1+\delta_{-1})(d\epsilon)
\end{equation}
to indicate the law of a symmetric random sign.
\begin{definition}[Poisson layering model]\label{d:plm} {\rm A {\it Poisson layering model} over $D$ with parameters $\lambda>0$ and $\beta \in [0,2\pi)$
is a small parameter family $\mathbb{V} = \mathbb{V}_{\lambda,\beta}$ of $\nu$ controlled complex exponential random fields, as in Definition \ref{def:spf} with $\zeta$ replaced by $\beta$ and with $X = N_{\lambda}$ a Poisson measure with control $\lambda\nu$. More explicitly,
\be
V^{\delta}(z)= V^{\delta}_{\lambda,\beta}(z) = \exp(i \beta N_{\lambda}(h_z^\delta)), \quad z \in D, \label{e:vlb}
\ee
where the measure $\nu$ and the kernel ${\bf h}$ are as in items (iii) and (iv) above.}
\end{definition}

\begin{remark}\label{r:muffat} {\rm 
\begin{enumerate}

\item[]{\empty}

\item For Poisson layering models, the choice of replacing the letter $\zeta$ with $\beta$ in~\eqref{e:spf} is made in order to facilitate the connection with \cite{camia2016conformal}.

\item A more precise terminology than the one used in Definition \ref{d:plm} would make reference to the measure $\mu_D^*$; this would lead to expressions such as  ``$\mu_D^*$-  (or $\mu_D^{loop}$-, $\mu_D^{disk}$-, $\mu_D^m$-) Poisson layering models and fields,'' but we prefer the terminology introduced above for the sake of simplicity. In the next section, we will use the same kernels as here, but take $X=G$, where $G$ is a Gaussian measure with control $\mu_D^*$; we will call the resulting fields \emph{Gaussian} layering fields.
\end{enumerate}}
\end{remark}

\subsection{Convergence of renormalized Poisson layering fields in negative Sobolev spaces}\label{ss:conv_layering_fields}

We can finally present the main result of the section, which states that the Poisson layering models $V_{\lambda,\beta}^{\delta}$ with intensity measure 
$\lambda\nu = \lambda  \mu^*_D\otimes {S} $, where we have used \eqref{e:rsign}, and $*=loop, m$ or $disk$ and $\beta \in [0,2\pi)$,
when renormalized with an appropriate scaling factor, converge in second mean in any Sobolev space ${\mathcal H}^{-\alpha}$ with $\alpha>3/2$.
The proof, which is a direct application of Theorem \ref{convergence}, requires a certain amount of control on the functions $\alpha^*_{\delta, R}(z)$,
in order to verify the condition of Theorem \ref{convergence}. This is where the restriction of the general framework of Definition \ref{d:plm} to the measures
$\mu^{*}$ with $*=loop, m, disk$ plays a crucial role.

\begin{theorem} \label{ExistenceLayeringField}
Fix $\lambda>0$ and $\beta \in [0,2\pi)$ and let $*=loop, m$ or $disk$.
Let $D$ be the unit disk $\D$ when $*=disk$ or any domain conformally equivalent to $\D$ when $*=loop,m$.
Let $\mathbb{V}_{\lambda, \beta}$ be a Poisson layering model over $D$,
with parameters $\lambda, \beta$ and intensity measure $\lambda\nu = \lambda  \mu^*_D\otimes {S}$, as in Definition \ref{d:plm}.
Let $\Delta^{loop}_{\lambda, \beta} = \Delta^{m}_{\lambda, \beta} = \frac{\lambda}{10}(1-\cos\beta)$ and
$\Delta^{disk}_{\lambda, \beta} = \lambda \frac{\pi}{2}(1-\cos\beta)$. If $\Delta^*_{\lambda, \beta} <1/2$  and $\alpha>3/2$, as $\delta\to 0$ the renormalized random field on $D$
\begin{equation}\label{e:crrf}
z \mapsto \delta^{-2\Delta^*_{\lambda, \beta} } V_{ \lambda, \beta}^\delta(z) =\delta^{-2\Delta^*_{\lambda, \beta} }  \exp\{ i\beta N_\lambda(h_z^\delta)\}, \quad z\in D,
\end{equation}
(regarded as a random generalized function) converges in second mean in the Sobolev space $\mathcal{H}^{-\alpha}$, in the sense of \eqref{2.4bis}, to a random generalized function $V^*_{\lambda, \beta}$, which we call a \emph{Poisson layering field}.
\end{theorem}

\begin{proof} %[Proof of Theorem \ref{ExistenceLayeringField}]
If $D$ is a bounded simply connected domain with a $C^1$ boundary, the result corresponds to Theorem \ref{ExistenceLayeringFieldBoundedDomains}.
For more general domains, it follows from Theorem \ref{thm:conformal_covariance}.
\end{proof}

\begin{remark}{\rm The quantity $\Delta^{*}_{\lambda,\beta}$ is called the \emph{conformal dimension} of the field $V^*_{\lambda, \beta}$.
}
\end{remark}

\subsection{Gaussian layering fields} \label{ss:GaussianLayering}

We consider the elements $A = M_D\times \R$, $\mathscr{A} = \mathcal{M}_D\otimes \mathscr{B}$, $\nu = \mu^*_D\otimes S$, and ${\bf h}$ introduced at item (i)--(iv) of Section \ref{ss:PoissonLayering} (recall notation \eqref{e:rsign}), and also consider a Gaussian measure $G$ on $(A, \mathscr{A})$ with control $\nu$, as defined in \eqref{Gauss}. 

\begin{remark} {\rm In the discussion to follow, we will be only interested in the properties of the Gaussian measure
\begin{equation}\label{e:gmrk}
B \mapsto G(f_B),
\end{equation}
where $B\in \mathcal{M}_D$ is such that $\mu^*_D(B)<\infty$ and, for $x = (\gamma, \epsilon)$, $f_B(x) := \epsilon {\bf 1}_B(\gamma)$. A direct covariance computation shows that the Gaussian measure defined in \eqref{e:gmrk} has the same distribution as
\be
B\mapsto G_0(B), \quad \mu^*_D(B)<\infty,
\ee
where $G_0$ is a Gaussian measure on $(M_D, \mathcal{M}_D)$ with control $\mu^*_D$, in such a way that the use of the measure $\nu$ on the product space $M_D\times \R$ is somehow redundant. Our choice of using the measure $\nu$ for expressing the law of the underlying Gaussian integrators is motivated by computational convenience --- see e.g. Section \ref{s:convtoGMC}.}
\end{remark}

%consider the following elements:
%\begin{itemize}
%	\item[(i)]  the set $M_D$ whose elements of are denoted by $\gamma$;
%	\item[(ii)] the $\sigma$-algebra $\mathcal{ M}_D$;
%	\item[(iii)] the measure $\nu(d\gamma)=\mu_D^{*}(d\gamma)$, where $\mu_D^*$ is a measure on unrooted loops contained in $D$, verifying $\mu_D^{*}(\gamma\subset D: z \in \overline \gamma) = \infty$;
%	\item[(iv)] a {\it locally exploding} collection of kernels $ {\bf h}$ of the type  \eqref{e:hkernels}, where
%	 $h_z^{\delta}(x) = {\bf 1}_{A_{\delta, D}(z)}(\gamma)$, with $z \in D$, and the definition of the set $A_{\delta, D}(z)$ is given in \eqref{alpha2}.
%\end{itemize}

With this notation, according to the discussion following \eqref{Gauss} one has that, for each $z \in D$, $G(h_z^\delta)$ is a centered Gaussian random variable
with variance $\alpha^{*}_{\delta,D}(z)$, and $\{ G(h_z^\delta) \}_{z \in D}$ is a Gaussian field with covariance kernel $K_{\delta}^{*}(z,w) = \alpha^{*}_{\delta,D}(z,w)$.
We also introduce the generalized Gaussian field ${\mathcal G}^{*} = \{ G(h^{0}_{z}) \}_{z \in D}$ with covariance kernel $K^{*}(z,w) = \alpha^{*}_{0,D}(z,w)$. The field
$G^{*}$ can be realized as a random element of the Sobolev space $\mathcal{H}^{-\alpha}(D)$, for some $\alpha>0$, 
%${\mathcal S}^{\prime}({\mathbb R}^2)$ of real tempered distributions on ${\mathbb R}^2$
whose law is given by the centered Gaussian distribution determined by the covariance formula
\be
\text{Cov}({\mathcal G}^{*}(f),{\mathcal G}^{*}(g)) = \int_{D} \int_{D} f(z) \alpha^{*}_{0,D}(z,w) g(w) dz dw,
\ee
for all test functions $f,g$ in the space $\mathcal{H}^{\alpha}(D)$ dual to $\mathcal{H}^{-\alpha}(D)$.
%Schwartz space ${\mathcal S}({\mathbb R}^2)$.
We can also interpret the Gaussian field $\{ G(h_z^\delta) \}_{z \in D}$ as a random element of $\mathcal{H}^{-\alpha}(D)$ by defining
${\mathcal G}^{\delta}(f) := \int_D G(h^{\delta}_{z}) f(z) dz$ for each $f \in \mathcal{H}^{\alpha}(D)$.
We can then observe that, using Fubini's theorem and the fact that $\lim_{\delta \to 0} \alpha^{*}_{\delta,D}(z,w) = \alpha^{*}_{0,D}(z,w)$
for every $z,w \in D$,
\begin{eqnarray}
\text{Cov}({\mathcal G}^{\delta}(f),{\mathcal G}^{\delta}(g)) & = & \E \Big( \int_{D} \int_{D} G(h^{\delta}(z)) G(h^{\delta}(w)) f(z) g(z) dz dw \Big) \nonumber \\
& = & \int_{D} \int_{D} f(z) \alpha^{*}_{\delta,D}(z,w) g(w) dz dw \nonumber \\
& \stackrel{\delta \to 0}{\longrightarrow} & %\int_{D} \int_{D} f(z) \alpha^{*}_{0,D}(z,w) g(w) dz dw = 
\text{Cov}({\mathcal G}^{*}(f),{\mathcal G}^{*}(g)). \label{eq:cov-conv}
\end{eqnarray}
By an argument analogous to that in the proof of Theorem \ref{convergence}, this implies that the field ${\mathcal G}^{\delta}$ converges to ${\mathcal G}^{*}$,
as $\delta \to 0$, in second mean in the Sobolev space $\mathcal{H}^{-\alpha}(D)$ for every $\alpha>3/2$.

Next, we introduce the main object of interest in the present subsection.

\begin{definition}[Gaussian layering model]\label{d:glm} {\rm A {\it Gaussian layering model} over $D$ with parameter $\xi \in \R\backslash \{0\}$ is a small parameter family $\mathbb{W}_{\xi} = \mathbb{W}_{\xi,D}$ of $\nu$ controlled complex exponential random fields, as in Definition \ref{def:spf} with $\zeta$ replaced by $\xi$ and with $X =G$ a Gaussian measure with control $\nu$. More explicitly,
\be
W^{\delta}_{\xi}(z) = W^{\delta}_{\xi,D}(z) = \exp(i \xi G(h_z^\delta)), \quad z \in D, \label{e:vxi}
\ee
where the measure $\nu$ and the kernel ${\bf h}$ are as in items (iii) and (iv) above.}
\end{definition}

The definition above is interesting mainly because, as $\delta \to 0$, $W_{\xi}^{\delta}$ converges to a (deterministically) ``tilted''
imaginary Gaussian multiplicative chaos \cite{ComplexGMC,GMC_imaginary}, as defined in the introduction just before Theorem \ref{MainTheorem}.
%
%in the sense of \cite{ComplexGMC,GMC_imaginary}, as shown by the next theorem. Assuming that $X : D \to {\mathbb C}$ is a Gaussian process with covariance kernel
%\be \label{eq:covkernel}
%K(z,w) := \E (X(z)X(w)) = \log^{+}\frac{1}{|z-w|} + g(z,w),
%\ee
%where $g$ a bounded continuous function over $D \times D$, in this article, we say that a random generalized function $M^{\xi}$ is an
%\emph{imaginary Gaussian multiplicative chaos} in $D \subset {\mathbb C}$ with parameter $\xi>0$, and denoted it
%\be \label{eq:gmc}
%M_{\xi} = e^{i \xi X},
%\ee
%if there exists a sequence $\{ X_{\delta} \}_{\delta}$ of Gaussian fields converging to $X$ such that the measures
%\be \nonumber
%M_{\xi}^{\delta} (z) = e^{i \xi X_{\delta}(z) + \frac{\xi^2}{2} \E (X^2_{\delta}(z))},
%\ee
%are well defined and converge, as $\delta \to 0$, to $M_{\xi}$.
%
\begin{theorem} \label{ExistenceGaussianLayeringField}
Fix $\xi>0$ and let $*=loop, m$ or $disk$.
Let $D$ be the unit disk $\D$ when $*=disk$ or any domain conformally equivalent to $\D$ when $*=loop,m$.
Let $\mathbb{W}_{\xi,D}$ be a Gaussian layering model over $D$, with parameter $\xi$ and intensity measure $\mu^*_D$, as in Definition \ref{d:glm}.
Let $\Delta_{\xi}^{loop} = \Delta_{\xi}^{m} = \xi^2/20$ and $\Delta_{\xi}^{disk} = \pi\xi^2/4$.
If $\Delta_{\xi}^*<1/2$  and $\alpha>3/2$, as $\delta\to 0$, the random field on $D$
\be
z\mapsto\delta^{-2\Delta_\xi^*} W_{\xi,D}^\delta(z) :=\delta^{-2\Delta_\xi^*}  e^{i \xi G(h_z^\delta)}
\ee
(regarded as a random distribution) converges in square mean in the negative Sobolev space $\mathcal{H}^{-\alpha}$, in the sense of \eqref{2.4bis},
to a random distribution $W^{*}_{\xi,D}$, which we call a \emph{Gaussian layering field}. Moreover,
\be \label{eq:rdderivative}
\frac{dW^{*}_{\xi,D}}{dM_{\xi,D}^{*}}(z) = e^{-\frac{\xi^2}{2}\Theta^{*}_{D}(z)},
\ee
where $\Theta^{loop}_{D}(z) = \frac{1}{5} \log d_{z} + \alpha^{loop}_{d_{z},D}(z)$,
$\Theta^{m}_{D}(z) = \Theta^{loop}_{D}(z) - \hat\alpha_{0,D}(z)$,
$\Theta^{disk}_{D}(z) = \pi \log d_{z} + \alpha^{disk}_{d_{z},D}(z)$, 
and $M^{*}_{\xi,D}$ is an imaginary Gaussian multiplicative chaos in $D$ with parameter $\xi$ and covariance kernel $K^{*}_D(z,w) = \alpha^{*}_{D}(z,w)$.
\end{theorem}

\begin{proof}
If $D$ is a bounded simply connected domain with a $C^1$ boundary, the result corresponds to Theorem \ref{ExistenceGaussianLayeringFieldBoundedDomains}.
For more general domains, it follows from Theorem \ref{thm:conformal_covariance}.
\end{proof}

\section{Chaos expansions} \label{sec:chaos}

In this section we show that the action of the fields constructed in Sections \ref{ss:conv_layering_fields} and \ref{ss:GaussianLayering} on smooth test functions admits a {\it Wiener-It\^o chaos expansion}.
%
%One of the main purposes of the present paper is to show that, if $\lambda \to \infty$ and $\beta \to 0$ in such a way that $\lambda \beta^2 \to \xi^2$, which is to say
%as $\Delta^*_{\lambda, \beta} \to\Delta^*_\xi $, then the Poisson Random Fields converge to the Gaussian ones; we are able to do this, however, only
%by assuming that $2\Delta_{\lambda, \beta}^* , 2\Delta^*_\xi < \frac{1}{2}$. {As discussed below, we conjecture that such a restriction is actually an artefact of the techniques adopted in the present section, and that the convergence results evoked above should hold for the full range $2\Delta_{\lambda, \beta}, 2\Delta^*_\xi < 1$},.
%
To accomplish this task, we first briefly recall the theory of Wiener-It\^o chaos expansions of general Poisson and Gaussian measures, and then present the main results of this section.
A detailed discussion of Wiener chaos can be found in \cite{Last} (Poisson case), \cite{Gio_ivan} (Gaussian case) and \cite{Gio_murad_taqqu} (both Gaussian and Poisson). 

In the next section, we will show that each term of the Wiener-It\^o chaos expansion of the Poisson layering field converges to the corresponding term of the Wiener-It\^o chaos
expansion of the Gaussian layering field, which implies convergence of the Poisson layering field to the Gaussian layering field in the sense of finite dimensional distributions.

\subsection{Chaos expansion for the Poisson layering field}\label{ss:chaospoisson}
%General results on chaotic expansions: Poisson case}\label{ss:gcp}

{Until further notice, we will now consider the general situation of a Poisson random measure $N_\lambda$, defined on a measurable space $(A, \mathscr A)$, with control $\lambda \nu$, where $\lambda>0$ and $\nu$ is $\sigma$-finite and non-atomic --- see Section \ref{ss:concepts}. We also denote by $G$ a Gaussian measure on $(A, \mathscr{A})$ with control $\nu$.

One of the pivotal points of our analysis (allowing us to effectively connect Poisson and Gaussian structures) will be the following elementary consequence of the multivariate Central Limit Theorem (CLT): as $\lambda\to \infty$, the standardized compensated random measure $\lambda^{- 1/2}\widehat{N}_\lambda$ converges to $G$ in the following sense: for every $d\geq 1$ and every $B_1,...,B_d\in \mathscr{A}$ such that $\nu(B_i)<\infty$ ($i=1,...,d$),
\begin{equation}\label{e:clt}
\left(\lambda^{- 1/2}\widehat{N}_\lambda(B_1),..., \lambda^{- 1/2}\widehat{N}_\lambda(B_d)\right) \Longrightarrow (G(B_1), ..., G(B_d)),
\end{equation}
where $\Longrightarrow$ indicates convergence in law. }

The classical Wiener-It\^o representation property of $N_\lambda$ (see e.g. \cite{Last, Last_Penrose_lectures, Gio_murad_taqqu}) states that, for every random variable $Y\in L^2(\sigma(N_\lambda))$, there exists a unique sequence of ($\nu^q$-almost everywhere) symmetric kernels $\{f_q : q=1,2,...\}$, such that $f_q\in L^2(A^q, \mathscr{A}^q,\nu^q) = L^2(\nu^q)$ and
\begin{equation}\label{e:chaos}
Y = \E[Y]+\sum_{q=1}^\infty I^{N_\lambda}_q(f_q),
\end{equation}
where the series converges in $L^2(\mathbb{P})$ and 
\be
I_q^{N_\lambda}(f_q) := \int_{A^q} f(x_1,...,x_q){\bf 1}_{\{x_k\neq x_l, \, \forall k\neq l\}}  \, \widehat{N}_\lambda(dx_1)\cdots  \widehat{N}_\lambda(dx_q)
\ee
stands for the multiple Wiener-It\^o integral of order $q$ of $f_q$ with respect to $ \widehat{N}_\lambda$.
For every $f\in L^2(\nu)$, we write $ I_1^{N_\lambda} = \int_A f(x)  \widehat{N}_\lambda(dx) := \widehat{N}_\lambda(f)$. One should observe that chaotic decompositions are usually stated and proved in the case of real-valued random variables $Y$; however, the result extends immediately to the case of a square-integrable random variable with the form $Y_{\mathbb{C}} = Y_R+ i Y_I$ (with $i = \sqrt{-1}$), by separately applying \eqref{e:chaos} to $Y = Y_R$ and $Y = Y_I$. Note that if $f = f_R+ i f_I$ is a complex-valued element of  $L^2(\nu^q)$, then the integral $I_q(f)$ is simply defined as $I_q(f) = I_q(f_R) + i I_q(f_I)$. All general results discussed below apply both to the real and complex case, so that we will only distinguish the two instances when there is some risk of confusion.
One crucial relation we are going to exploit in the sequel is the following {\it isometric property} of multiple stochastic integrals (see e.g. \cite[Proposition 5.5.3]{Gio_murad_taqqu}): if $f\in L^2(\nu^q)$ and $g\in L^2(\nu^p)$ are two symmetric kernels, with $p,q\geq 1$, then
\be\label{e:isom}
\E\left[I^{N_\lambda}_q(f) \overline{ I^{N_\lambda}_p(g)}\right] =  \delta_{pq} \, q! \lambda^q \int_{A^q} f(x_1,...x_q) \overline{ g(x_1,...,x_q)}\, \nu(dx_1)\cdots\nu(dx_q), 
\ee
where $\delta_{pq}$ is Kronecker's symbol.

% *** here I removed the distinction between centered and not centered Poisson process ****
%It is also convenient to use the notation
%$$
%I_1(f) = \int_A f(x)  {N}_\lambda(dx), \,\, \mbox{and}\,\, N_\lambda(f) = \int_A f(x) \, N_\lambda(dx);
%$$
%note that $I_1(f)$ is well defined whenever $f\in L^2(\mu)$, whereas in order for $N_\lambda(f)$ to be well defined we require that $f\in L^1(\mu)$.

We will now explain how to explicitly compute the decomposition \eqref{e:chaos} for a given random variable $Y$. For every $x\in A$, the {\it add-one cost operator} at $x$, written $D_x$, is defined as follows: for every $\sigma(N_\lambda)$-measurable random variable $Y = Y(N_\lambda)$
\be
D_xY := Y(N_\lambda + \delta_x ) - Y(N_\lambda),
\ee
where $\delta_x$ stands for the Dirac mass at $x$. As explained in \cite{Last_Penrose_PTRF, Last_Penrose_lectures}, the mapping on $A\times \Omega$ given by $(x,\omega)\mapsto D_xY(N_\lambda(\omega))$ is always jointly measurable. For every $q\geq 2$, one then recursively defines the operators 
\be
D^q_{x_1,...,x_q}Y := D_{x_1}[D_{x_2,...,x_q}^{q-1} Y], \quad x_1,...,x_q\in A
\ee
(that also automatically define jointly measurable mappings -- see \cite{Last_Penrose_lectures}), where we set $D=D^1$. 

In the discussion to follow, we will extensively apply a formula by Last and Penrose \cite[Theorem 1.3]{Last_Penrose_PTRF} stating that, if $Y$ is a square-integrable functional of $N_\lambda$, then automatically, for every $q\geq 1$, the mapping on $A^q$
$$
(x_1,..., x_q) \mapsto \E[D^q_{x_1,...,x_q} Y]
$$
is well-defined, symmetric and square-integrable with respect to $\nu^q$ and moreover, for every integer $q\geq 1$, one can choose the symmetric kernel $f_q$ in \eqref{e:chaos} as 
\begin{equation}\label{e:lastpenrose}
f_q(x_1,...,x_q) := \frac{1}{q!} \E[D^q_{x_1,...,x_q} Y].
\end{equation}

We will also make use  of the elementary identities contained in the next statement (the proof is by recursion).

\begin{lem}\label{l:chain}
Let
\be
Y = Y(N_\lambda) := \exp\left\{ i \beta N_\lambda(h)\right\}, 
\ee
where $h\in L^1(\nu)$ is real-valued. Then, for every $q\geq 1$,
\begin{eqnarray}
\label{e:chain1} D^q_{x_1,...,x_q} Y(N_\lambda) &= &Y(N_\lambda) \times D^q_{x_1,...,x_q} Y(\emptyset)\\
\label{e:chain2} &=& Y(N_\lambda) \times \prod_{k=1}^q \left[ \exp\{ i \beta h(x_k)\} -1 \right] \\
\label{e:chain3} &=& Y(N_\lambda)\times (\exp\{ i \beta h\} -1)^{\otimes q} (x_1,...,x_q),
\end{eqnarray}
where $\emptyset$ in \eqref{e:chain1} indicates the point measure with empty support, that is, the zero measure.
\end{lem}

%{\bf {\teal [Remark by Gio: ``I killed the former Remark 5.1, and my subsequent ramblings"]}}
\begin{rmk}\label{r:1}{\rm
It follows from Lemma \ref{l:chain}, relation \eqref{e:lastpenrose} and Kintchine's formula (see, e.g., \cite[Proposition 5.3.5 and Example 5.3.6 (ii)]{Gio_murad_taqqu}) that
\begin{eqnarray}\label{e:1}
e^{ i \beta \hat N_\lambda(h)} &=& e^{\lambda\int_A(e^{i \beta h(x)} - 1 - i\beta h(x) )  \nu(dx)}
\left( 1+ \sum_{q=1}^\infty \frac{1}{q!}  I_q( (\exp\{ i \beta h\} -1)^{\otimes q} ) \right)
\end{eqnarray}
Note that, for the models studied in this paper, $\int_A h\, d\nu = 0$, so trivially $\hat N_\lambda(h) = {N}_\lambda(h)$.}
\end{rmk}
%
%%{\red [[This formula is never referred to in the rest of the paper, but maybe it should be referred to in
%%the section about general gaussian expansion (see later)]]}
%
%%The following result is one of the staples of our analysis. It deals with the  $\delta\to 0$ limit of a Poisson-based small parameter family of $\nu$ controlled random fields, in the sense of Section \ref{ss:generalkernels}.
%
%{\teal \bf [Important remark: the relevance of the final part of the next statement (and of other statements in this Section) relies on the inclusion $C_0^\infty(D) \subset \mathcal{H}_0^\alpha$: I confess that such an inclusion is not completely obvious to me -- see also the first section in the Appendix. I also removed the assumption that $D$ has a $C^1$ boundary, I suddenly did not see anymore why it was needed!] }

We are now ready to present and prove one of the main results of this section.
\begin{theorem}\label{p:chaos1} 
Let $V^{*}_{\lambda,\beta} \in {\mathcal H}^{-\alpha}$ be as in Theorem \ref{ExistenceLayeringField}, and recall the notation \eqref{e:accadelta}. Then, for every $\varphi \in C^{\infty}_{0}(D)$, $V^{*}_{\lambda,\beta}(\varphi)$ admits a Wiener-It\^o chaos expansion \eqref{e:chaos} with kernels $f_q = f^{\varphi}_{q} = f^{\varphi}_{q,\lambda,\beta} \in L^2(\nu^q)$ given by
\be \label{chaos_exp1}
f^{\varphi}_{q} (x_1,...,x_q) = \frac{1}{q!}\int_D \varphi(z) \langle V^*_{\lambda, \beta}(z) \rangle \left[ e^{i\beta h_z }-1\right]^{\otimes q}(x_1,...,x_q) \, dz, \quad q\geq 1.
\ee
\end{theorem}

\begin{proof}
Convergence of $\delta^{-2 \Delta^{*}_{\lambda,\beta}} V^{\delta}_{\lambda,\beta}$ to $V^{*}_{\lambda,\beta}$ in second mean in the Sobolev space ${\mathcal H}^{-\alpha}$
as $\delta \to 0$ implies that, for every $\varphi \in C^{\infty}_{0}$, $\delta^{-2 \Delta^{*}_{\lambda,\beta}} V^{\delta}_{\lambda,\beta}(\varphi)$ converges to
$V^{*}_{\lambda,\beta}(\varphi)$ in $L^2(\P)$:
\be
\big\langle | \delta^{-2 \Delta^{*}_{\lambda,\beta}} V^{\delta}_{\lambda,\beta}(\varphi) - V^{*}_{\lambda,\beta}(\varphi) |^2 \big\rangle \notag
\leq \big\langle || \delta^{-2 \Delta^{*}_{\lambda,\beta}} V^{\delta}_{\lambda,\beta} - V^{*}_{\lambda,\beta} ||^2_{{\mathcal H}^{-\alpha}} \big\rangle ||\varphi||^2_{{\mathcal H}^{\alpha}_{0}} \stackrel{\delta \to 0}{\longrightarrow} 0,
\ee
where we have used the fact that $V^{\delta}_{\lambda,\beta}$ and $V^{*}_{\lambda,\beta}$ are continuous linear functional from ${\mathcal H}^{\alpha}_0$ to $\mathbb R$,
combined with elementary properties of norms. %(specifically, $| \cdot |, || \cdot ||_{{\mathcal H}^{-\alpha}}$ and $|| \cdot ||_{{\mathcal H}^{\alpha}_0}$).
The random variable $\delta^{-2 \Delta^{*}_{\lambda,\beta}} V^{\delta}_{\lambda,\beta}(\varphi)$ is in $L^2(\sigma(N_{\lambda}))$ and therefore admits a chaos expansion.
Moreover, because of the $L^2$ convergence above, the same is true for $V^{*}_{\lambda,\beta}(\varphi)$. In addition, it follows that the kernels $f^{\delta,\varphi}_{q}$ of the
chaos expansion of $\delta^{-2 \Delta^{*}_{\lambda,\beta}} V^{\delta}_{\lambda,\beta}(\varphi)$ converge in $L^2(\nu^q)$ to the kernels $f^{\varphi}_q$ of the chaos expansion
of $V^{*}_{\lambda,\beta}(\varphi)$. We will now show that this last statement implies \eqref{chaos_exp1}.

To obtain this, we note that, using Lemma \ref{l:chain} and Fubini's theorem, the kernels of the chaos expansion of
$\delta^{-2 \Delta^{*}_{\lambda,\beta}} V^{\delta}_{\lambda,\beta}(\varphi)$ can be written as
\begin{eqnarray}
f^{\delta,\varphi}_{q} (x_1,...,x_q) & = & \frac{1}{q!} \delta^{-2\Delta^*_{\lambda,\beta}} \E \Big( D^q_{x_1,\ldots,x_q} \int_D e^{i \beta N_{\lambda}(h^{\delta}_z)} \varphi(z) \, dz \Big) \\
& = & \frac{1}{q!} \int_D \varphi(z) \langle \delta^{-2\Delta^*_{\lambda,\beta}} e^{i \beta N_{\lambda}(h^{\delta}_z)} \rangle \prod_{k=1}^q (e^{i \beta h^{\delta}_z(x_k)} - 1) \, dz \\
& = & \frac{1}{q!} \int_D F^{\delta,\varphi}_{q;x_1,\ldots,x_q}(z) \, dz,
\end{eqnarray}
where
\begin{eqnarray}
F^{\delta,\varphi}_{q;x_1,\ldots,x_q}(z) & := & \varphi(z) \delta^{-2\Delta^*_{\lambda,\beta}} \langle e^{i \beta N_{\lambda}(h^{\delta}_z)} \rangle \Pi_{k=1}^q (e^{i \beta h^{\delta}_z(x_k)} - 1) \\
& \stackrel{\delta \to 0}{\longrightarrow} & \varphi(z) \langle V^*_{\lambda, \beta}(z) \rangle \left[ e^{i\beta h_z }-1\right]^{\otimes q}(x_1,...,x_q)
\end{eqnarray}
pointwise for all $z \in D$.
In addition, we will show below the following upper bound:
\be \label{eq:integrability}
|F^{\delta,\varphi}_{q;x_1,\ldots,x_q}(z)| \leq C^* 2^q |\varphi(z)| d_{z}^{-2\Delta^*_{\lambda,\beta}},
\ee
where $C^*=1$ for $* = loop, disk$ and $C^* = e^{2\lambda(1-\cos\beta)\hat\mu(\gamma : z \in \gamma, \gamma \subset D)}$ for $*=m$.
Before proving \eqref{eq:integrability}, we show how it implies \eqref{chaos_exp1}. Since $C^*$ does not depend on $z$ and $\varphi$ is in $C^{\infty}_0$,
the right hand side of \eqref{eq:integrability} is integrable on $D$ for $2\Delta^*_{\lambda,\beta} < 1$. Hence,
\be
F^{\delta,\varphi}_{q;x_1,\ldots,x_q}(z) \stackrel{\delta \to 0}{\longrightarrow}
\varphi(z) \langle V^*_{\lambda, \beta}(z) \rangle \left[ e^{i\beta h_z }-1\right]^{\otimes q}(x_1,...,x_q) \quad \text{ in } L^1(D),
\ee
which implies that 
\be
f^{\delta,\varphi}_q(x_1,\dots,x_q) \stackrel{\delta \to 0}{\longrightarrow} \frac{1}{q!}\int_D \varphi(z) \langle V^*_{\lambda, \beta}(z) \rangle \left[ e^{i\beta h_z }-1\right]^{\otimes q}(x_1,...,x_q) \, dz
\ee
for every $(x_1,\ldots,x_q) \in {\teal A}^q$. Since $f^{\delta,\varphi}_q$ converges in $L^2(\nu^q)$ to $f^{\varphi}_q$, as $\delta \to 0$, and $\nu^q$ is $\sigma$-finite,
the considerations above imply that \eqref{chaos_exp1} holds.

In order to conclude the proof, we need to prove the upper bound \eqref{eq:integrability}. To do this, using Equation~(4.2) of \cite{camia2016conformal}, we write
\be \label{eq:CGK}
\big\langle e^{i \beta N_{\lambda}(h^{\delta}_z)} \big\rangle^*_D = e^{-\lambda \alpha^*_{\delta,D}(z)(1-\cos\beta)}.
\ee
For $*=loop,disk$, using \eqref{alpha5} and \eqref{alpha8}, respectively, we then have that
\begin{eqnarray}
\delta^{-2\Delta^*_{\lambda,\beta}} \big\langle e^{i \beta N_{\lambda}(h^{\delta}_z)} \big\rangle^*_D & = & \delta^{-2\Delta^*_{\lambda,\beta}}
e^{-\lambda \big( \alpha^*_{\delta,d_{z}}(z) + \alpha^*_{d_{z},D}(z) \big) (1-\cos\beta)} \\
& = & d_{z}^{-2\Delta^*_{\lambda,\beta}} e^{-\lambda \alpha^*_{d_{z},D}(z) (1-\cos\beta)} \label{eq:one_point_function_loop} \\
& \leq & d_{z}^{-2\Delta^*_{\lambda,\beta}}, \label{eq:integrability_loop_disk}
\end{eqnarray}
which suffices to obtain \eqref{eq:integrability}.

For $*=m$, by an application of Lemma \ref{MassiveOnePointFunctionEstimate}, using \eqref{MassiveAlphaEstimate2}, we have that, for $\delta$ sufficiently small,
\be
\alpha^m_{\delta,D}(z) \geq \alpha^{loop}_{\delta,D}(z) - 2 \hat\mu(\gamma : z \in \bar\gamma, \gamma \subset D).
\ee
Hence, using the fact that $\Delta^m_{\lambda,\beta}=\Delta^{loop}_{\lambda,\beta}$ and the upper bound \eqref{eq:integrability_loop_disk} for $*=loop$, we have that
\begin{eqnarray}
\delta^{-2\Delta^m_{\lambda,\beta}} \big\langle e^{i \beta N_{\lambda}(h^{\delta}_z)} \big\rangle^m_D & = & \delta^{-2\Delta^m_{\lambda,\beta}}
e^{-\lambda \alpha^m_{\delta,D}(z) (1-\cos\beta)} \\
& \leq & \delta^{-2\Delta^{loop}_{\lambda,\beta}} e^{-\lambda \alpha^{loop}_{\delta,D}(z) (1-\cos\beta)} e^{2 \lambda (1-cos\beta) \hat\mu(\gamma : z \in \bar\gamma, \gamma \subset D)} \\
& \leq & d_{z}^{-2\Delta^m_{\lambda,\beta}} e^{2 \lambda (1-\cos\beta) \hat\mu(\gamma : z \in \bar\gamma, \gamma \subset D)}, \label{eq:bound_massive_one_point_function}
\end{eqnarray}
as desired.
\end{proof}

\subsection{Chaos expansion for the Gaussian layering field}\label{ss:gcg}

As before, in what follows $G$ is a Gaussian measure on $(A, \mathscr A)$, with $\sigma$-finite and non-atomic control $\nu$. Similarly to what we explained in the previous section, for every $F\in L^2(\sigma(G))$ (real or complex), there exists a unique sequence of symmetric kernels $\{w_q : q=1,2,...\}$, such that $w_q\in L^2(\nu^q)$ and
\begin{equation}\label{e:Gaussianchaos}
F = \E[F]+\sum_{q=1}^\infty I^G_q(w_q),
\end{equation}
where the series converges in $L^2(\mathbb{P})$ and 
\be
 I_q^G(w_q) := \int_{A^q} w(x_1,...,x_q) G(dx_1)\cdots  G(dx_q)
\ee
is a multiple Wiener-It\^o integral of order $q$ of $w_q$ with respect to $G$.
For every $w\in L^2$, we write as before $ I^G_1(w) = \int_A w(x)  G(dx) := G(w)$. See, e.g., \cite{Gio_ivan, Gio_murad_taqqu} for a detailed discussion of these classical objects and results. We will also use the following consequence of \cite[Proposition 5.5.3]{Gio_murad_taqqu}): if $f\in L^2(\nu^q)$ and $g\in L^2(\nu^p)$ are two symmetric kernels, with $p,q\geq 1$, then
\be\label{e:isom2}
\E\left[I^{G}_q(f) \overline{ I^{G}_p(g)}\right] =  \delta_{pq} q! \int_{A^q} f(x_1,...x_q) \overline{ g(x_1,...,x_q)}\, \nu(dx_1)\cdots\nu(dx_q), 
\ee
where $\delta_{pq}$ is Kronecker's symbol.

In order to carry out our analysis, we will make use some Gaussian ersatz of formula \eqref{e:lastpenrose} and Lemma \ref{l:chain}. The needed result corresponds to the celebrated {\it Stroock's formula} for smooth Gaussian functionals (first proved in \cite{stroock_chaos}). Since the general statement of Stroock's formula requires a fair amount of preliminary technical considerations, we content ourselves with one of its corollaries, stated in the lemma below. This yields a Gaussian counterpart to the content of Lemma \ref{l:chain}.

\begin{lem}\label{l:chain2}  For every real-valued $h\in L^2(\nu)$ and every real $\xi$, one has that the chaotic decomposition \eqref{e:Gaussianchaos} of $F =e^{ i \xi G(h)}$ is given by
\begin{eqnarray}\label{e:k}
e^{ i \xi G(h)} \!=\!  \exp\left\{ -\frac{\xi^2}{2}\int_A h(x)^2\, \nu(dx)\right\}
\left(1 + \sum_{q=1}^\infty \frac{ i^q \xi^q}{q!} I^G_q(h^{\otimes q} (x_1,...,x_q)) \right)\!\!.
\end{eqnarray}\end{lem}

\medskip

\begin{remark} {\rm {There exists a direct connection between Equation \eqref{e:1} and Equation \eqref{e:k} that we shall implicitly exploit in the discussion to follow. To see this,} take $h\in L^1(\nu)\cap L^3(\nu)$ (and therefore $h\in L^2(\nu)$) such that $\nu(h) = 0$. Then, it is easily seen that, if $\lambda\to \infty$ and $\beta\to 0$ in such a way that $\lambda\beta^2
\to \xi^2$, one has that
\begin{equation}\label{e:utile1}
e^{\lambda\int_A(e^{i \beta h(x)} - 1)  \nu(dx)}
\to  \exp\left\{ -\frac{\xi^2}{2}\int_A h(x)^2\, \nu(dx)\right\}
\end{equation}
and, for every fixed $x_1,...,x_q\in A$,
\begin{equation}\label{e:utile2}
\lambda^{\frac q2} (\exp\{ i \beta h\} -1)^{\otimes q} (x_1,...,x_q) \to i^q \xi^q  h^{\otimes q} (x_1,...,x_q).
\end{equation}
}
\end{remark}

The following statement is the Gaussian analog of Theorem \ref{p:chaos1}. The proof is very similar to that of of Theorem \ref{p:chaos1}, but there are some differences, so we spell it out
for completeness and for the reader's convenience.
%{The proof (left to the reader) is analogous to the one of Theorem \ref{p:chaos1}, except that now one has to use \eqref{e:k}, instead of \eqref{e:1}. }

\begin{theorem}\label{p:chaos2}
Let $W^{*}_{\xi} \in {\mathcal H}^{-\alpha}$ be as in Theorem \ref{ExistenceGaussianLayeringField}. Then, for every $\varphi \in C^{\infty}_{0}(D)$, $V^{*}_{\xi}(\varphi)$
admits a Wiener-It\^o chaos expansion \eqref{e:Gaussianchaos} with kernels $w_q = w^{\varphi}_{q} = w^{\varphi}_{q,\xi} \in L^2(\nu^q)$ given by
\be \label{chaos_exp2}
w^{\varphi}_{q} (x_1,...,x_q) = \frac{i^q  \xi^q} {q!}\int_D \varphi(z) \langle W^*_{\xi}(z) \rangle  h_z ^{\otimes q}(x_1,...,x_q) \, dz, \quad q\geq 1.
\ee
\end{theorem}

\begin{proof}
Convergence of $\delta^{-2 \Delta^{*}_{\xi}} W^{\delta}_{\xi}$ to $W^{*}_{\xi}$ in second mean in the Sobolev space ${\mathcal H}^{-\alpha}$ as $\delta \to 0$
implies that, for every $\varphi \in C^{\infty}_{0}$, $\delta^{-2 \Delta^{*}_{\xi}} W^{\delta}_{\xi}(\varphi)$ converges to $W^{*}_{\xi}(\varphi)$ in $L^2(\P)$:
\be
\big\langle | \delta^{-2 \Delta^{*}_{\xi}} W^{\delta}_{\xi}(\varphi) - W^{*}_{\xi}(\varphi) |^2 \big\rangle \notag
\leq \big\langle || \delta^{-2 \Delta^{*}_{\xi}} W^{\delta}_{\xi} - W^{*}_{\xi} ||^2_{{\mathcal H}^{-\alpha}} \big\rangle ||\varphi||^2_{{\mathcal H}^{\alpha}_{0}} \stackrel{\delta \to 0}{\longrightarrow} 0,
\ee
where we have used the fact that $W^{\delta}_{\xi}$ and $W^{*}_{\xi}$ are continuous linear functionals from ${\mathcal H}^{\alpha}_0$ to $\mathbb R$,
combined with elementary properties of norms. %(specifically, $| \cdot |, || \cdot ||_{{\mathcal H}^{-\alpha}}$ and $|| \cdot ||_{{\mathcal H}^{\alpha}_0}$).
The random variable $\delta^{-2 \Delta^{*}_{\xi}} W^{\delta}_{\xi}(\varphi)$ is in $L^2(\sigma(G))$ and therefore admit a chaos expansions.
Moreover, because of the $L^2$ convergence above, the same is true for $W^{*}_{\xi}(\varphi)$. In addition, it follows that the kernels $w^{\delta,\varphi}_{q}$ of the
chaos expansion of $\delta^{-2 \Delta^{*}_{\xi}} W^{\delta}_{\xi}(\varphi)$ converge in $L^2(\nu^q)$ to the kernels $w^{\varphi}_q$ of the chaos expansion
of $W^{*}_{\xi}(\varphi)$. We will now show that this last statement implies \eqref{chaos_exp2}.

Using Lemma \ref{l:chain2} and a stochastic Fubini theorem to interchange stochastic and deterministic integration, we can write
\begin{eqnarray}
\delta^{-2\Delta^*_{\xi}} W^{\delta}_{\xi}(\varphi) & = & \delta^{-2\Delta^*_{\xi}} \int_D \varphi(z) e^{i \xi G(h^{\delta}_z)} \, dz = \delta^{-2\Delta^*_{\xi}} \int_D \varphi(z) e^{-\frac{\xi^2}{2} \int_A (h^{\delta}_z(x))^2 \nu(dx)} \, dz \nonumber \\
& + & \delta^{-2\Delta^*_{\xi}} \sum_{q=1}^{\infty} \frac{i^q \xi^q}{q!} \int_D \varphi(z) e^{-\frac{\xi^2}{2} \int_A (h^{\delta}_z(x))^2 \nu(dx)} I^G_q ( (h^{\delta}_z)^{\otimes q} )  \, dz \nonumber \\
& = & \delta^{-2\Delta^*_{\xi}} \int_D \varphi(z) e^{-\frac{\xi^2}{2} \int_A (h^{\delta}_z(x))^2 \nu(dx)} \, dz \notag \\
& + & \delta^{-2\Delta^*_{\xi}} \sum_{q=1}^{\infty} I^G_q \Big( \frac{i^q \xi^q}{q!} \int_D \varphi(z) e^{-\frac{\xi^2}{2} \int_A (h^{\delta}_z(x))^2 \nu(dx)} (h^{\delta}_z)^{\otimes q} ) \, dz \Big).
\end{eqnarray}
Comparing the last expression with \eqref{e:Gaussianchaos}, we see that the kernels of the chaos expansion of $\delta^{-2 \Delta^{*}_{\xi}} W^{\delta}_{\xi}(\varphi)$ can be written as
\begin{eqnarray}
w^{\delta,\varphi}_{q} (x_1,...,x_q) & = & \delta^{-2\Delta^*_{\xi}} \frac{i^q \xi^q}{q!} \int_D \varphi(z) e^{-\frac{\xi^2}{2} \int_A (h^{\delta}_z(x))^2 \nu(dx)} (h^{\delta}_z)^{\otimes q} (x_1,\ldots,x_q) \, dz \nonumber \\
& = & \delta^{-2\Delta^*_{\xi}} \frac{i^q \xi^q}{q!} \int_D W^{\delta,\varphi}_{q;x_1,\ldots,x_q}(z) \, dz,
\end{eqnarray}
where
\be
W^{\delta,\varphi}_{q;x_1,\ldots,x_q}(z) := \varphi(z) \delta^{-2\Delta^*_{\xi}} e^{-\frac{\xi^2}{2} \int_A (h^{\delta}_z(x))^2 \nu(dx)} (h^{\delta}_z)^{\otimes q} (x_1,\ldots,x_q)
\ee
Using the equality (see, e.g., \cite[Example 5.3.6 (i)]{Gio_murad_taqqu})
\be \label{eq:one_point_function_Gauss_1}
e^{-\frac{\xi^2}{2} \int_A (h^{\delta}_z(x))^2 \nu(dx)} = \big\langle W^{\delta}_{\xi}(z) \big\rangle,
\ee
we have that
\begin{eqnarray}
W^{\delta,\varphi}_{q} (x_1,...,x_q) & = & \varphi(z) \delta^{-2\Delta^*_{\xi}} \big\langle W^{\delta}_{\xi}(z) \big\rangle \prod_{k=1}^q h^{\delta}_z(x_k) \notag \\
& \stackrel{\delta \to 0}{\longrightarrow} & \varphi(z) \big\langle W^{*}_{\xi}(z) \big\rangle \prod_{k=1}^q h_z(x_k)
\end{eqnarray}
pointwise for all $z \in D$. Moreover, using \eqref{eq:one_point_function_Gauss_1} and the definition of $h^{\delta}_z$ (see Section \ref{ss:GaussianLayering}), we see that
\be \label{eq:one_point_function_Gauss_2}
\big\langle W^{\delta}_{\xi}(z) \big\rangle= e^{-\frac{\xi^2}{2} \int_A (h^{\delta}_z(x))^2 \nu(dx)} = e^{-\frac{\xi^2}{2} \alpha^*_{\delta,D}(z)}.
\ee
Comparing this with \eqref{eq:CGK}, and using the arguments following that equation, we obtain the uniform bound
\be
|W^{\delta,\varphi}_{q;x_1,\ldots,x_q}(z)| \leq C^* |\varphi(z)| d_{z}^{-2\Delta^*_{\xi}},
\ee
where $C^*=1$ for $* = loop, disk$ and $C^* = e^{2\lambda(1-\cos\beta)\hat\mu(\gamma : z \in \gamma, \gamma \subset D)}$ for $*=m$,
as in Equation \eqref{eq:integrability}. Since $C^*$ does not depend on $z$ and $\varphi$ is in $C^{\infty}_0$,
\be
W^{\delta,\varphi}_{q;x_1,\ldots,x_q}(z) \stackrel{\delta \to 0}{\longrightarrow}
\varphi(z) \langle W^*_{\xi}(z) \rangle h_z^{\otimes q}(x_1,...,x_q) \quad \text{ in } L^1(D),
\ee
which implies that 
\be
w^{\delta,\varphi}_q(x_1,\dots,x_q) \stackrel{\delta \to 0}{\longrightarrow} \frac{i^q \xi^q}{q!}\int_D \varphi(z) \langle W^*_{\xi}(z) \rangle h_z^{\otimes q}(x_1,...,x_q) \, dz
\ee
for every $(x_1,\ldots,x_q) \in A^q$. Since $w^{\delta,\varphi}_q$ converges in $L^2(\nu^q)$ to $w^{\varphi}_q$, as $\delta \to 0$, and $\nu^q$ is $\sigma$-finite,
the considerations above imply that \eqref{chaos_exp2} holds.
\end{proof}

\section{Convergence of layering fields to complex Gaussian multiplicative chaos}\label{s:convtoGMC}

In this section, using the Wiener-It\^o chaos expansion presented in the previous section, we prove that the Poisson layering fields constructed in Theorem \ref{ExistenceLayeringField}
converge to the Gaussian layering fields constructed in Theorem \ref{ExistenceGaussianLayeringField} in the appropriate limit of the parameters $\beta \to 0$ and $\lambda \to \infty$.
The kind of converge we will be dealing with is considerably weaker than the one described in \eqref{2.4bis} (and encountered several times in the previous sections). Obtaining convergence
in a stronger sense (e.g., as in \eqref{2.4bis}) seems to be a challenging problem, and we leave it open for future research.

\begin{definition}[\bf Finite dimensional distributions]\label{d:fdd} {\rm  Fix a simply connected bounded domain $D$, and consider a family of random distributions $\{V, V_n : n\geq 1\}$ with values in the negative Sobolev space $\mathcal{H}^{ -\alpha}$, for some $\alpha>0$. We say that $V_n$ {\it converges to $V$ in the sense of finite dimensional distributions}, written
\be
V_n\fdd V ,
\ee
if, for every $\varphi\in C_0^\infty(D)$, one has that $V_n(\varphi)$ converges in distribution (as a complex-valued random variable) to $V(\varphi)$. By linearity and the Cramer-Wold device, if $V_n\fdd V$, then 
\be
(V_n(\varphi_1),...,(V_n(\varphi_m))  \Longrightarrow (V(\varphi_1),..., V(\varphi_m)),
\ee
for every finite vector of test functions $(\varphi_1,...,\varphi_m)$, where $\Longrightarrow$ stands for convergence in distribution of random vectors with values in $\mathbb{C}^m$.
}
\end{definition}

We are now ready to state and prove our main convergence result. Note that, although Theorem~\ref{ExistenceGaussianLayeringField} guarantees the existence of $W^*_{\xi}$
for $\xi^2 \in (0,10)$ for $*=loop,m$ (respectively, $\xi^2 \in (0,2/\pi)$ for $*=disk$), we can prove the result below only for $\xi^2 \in (0,5)$ for $*=loop,m$ (respectively,
$\xi^2 \in (0,\pi)$ for $*=disk$). We do not know if this is an artifact of our methods or if the convergence result does not hold for $\xi^2 \geq 5$ (respectively, $\xi^2 \geq 1/\pi$).

\begin{theorem}\label{p:cil}
Take $\lambda>0$ and $\beta \in [0,2\pi)$, let $*=loop, m$ and let $D \subsetneq {\mathbb C}$ be any domain conformally equivalent to the unit disk $\D$.
Let $V^*_{\lambda,\beta} \in {\mathcal H}^{-\alpha}$ be a Poisson layering field in $D$ as in Theorem~\ref{ExistenceLayeringField}
and $W^*_{\xi} \in {\mathcal H}^{-\alpha}$ be a Gaussian layering field in $D$ as in Theorem~\ref{ExistenceGaussianLayeringField}.
If $\Delta^*_{\xi} = \xi^2/20 \in (0,1/4)$ (that is, $\xi^2 \in (0 , 5)$), then
$V^*_{\lambda,\beta} \fdd W^*_{\xi}$ as $\lambda \to \infty$ and $\beta \to 0$ with $\lambda\beta^2 \to \xi^2$.
The same conclusion holds for $*=disk$ with $D=\D$ and $\xi^2 \in (0,1/\pi)$.
\end{theorem}

\begin{proof}
We only focus on $*=loop,m$ since the proof for $*=disk$ is completely analogous.
Thanks to Theorem~\ref{thm:conformal_covariance}, it suffices to consider fields defined on the unit disk $\D$.
Accordingly, for the rest of the proof we restrict our attention to the case $D=\D$. We will subdivide the proof into two parts. For each $\varphi \in C^{\infty}_{0}(\D)$, let
\be
V^*_{\lambda,\beta}(\varphi) = \langle V^*_{\lambda,\beta}(\varphi) \rangle + \sum_{q=1}^{\infty} I^{N_{\lambda}}_q(f^{\varphi}_{q,\lambda,\beta})
\ee
and
\be
W^*_{\xi}(\varphi) = \langle W^*_{\xi}(\varphi) \rangle + \sum_{q=1}^{\infty} I^G_q(w^{\varphi}_{q,\xi})
\ee
be the Wiener-It\^o chaos expansions of $V^*_{\lambda,\beta}(\varphi)$ and $W^*_{\xi}(\varphi)$, respectively, as in Theorems~\ref{p:chaos1} and~\ref{p:chaos2}.
In the second part of the proof, we will show that, for every $\varphi \in C^{\infty}_{0}(\D)$, the following conditions hold.
\begin{enumerate}
\item As $\lambda\to \infty$ and $ \beta \to 0$ with $\lambda\beta^2\to \xi^2$,
\be \label{eq:convergence-mean}
\langle V^*_{\lambda,\beta}(\varphi) \rangle \to \langle W^*_{\xi}(\varphi) \rangle.
\ee
\item One has that
\begin{equation}\label{e:sum}
\sum_{q=1}^\infty q! \| w_{q,\xi}^{\varphi}\|^2_{L^2(\nu^q)} <\infty.
\end{equation}
\item As $\lambda\to \infty$ and $ \beta \to 0$ with $\lambda\beta^2\to \xi^2 \in (0,5)$,
\begin{equation}\label{e:l2}
\| \lambda^{\frac q2}  f^{\varphi}_{q,\lambda,\beta} - w^{\varphi}_{q,\xi}\|^2_{L^2(\nu^q)} \rightarrow 0, \quad \forall q\geq 1.
\end{equation} 
\item The following asymptotic relation holds:

\begin{equation}\label{e:resto}
\lim_{N \to \infty} \limsup_{\lambda, \beta} 
\sum_{q=N+1}^\infty q! \| { \lambda^{\frac q2}} f^{\varphi}_{q,\lambda,\beta}\|^2_{L^2(\nu^q)} = 0,
\end{equation}
where the $ \limsup$ is for $\lambda\to \infty$ and $ \beta \to 0$, with $\lambda\beta^2\to \xi^2 \in (0 , 5)$.
\end{enumerate}
Before showing that conditions 1-4 above hold, we now prove how they imply that, for every $\varphi \in C^{\infty}_{0}(\D)$,
\be \label{claim}
V^*_{\lambda,\beta}(\varphi) \Longrightarrow W^*_{\xi}(\varphi) \text{ as } \lambda \to \infty \text{ and } \beta \to 0 \text{ with } \lambda\beta^2 \to \xi^2 \in (0,5).
\ee
The proof of this claim consists in a generalization of the arguments used in \cite[p. 337-339]{S}.
For every $q\geq 1$, we say that an element $g_q$ of $L^2(\nu^q)$ is \emph{simple} if $g_q$ is symmetric and has the form
\be
(x_1,...,x_q) \mapsto g_q(x_1,...,x_q) = \sum_{\ell=1}^R c(\ell)\, {\bf 1}_{A^{\ell}_1}\cdots {\bf 1}_{A^{\ell}_q},
\ee
where $R$ is a finite integer, $c(\ell)$ are constants and, for fixed $\ell$, the $A^{\ell}_j$'s are pairwise disjoint sets with finite $\nu$-measure.
Note that, if $Z$ is either one of the random measures $G$ or ${N}_\lambda$, then $I^Z_q(g_q)$ is the homogeneous polynomial of degree $q$ given by
\be
I^Z_q(g_q) =  \sum_{\ell=1}^R c(\ell)\, Z(A^{\ell}_1 ) \cdots Z(A^{\ell}_q).
\ee
We now observe that, in view of \eqref{e:sum} and \eqref{e:resto}, for every $\varepsilon >0$ there exist $N_0\geq 1$ and $\lambda_0, \beta_0>0$, such that: (a)
\be \label{e:wsmall}
\sum_{q=N_0+1}^\infty q! \| w_{q,\xi}^{\varphi}\|^2_{L^2(\nu^q)} \leq \varepsilon,
\ee
(b) for every $\lambda\geq \lambda_0$ and $\beta \leq \beta_0$,
\be\label{e:zuzz}
\sum_{q=N_0+1}^\infty q! \| {\lambda^{\frac q2}} f_{q,\beta, \lambda}^{\varphi} \|^2_{L^2(\nu^q)}\leq \varepsilon,
\ee
and (c) as a consequence of \eqref{e:l2},
\begin{equation}\label{e:zza}
\sum_{q=1}^{N_0}  q! \| {\lambda^{\frac q2}} f_{q,\beta, \lambda}^{\varphi} -w_{q,\xi}^{\varphi}  \|^2_{L^2(\nu^q)}\leq \varepsilon.
\end{equation}
The key point is now that, by a standard density argument (see e.g. \cite[Lemma 5.5.2]{Gio_murad_taqqu}), for $N_0$ as above, there exist simple kernels $g_q\in L^2(\nu^q)$, $q=1,...,N_0$, such that
\be
\sum_{q=1}^{N_0} q! \| g_q - w_{q,\xi}^{\varphi}\|^2_{L^2(\nu^q)}\leq \varepsilon,
\ee
and consequently, for $\lambda\geq \lambda_0$ and $\beta \leq \beta_0$,
\begin{eqnarray} \label{eq:closeness}
&& \sum_{q=1}^{N_0} q! \| {\lambda^{\frac q2}} f_{q,\beta, \lambda}^{\varphi} - g_q\|^2_{L^2(\nu^q)} +  \sum_{q=1}^{N_0} q! \| g_q - w_{q,\xi}^{\varphi}\|^2_{L^2(\nu^q)} \notag \\
&& \leq 2 \sum_{q=1}^{N_0}  q! \| {\lambda^{\frac q2}} f_{q,\beta,\lambda}^{\varphi} -w_{q,\xi}^{\varphi}  \|^2_{L^2(\nu^q)} +3 \sum_{q=1}^{N_0} q! \| g_q - w_{q,\xi}^{\varphi}\|^2_{L^2(\nu^q)}\leq 5\varepsilon.
\end{eqnarray}
Now consider a bounded $1$-Lipschitz mapping $\psi : \mathbb{C}\to \R$. We have that, for fixed $\varepsilon$ and $N_0$ and for $\lambda\geq \lambda_0$ and $\beta \leq \beta_0$ as above,
\begin{eqnarray*}
 && | \E[\psi ( V^*_{\lambda, \beta}(\varphi) )  ] - \E[\psi ( W^*_{\xi}(\varphi) ) ] | \leq  (2+2\sqrt{5}) \sqrt{\varepsilon} \\
 &&\quad + \left| \E\left[\psi \left(  \langle V^*_{\lambda, \beta}(\varphi) \rangle + \sum_{q=1}^{N_0} I^{N_{\lambda}}_q(\lambda^{-q/2} g_q) \right)  \right] - \E\left[\psi \left( \langle W^*_{\xi}(\varphi) \rangle + \sum_{q=1}^{N_0} I_q^G(g_q) \right) \right]  \right| .
\end{eqnarray*}
The previous bound is obtained by applying twice the triangle inequality, and then by exploiting the estimate (deduced by virtue of the Lipschitz nature of $\psi$ and of the isometric property \eqref{e:isom})
\begin{eqnarray*}
&& \left| \E[\psi ( V^*_{\lambda, \beta}(\varphi) )  ] - \E\left[\psi \left(  \langle V^*_{\lambda, \beta}(\varphi) \rangle + \sum_{q=1}^{N_0} I^{N_{\lambda}}_q(\lambda^{-q/2} g_q) \right)  \right]\right|\\ 
&& \leq \sqrt{  \sum_{q=N_0+1}^\infty q! \| {\lambda^{\frac q2}} f_{q,\beta, \lambda,\varphi} \|^2_{L^2(\nu^q)} } +\sqrt{\sum_{q=1}^{N_0} q! \| { \lambda^{\frac q2}} f_{q,\beta, \lambda,\varphi} - g_q\|^2_{L^2(\nu^q)}  },
\end{eqnarray*}
together with \eqref{e:zuzz}, \eqref{eq:closeness} and \eqref{e:wsmall}, as well as a similar bound (based instead on the isometric property \eqref{e:isom2}) for the quantity
\begin{equation*}
\left| \E[\psi ( W^*_{\xi}(\varphi) ) ] - \E\left[\psi \left( \langle W^*_{\xi}(\varphi) \rangle + \sum_{q=1}^{N_0} I_q^G(g_q) \right) \right]\right|.
\end{equation*}
Using \eqref{eq:convergence-mean} and \eqref{e:clt}, we eventually infer that
\be
\langle V^*_{\lambda, \beta}(\varphi) \rangle + \sum_{q=1}^{N_0} I^{N_{\lambda}}_q(\lambda^{-q/2} g_q)\Longrightarrow \langle W^*_{\xi}(\varphi) \rangle + \sum_{q=1}^{N_0} I_q^G(g_q),
\ee
as $\lambda\to \infty$. In view of this, since $\varepsilon$ is arbitrary, the Portmanteau theorem implies the desired conclusion \eqref{claim}. We are now going to verify that Conditions 1-4 are satisfied in our setting. To this end, recall the notation \eqref{alpha6}, applied below in the case $D = \mathbb{D}$, as well as \eqref{e:accadelta} and \eqref{e:rsign}.
%Since $D=\D$ is fixed and there is no danger of confusion, in the rest of the proof we will use $d_z$ for $d_{z,\D} = \text{dist}(z,\partial\D)$.

\noindent 1. This point follows from Corollary~\ref{cor:limits}.

\noindent 2. For $*=loop$, using \eqref{chaos_exp2} and \eqref{eq:one_point_function_gauss}, we have
\begin{eqnarray}
q! \| w_{q,\xi}^{\varphi}\|^2_{L^2(\nu^q)}
& = & q! \int_{A^q} \Big| \frac{i^q \xi^q}{q!} \int_{\D} \varphi(z) \langle W^*_{\xi}(z) \rangle \big( \Pi_{k=1}^{q} h_z(x_k) \big) dz \Big|^2 \prod_{k=1}^{q} \nu(dx_k) \\
& \leq & \frac{\xi^{2q}}{q!} \int_{\D} \int_{\D} \varphi(z) \varphi(t) d_{z}^{-2\Delta^*_{\xi}} d_{t}^{-2\Delta^*_{\xi}} \big( \alpha^*_{\D}(z,t) \big)^q dz dt. \label{term2inequality}
\end{eqnarray}
We now use the inequality $\alpha^{loop}_{\D}(z,t) \leq \alpha^{loop}_{B_{z,|z-t|},B_{z,2}}(z)$ and \eqref{alpha5} to obtain
\be \label{eq:intermediate_bound_case1}
q! \| w_{q,\xi}^{\varphi}\|^2_{L^2(\nu^q)}
\leq \int_{\D} \int_{\D} \varphi(z) \varphi(t) d_{z}^{-2\Delta^*_{\xi}} d_{t}^{-2\Delta^*_{\xi}} \frac{1}{q!} \Big( \frac{\xi^2}{5} \log\frac{2}{|z-t|} \Big)^q dz dt,
\ee
which implies, together with Lemma~\ref{finitenessunitdisk},
\be \label{eq:final_bound_case1}
\sum_{q=1}^{\infty} q! \| w_{q,\xi}^{\varphi}\|^2_{L^2(\nu^q)}
\leq \int_{\D} \int_{\D} \varphi(z) \varphi(t) d_{z}^{-2\Delta^*_{\xi}} d_{t}^{-2\Delta^*_{\xi}} \Big[ \Big( \frac{2}{|z-t|} \Big)^{\frac{\xi^2}{5}} - 1 \Big] dz dt < \infty.
\ee
For $*=m$, using \eqref{eq:massive_one_point_function_gauss} leads again to \eqref{eq:intermediate_bound_case1} and \eqref{eq:final_bound_case1},
with the additional insertion under the integral signs of the factor
\be
e^{\frac{\xi^2}{2} \hat\mu_{\D}(\gamma: z \in \bar\gamma, \diam{\gamma} \leq d_z)} e^{\frac{\xi^2}{2} \hat\mu_{\D}(\gamma: t \in \bar\gamma, \diam{\gamma} \leq d_t)},
\ee
which gives again the desired conclusion.

\noindent 3. Using \eqref{chaos_exp1} and \eqref{chaos_exp2}, and the obvious inequality $|a+b|^2 \leq {\teal} 2 |a|^2 + {\teal 2} |b|^2$, we have that
\begin{eqnarray}
&& \Big| { \lambda^{\frac q2}} f^{\varphi}_{q,\lambda,\beta}(x_1,\ldots,x_q) - w^{\varphi}_{q,\xi}(x_1,\ldots,x_q) \Big|^2 \notag \\
&& \quad = \Big| \frac{\lambda^{q/2}}{q!} \int_{\D} \varphi(z) \langle V^*_{\lambda,\beta}(z) \rangle \big( e^{i \beta h_z} - 1 \big)^{\otimes q}(x_1,...,x_q) dz \notag \\
&& \quad \quad \quad - \frac{(i \xi)^q}{q!} \int_{\D} \varphi(z) \langle W^*_{\xi}(z) \rangle h_z^{\otimes q}(x_1,...,x_q) dz \Big|^2 \\
&& \quad = \Big| \frac{1}{q!} \int_{\D} \varphi(z) \langle W^*_{\xi}(z) \rangle \Big[ \prod_{k=1}^q \lambda^{1/2} \big(e^{i \beta h_z(x_k)} - 1 \big) - \prod_{k=1}^q i \xi h_z(x_k) \Big] dz \notag \\
&& \quad \quad \quad + \frac{\lambda^{q/2}}{q!} \int_{\D} \varphi(z) \big( \langle V^*_{\lambda,\beta}(z) \rangle - \langle W^*_{\xi}(z) \rangle \big) \big( e^{i \beta h_z} - 1 \big)^{\otimes q}(x_1,...,x_q) dz \Big|^2 \\
&& \quad \leq 2 \Big| \frac{1}{q!} \int_{\D} \varphi(z) \langle W^*_{\xi}(z) \rangle \Big[ \prod_{k=1}^q \lambda^{1/2} \big(e^{i \beta h_z(x_k)} - 1 \big) - \prod_{k=1}^q i \xi h_z(x_k) \Big] dz \Big|^2 \notag \\
&& \quad \quad \quad + 2 \Big| \frac{\lambda^{q/2}}{q!} \int_{\D} \varphi(z) \big( \langle V^*_{\lambda,\beta}(z) \rangle - \langle W^*_{\xi}(z) \rangle \big) \big( e^{i \beta h_z} - 1 \big)^{\otimes q}(x_1,...,x_q) dz \Big|^2.
\end{eqnarray}
Letting
\be
I_1(x_1,...,x_q) := \int_{\D} \varphi(z) \langle W^*_{\xi}(z) \rangle \Big[ \prod_{k=1}^q \lambda^{1/2} \big(e^{i \beta h_z(x_k)} - 1 \big) - \prod_{k=1}^q i \xi h_z(x_k) \Big] dz
\ee
and
\be
I_2(x_1,...,x_q) := \int_{\D} \varphi(z) \big( \langle V^*_{\lambda,\beta}(z) \rangle - \langle W^*_{\xi}(z) \rangle \big) \big( e^{i \beta h_z} - 1 \big)^{\otimes q}(x_1,...,x_q) dz,
\ee
we have
\begin{eqnarray}
\| \lambda^{\frac q2}  f^{\varphi}_{q,\lambda,\beta} - w^{\varphi}_{q,\xi}\|^2_{L^2(\nu^q)} & \leq & \frac{2}{(q!)^2} \int_{A^q} \big| I_1(x_1,...,x_q) \big|^2 \prod_{k=1}^{q} \nu(dx_k) \notag \\
& + & \quad \frac{2 \lambda^q}{(q!)^2} \int_{A^q} \big| I_2(x_1,...,x_q) \big|^2 \prod_{k=1}^{q} \nu(dx_k). \label{term3inequality}
\end{eqnarray}
For $*=loop$, using Fubini's theorem and the independence between the loops $\gamma_k$ and their signs $\epsilon_k$ under $\nu^q$, as well as \eqref{eq:one_point_function_gauss},
the first term in the right hand side of \eqref{term3inequality} can be bounded as follows:
\begin{eqnarray}
&& \int_{A^q} \big| I_1(x_1,...,x_q) \Big|^2 \prod_{k=1}^{q} \nu(dx_k) \notag \\
&& \quad \leq \int_{ \{-1,1\} ^q } \Big| \prod_{k=1}^q \lambda^{1/2} \big(e^{i \beta \epsilon_k} - 1 \big) - \prod_{k=1}^q i \xi \epsilon_k \Big|^2 S(d\epsilon_1)\cdots S(d\epsilon_q) \times \notag \\
&&  \quad \quad\quad \quad \quad \quad \times \int_{\D} \int_{\D} \varphi(z) \varphi(t) d_{z}^{-2\Delta^{loop}_{\xi}} d_{t}^{-2\Delta^{loop}_{\xi}} \big( \alpha^{loop}_{\D}(z,t) \big)^q dz dt. \label{subterm1}
\end{eqnarray}
For $*=m$, one can use \eqref{eq:massive_one_point_function_gauss} to obtain the same inequality, but with the additional factor
\be \label{additional_factor}
e^{\frac{\xi^2}{2} \hat\mu_{\D}(\gamma: z \in \bar\gamma, \diam{\gamma} \leq d_z)} e^{\frac{\xi^2}{2} \hat\mu_{\D}(\gamma: t \in \bar\gamma, \diam{\gamma} \leq d_t)}
\ee
under the second integral in the right hand side of \eqref{subterm1}.
The second term on the right hand side of \eqref{subterm1} is the same as the double integral in the right hand side of \eqref{term2inequality}, which has already been shown to be finite;
the first term is clearly finite and tends to $0$ as $\lambda \to \infty, \beta \to 0$ with $\lambda\beta^2 \to \xi^2$. Hence, we conclude that
\be
\int_{A^q} \big| I_1(x_1,...,x_q) \Big|^2 \prod_{k=1}^{q} \nu(dx_k) \to 0 \text{ as } \lambda \to \infty, \beta \to 0 \text{ with } \lambda\beta^2 \to \xi^2.
\ee
We now turn to the second summand on the right-hand side of \eqref{term3inequality}.
For $*=loop$, using again Fubini's theorem and the independence between the loops $\gamma_k$ and their signs $\epsilon_k$, as well as \eqref{eq:one_point_function_gauss},
the second term on the right-hand side of \eqref{term3inequality} can be bounded as follows:
\begin{eqnarray}
&& \int_{A^q} \big| I_2(x_1,...,x_q) \big|^2 \prod_{k=1}^{q} \nu(dx_k) \notag \\
&& \quad \leq \lambda^q \int_{\{-1,1\}^q} \Big| \prod_{k=1}^q \big(e^{i \beta \epsilon_k} - 1 \big) \Big|^2 S(d\epsilon_1)\cdots S(d\epsilon_q) \times  \notag \\
&& \quad \quad \times \int_{\D} \int_{\D} \varphi(z) \varphi(t) \left( \frac{e^{-\lambda (1-\cos\beta) \alpha^{loop}_{d_z,\D}(z)}}{d_z^{2\Delta^{loop}_{\lambda,\beta}}} - \frac{e^{-\frac{\xi^2}{2} \alpha^{loop}_{d_z,\D}(z)}}{d_z^{2\Delta^{loop}_{\xi}}} \right) \notag \\
&& \quad \quad \quad  \left( \frac{e^{-\lambda (1-\cos\beta) \alpha^{loop}_{d_t,\D}(t)}}{d_t^{2\Delta^{loop}_{\lambda,\beta}}} - \frac{e^{-\frac{\xi^2}{2} \alpha^{loop}_{d_t,\D}(t)}}{d_t^{2\Delta^{loop}_{\xi}}} \right) \big( \alpha^{loop}_{\D}(z,t) \big)^q dz dt := A\times B. \label{subterm2}
\end{eqnarray}
For $*=m$, one can use \eqref{eq:massive_one_point_function_loop} and \eqref{eq:massive_one_point_function_gauss} to obtain an upper bound of the same type,
with some additional factors that are bounded over $\D$. To conclude, we will now use the fact that
\be \label{e:goo}
\alpha^{m}_{\D}(z,t) \leq \alpha^{loop}_{\D}(z,t) \leq \frac 15 \log \Big( \frac{2}{|z-t|} \Big).
\ee
Indeed, since $\xi^2 < 5$ by assumption, one has that, for $\lambda$ large enough and $\beta$ small enough, the term $A$ defined in \eqref{subterm2} verifies $A\leq \eta^q$, for some $\eta<5$. Exploiting \eqref{e:goo} together with the fact that $x^q < q! e^x$ for every $q\geq 1$ and $x>0$, we conclude that, for $\lambda$ large and $\beta$ small enough,
\begin{eqnarray*}
&&A\times B \leq q! \int_{\D} \int_{\D} \varphi(z) \varphi(t) \left( \frac{e^{-\lambda (1-\cos\beta) \alpha^{loop}_{d_z,\D}(z)}}{d_z^{2\Delta^{loop}_{\lambda,\beta}}} - \frac{e^{-\frac{\xi^2}{2} \alpha^{loop}_{d_z,\D}(z)}}{d_z^{2\Delta^{loop}_{\xi}}} \right) \notag \\
&&\quad\quad\quad\quad\quad\quad\quad\quad \quad\quad\quad\quad\quad\left( \frac{e^{-\lambda (1-\cos\beta) \alpha^{loop}_{d_t,\D}(t)}}{d_t^{2\Delta^{loop}_{\lambda,\beta}}} - \frac{e^{-\frac{\xi^2}{2} \alpha^{loop}_{d_t,\D}(t)}}{d_t^{2\Delta^{loop}_{\xi}}} \right) \left( \frac{2}{|z-t|} \right)^{\eta/5} dz dt
\end{eqnarray*}
Such a relation, together with the dominated convergence theorem and Lemma~\ref{finitenessunitdisk}, yields immediately that
\be
\int_{A^q} \big| I_2(x_1,...,x_q) \big|^2 \prod_{k=1}^{q} \nu(dx_k) \to 0 \text{ as } \lambda \to \infty, \beta \to 0 \text{ with } \lambda\beta^2 \to \xi^2.
\ee

\noindent 4. For $*=loop$, using \eqref{chaos_exp1} and Fubini's theorem, we have that, for every $n\geq 1$,
\begin{eqnarray}
&& \sum_{q=n}^{\infty} q! \| { \lambda^{\frac q2}} f^{\varphi}_{q,\lambda,\beta}\|^2_{L^2(\nu^q)} \notag \\
&& \quad = \sum_{q=1}^{\infty} \frac{\lambda^q}{q!} \int_{A^q} \Big| \int_{\D} \varphi(z) \langle V^*_{\lambda,\beta}(z) \rangle \big( e^{i \beta h_z} - 1 \big)^{\otimes q}(x_1,...,x_q) dz \Big|^2 \prod_{k=1}^{q} \nu(dx_k) \\
&& \quad \leq \sum_{q=n}^{\infty} \frac{(\lambda\beta^2)^q}{q!} \int_{\D} \int_{\D} \varphi(z) \varphi(t) d_{z}^{-2\Delta^{loop}_{\lambda,\beta}} d_{t}^{-2\Delta^{loop}_{\lambda,\beta}} \big( \alpha^{loop}_{\D}(z,t) \big)^q dz dt \\
&& \quad \leq \sum_{q=n}^{\infty} \frac{(\lambda\beta^2)^q}{q!} \int_{\D} \int_{\D} \varphi(z) \varphi(t) \left( \frac{1}{d_{z}}\vee 1\right)^{2\Delta^{loop}_{\lambda,\beta}} \left( \frac{1}{d_{t}}\vee 1\right)^{2\Delta^{loop}_{\lambda,\beta}} \big( \alpha^{loop}_{\D}(z,t) \big)^q dz dt \notag\\
&&  \quad := R^{loop}(\lambda, \beta, n).
\end{eqnarray}
For $*=m$, using \eqref{eq:bound_massive_one_point_function}, we obtain the same upper bound with the additional factor
\be
e^{2 \lambda (1- \cos \beta) \hat\mu_{\D}(\gamma: z \in \bar\gamma, \diam{\gamma} \leq d_z)} e^{2 \lambda (1- \cos \beta) \hat\mu_{\D}(\gamma: t \in \bar\gamma, \diam{\gamma} \leq d_t)}
\ee
under the integral signs; the resulting upper bound is denoted by $R^{m}(\lambda, \beta, n)$. Now, reasoning as above, for $\lambda$ large and $\beta$ small enough one has that $\lambda\beta^2<\eta$, for some $\eta<5$, and also $2\Delta^{loop}_{\lambda,\beta} = \frac{\lambda}{5}(1-\cos \beta)< 1/2$. These considerations imply that, for every $n\geq 1$,
$$
\limsup_{\lambda, \beta} R^{loop}(\lambda, \beta, n)\leq \|\varphi\|_\infty^2 \int_{\D} \int_{\D} \left( \frac{1}{d_{z}}\vee 1\right)^{1/2 } \left( \frac{1}{d_{t}}\vee 1\right)^{1/2} H_n( \eta \alpha^{loop}_{\D}(z,t) )  dz dt,
$$
where $H_n(x) := \sum_{q=n}^{\infty} \frac{x^q}{q! }$, $x>0$. Since $H_n(x)\to 0$, as $n\to \infty$ and $H_n(x)< e^x$, by combining \eqref{e:goo} with Lemma \ref{finitenessunitdisk} and dominated convergence, we deduce that
\be
\lim_{n\to \infty} \limsup_{\lambda, \beta} R^{loop}(\lambda, \beta, n) = 0.
\ee
By a completely analogous argument (details are left to the reader) one can show that
\be
\lim_{n\to \infty} \limsup_{\lambda, \beta} R^{m}(\lambda, \beta, n) = 0,
\ee
and the proof is therefore concluded.
\end{proof}

\appendix
\section{Appendix }

\subsection{Background on Sobolev spaces}\label{a:sobolev}

We follow \cite[Section 5]{Spin_loops_berg_camia_lis}, and refer the reader e.g. to \cite[Chapter II-4]{BJS_Book} for a classical presentation of the material discussed in the present section. We fix a bounded simply connected domain $D\subset \R^2$, and denote by $\mathcal{H}_0^1$ the classical Sobolev Hilbert space associated with $D$, which is obtained as the closure of $C_0^\infty(D)$ (the class of smooth functions on $D$ having compact support) with respect to the norm $\|f\|^2_{\mathcal{H}_0^1} := \int_D | \nabla f(z)|^2\, dz$. There exists an orthonormal basis $\{u_i : i\geq 1\}$ of $L^2(D)$ composed of (normalized) eigenfunctions of the Laplace operator $-\Delta$ on $D$ with Dirichlet boundary condition; one has also that $\{u_i\}$ is an orthogonal basis of $\mathcal{H}_0^1$, and Green's formula implies that $\|u_i \|^2_{\mathcal{H}_0^1} = \lambda_i$, where $\lambda_i$ is the Laplace eigenvalue of $u_i$, with $0\leq \lambda_1\leq \lambda_2\leq \cdots \nearrow \infty$. As a consequence of these facts, each $f\in \mathcal{H}_0^1$ admits a unique orthogonal decomposition $f = \sum_i a_i u_i$, and $\|f\|^2_{\mathcal{H}_0^1} = \sum_i |a_i|^2 \lambda_i$.

Since $C_0^\infty(D) \subset \mathcal{H}_0^1$, each $f\in C_0^\infty(D)$ admits an orthogonal decomposition $f = \sum_i a_i u_i$ in $\mathcal{H}_0^1$. By analogy, one defines $\mathcal{H}_0^\alpha$ to be the closure of $C_0^\infty(D)$ with respect to the norm $\|f\|^2_{\mathcal{H}_0^\alpha} := \sum_i |a_i|^2 \lambda^\alpha_i$. The Sobolev space $\mathcal{H}^{-\alpha}$ is defined as the Hilbert dual of $\mathcal{H}_0^\alpha$, that is, $\mathcal{H}^{-\alpha}$ is the space of continuous linear functionals on $\mathcal{H}_0^\alpha$, endowed with the norm $\|h\|^2_{\mathcal{H}^{- \alpha} } := \sup_{f :  \|f\|_{\mathcal{H}^{\alpha}_0} \leq 1} |h(f)|$. One has that $L^2(D) \subset \mathcal{H}^{-\alpha}$. Also, the action of $h  = \sum_i a_iu_i\in L^2(D)$ on $f\in \mathcal{H}_0^\alpha$ is given by $h(f) = \int_D h(z) f(z)\, dz$, and moreover $\|h\|^2_{\mathcal{H}^{- \alpha} } = \sum_i \lambda^{-\alpha}_i |a_i|^2$.

We emphasize the following inclusions, that are used throughout the paper: for $\alpha>0$,
\begin{eqnarray}
&&C_0^\infty(D) \subset \mathcal{H}_0^\alpha \subset L^2(D) \subset \mathcal{H}^{-\alpha}.\label{e:inc1}
\end{eqnarray}

%\medskip

\begin{remark}{\rm In the previous discussion, we have implicitly used the fact that, if $f\in C_0^\infty(D)$, then
\begin{equation}\label{e:ww}
\sum |a_i|^2 \lambda_i^\alpha <\infty, \quad \forall \alpha>0,
\end{equation}
where $a_i = \langle f, u_i \rangle_{L^2(D)}$ as before. In order to prove \eqref{e:ww}, one can assume without loss of generality that $\alpha\geq 1$ is an integer. For such an $\alpha$ and for every $f\in C_0^\infty(D)$, one has that $(-\Delta)^\alpha f\in C_0^\infty(D)$, and consequently $(-\Delta)^\alpha f = \sum_i \langle (-\Delta)^\alpha f , u_i\rangle_{L^2(D)} u_i$, where the series converges in $L^2(D)$. Integration by parts yields moreover that $\langle (-\Delta)^\alpha f , u_i\rangle_{L^2(D)} = \langle f , (-\Delta)^\alpha u_i\rangle_{L^2(D)} = \lambda_i^\alpha \langle f , u_i\rangle_{L^2(D)}$, from which we deduce that 
\be
 \sum |a_i|^2 \lambda_i^\alpha = \langle (-\Delta)^\alpha f, f\rangle_{L^2(D)} \leq \|(-\Delta)^\alpha f\|_{L^2(D)} \,\cdot  \|f\|_{L^2(D)}<\infty,
\ee
as claimed.}
\end{remark}

%{\red
%Federico: Please check! If this is not clear, we can always take $\mathcal{H}_0^\alpha$ as our space of test functions. Then the problem is removed.
%
%The only nontrivial inclusion is the first one. For integer values of $\alpha$, this follows from the standard definition of the space $\mathcal{H}_0^\alpha$ as the closure
%of $C^{\infty}_0(D)$ under the norm associated with the Sobolev space $\mathcal{H}^\alpha = W^{\alpha,2}(D)$. The fact that the inclusion also holds for noninteger
%values of $\alpha$ follows from the observation that $\mathcal{H}_0^{\alpha_2}$ is compactly embedded in $\mathcal{H}_0^{\alpha_1}$ for $\alpha_2>\alpha_1$,
%which follows from the Rellich Theorem.
%}

\subsection{A useful lemma}

The following lemma is used several times in the paper.

\begin{lemma} \label{finitenessunitdisk}
Let $\mathbb D$ be the unit disk of $\mathbb C$, and 
$a,b ,c \in [0,1)$; then
 \begin{eqnarray}\label{e:bonus}
 \iint_{\mathbb D^2} \frac{1}{|z-t|^a} \frac{1}{(1-|z|)^b} \frac{1}{(1-|t|)^c} \ dz dt < \infty \, .
 \end{eqnarray} 
\end{lemma}

\begin{proof}
As the integrand is nonnegative, the double integral is well-defined and Tonelli's theorem implies that it is actually equal to an iterated integral, where the order of integration can be interchanged. Writing $z = r e^{i \theta}$ and $t = s e^{i \phi}$, the double integral becomes
\be
\int_0^1 \int_0^1 \int_0^{2\pi} \int_0^{2\pi} \frac{rs d\theta s\phi dr ds}{(r^2+s^2-2rs\cos(\theta-\phi))^{a/2}(1-r)^b(1-s)^c}.
\ee
Since $r^2+s^2-2rs\cos(\theta-\phi) = (r-s)^2+2rs(1-\cos(\theta-\phi)) \geq rs(1-\cos(\theta-\phi))$, this is bounded above by
\begin{align}
& \int_0^1 \int_0^1 \int_0^{2\pi} \int_0^{2\pi} r^{1-a/2} s^{1-a/2} (1-r)^{-b} (1-s)^{-c} (1-\cos(\theta-\phi))^{-a/2} d\theta d\phi dr ds \\
& = B(2-a/2,1-b) B(2-a/2,1-c) \int_0^{2\pi} \int_0^{2\pi} (1-\cos(\theta-\phi))^{-a/2} d\theta d\phi,
\end{align}
where $B$ is the beta function. Letting $u=\theta-\phi$, the last integral is bounded above by
\be
\int_0^{2\pi} \int_{-\phi}^{2\pi-\phi} (1-\cos u)^{-a/2} du d\phi \leq 2\pi \int_0^{2\pi} (1-\cos u)^{-a/2} du,
\ee
where the last integral is easily seen to be finite by expanding the cosine in power series.
\end{proof}

\subsection{Proofs from Section \ref{ss:poissononepoint}}\label{ss:proofponepoint}

The following estimate is needed in the proof of Lemma \ref{MassiveOnePointFunctionEstimate}.
\begin{lemma}\label{concentration_Brownian} For every fixed $b>0$,
\begin{equation} \nonumber
\sup_{0<a<b}\mu^{loop}(\gamma: z\in \overline{\gamma}, a < d(\gamma)\leq b, t_\gamma>T)\leq \frac{b^2}{2T} \, .
\end{equation}
\end{lemma}
\begin{proof}
By the scale invariance of Brownian motion and the definition of the Brownian loop soup (see Section \ref{ss:loops}), for each $0<a<b$, we have that
\begin{align*}
& \mu^{loop}(\gamma: z\in \overline{\gamma}, a < d(\gamma)\leq b, t_\gamma>T) \\
&\hskip 2cm = \mu^{loop}(\gamma: z\in \overline{\gamma}, \frac{a}{b} < d(\gamma)\leq 1, t_\gamma>\frac{T}{b^2})\\
&\hskip 2cm  \leq  \int\limits_{\frac{T}{b^2}}^{\infty}\int\limits_{B(w,1)}\frac{1}{2\pi t^2}\mu^{br}_{z,t}(\gamma:\frac{a}{b} < d(\gamma)\leq 1)dA(w)dt\\
&\hskip 2cm \leq \int\limits_{\frac{T}{b^2}}^{\infty}\int\limits_{B(z,1)}\frac{1}{2\pi t^2}dA(w)dt = \int\limits_{\frac{T}{b^2}}^{\infty}\frac{1}{2t^2}dt = \frac{b^2}{2T} \, ,
\end{align*}
as desired.
\end{proof}

\medskip

\begin{proof}[Proof of Lemma \ref{MassiveOnePointFunctionEstimate}]

The first inequality in \eqref{MassiveAlphaEstimate1} is a consequence of the definition. Then, using the obvious inclusion
\begin{equation} \notag
\{ \gamma: z \in \bar\gamma, \delta \leq \diam(\gamma) \leq R \} \subseteq
\bigcup_{n=0}^{\lfloor \log_2 \frac{R}{\delta} \rfloor} \Big\{ z \in \bar\gamma, \frac{R}{2^{n+1}} < \diam(\gamma) \leq \frac{R}{2^n} \Big\},
\end{equation}
we have
\begin{align*}
\alpha^{loop}_{\delta,R}(z) - \alpha^{m}_{\delta, R}(z)
&\leq \sum_{n=0}^{\lfloor \log_2{ \frac{R}{\delta}} \rfloor} 
\Big(\alpha^{loop}_{\frac{R}{2^{n+1}},\frac{R}{2^{n}}}(z)
-\alpha^m_{\frac{R}{2^{n+1}},\frac{R}{2^{n}}}(z) \Big)\\
& = \sum_{n=0}^{\lfloor \log_2{ \frac{R}{\delta}} \rfloor} 
(\mu^{loop}-\mu^m)\Big(\gamma: z \in \overline \gamma, 
\frac{R}{2^{n+1}} < d(\gamma) \leq \frac{R}{2^{n}},
t_{\gamma}\leq \frac{R}{2^{n}} \Big) \\
& \quad+ \sum_{n=0}^{\lfloor \log_2{ \frac{R}{\delta}} \rfloor} 
(\mu^{loop}-\mu^m)\Big(\gamma: z \in \overline \gamma, 
\frac{R}{2^{n+1}} < d(\gamma) \leq \frac{R}{2^{n}},
t_{\gamma}> \frac{R}{2^{n}} \Big)\\
&\leq   \sum_{n=0}^{\infty}  
\Big(1-\exp\big(-\frac{\overline m^2 R}{2^n}\big)\Big)
\mu^{loop}\Big(\gamma: z \in \overline \gamma, 
\frac{R}{2^{n+1}} < d(\gamma) \leq \frac{R}{2^{n}}\Big)
 \\
& \quad +\sum_{n=0}^{\infty} 
\mu^{loop}\Big(\gamma: z \in \overline \gamma, 
\frac{R}{2^{n+1}} < d(\gamma) \leq \frac{R}{2^{n}},
t_{\gamma}> \frac{R}{2^{n}} \Big)\\
&\leq \frac15 \log{2} \sum_{n=0}^{\infty} \Big(1-\exp\big(-\frac{\overline m^2 R}{2^n}\big) \Big) + \sum_{n=0}^{\infty}   \frac{R}{2^{n}}
\leq 2 R \Big( \overline m^2 \frac15 \log{2}  + 1 \Big),
\end{align*}
where we used Lemma \ref{concentration_Brownian} with $a=\frac{R}{2^{n+1}}$ and $ b=T=\frac{R}{2^{n}}$, and, in the inequality before the last one, 
the fact that $ \alpha^{loop}_{\delta_1,\delta_2}(z)  = \frac15 \log{ \frac{\delta_2}{\delta_1}}$. As the bound above does not depend on $\delta$,
and since $\alpha^{loop}_{\delta,R}(z) - \alpha^{m}_{\delta, R}(z) \geq 0$ is monotone in $\delta$, we have that
\begin{equation}
%\alpha^{loop}_{ R}(z) -\alpha^{m}_{ R}(z) = 
\lim\limits_{\delta \to 0} \big(\alpha^{loop}_{\delta,R}(z) - \alpha^{m}_{\delta, R}(z)\big) \leq 2 R \Big( \frac{\bar m}{5} \log 2 + 1 \Big) < \infty .
\end{equation}
Finally, as $D$ is bounded,
\begin{eqnarray}
\lim\limits_{\delta \to 0} \big(\alpha^{loop}_{\delta,D}(z) - \alpha^{m}_{\delta,D}(z)\big) & = & \lim\limits_{\delta \to 0} \big(\alpha^{loop}_{\delta,d_z}(z) - \alpha^{m}_{\delta, d_z}(z)\big) + \alpha^{loop}_{d_z, D}(z) - \alpha^{m}_{ d_z, D}(z) \notag \\
& \leq & c_2( \overline m^2, D) < \infty.
\end{eqnarray}
As a consequence,
\be
\lim\limits_{\delta \to 0} \big(\alpha^{loop}_{\delta,D}(z) - \alpha^{m}_{\delta,D}(z)\big) = \lim\limits_{\delta \to 0} \hat\alpha_{\delta,D}(z) = \hat\alpha_{0,D}(z),
\ee
which concludes the proof. \end{proof}

\medskip

\begin{proof}[Proof of Theorem \ref{t:2p1}]
Following the proof of Lemma A.1 in \cite{camia2016conformal}, scale invariance implies that, for $z \in D$ and $0<\delta<R$, we have
\begin{eqnarray}
\lefteqn{\mu^{*}(\gamma: z \in \bar \gamma, \delta \leq \diam(\gamma) \leq R)\label{e:soul}
- \mu^{*}(\gamma: z \in \bar \gamma,\gamma \not\subset B_{z, \delta}, \gamma \subset B_{z,R})} \\
& = & \mu^{*}(\gamma: z \in \bar \gamma, \diam(\gamma) \geq \delta, \gamma \subset B_{z,\delta})
- \mu^{*}(\gamma: z \in \bar \gamma,\diam(\gamma) \geq R, \gamma \subset B_{z,R}) =0,\notag
\end{eqnarray}
so that the first two terms are equal. Next, setting $\rho=|z_1-z_2|$, we have
\begin{eqnarray}
&&\alpha^*_D(z_1, z_2) - \alpha^*_{|z_1-z_2|,d_{z_1}}(z_1)
= \alpha^*_D(z_1, z_2) - \alpha^*_{\rho,d_{z_1}}(z_1) \label{last-line} \\
&& \quad =  \mu^{*}(\gamma: z_1 \in \bar \gamma,\gamma \not\subset B_{z_1, \rho}, \gamma \subset B_{z_1,d_{z_1}})
+ \mu^{*}(\gamma: z_1 \in \bar \gamma,\gamma \not\subset B_{z_1,d_{z_1}}, \gamma \subset D) \notag \\
&& \hskip 1cm  -\mu^{*}(\gamma: z_1 \in \bar \gamma,\gamma \not\subset B_{z_1, \rho}, 
\gamma \subset D, z_2 \notin \overline \gamma) 
-\alpha^*_{\rho,d_{z_1}}(z_1) \notag \\
&& \quad = \mu^{*}(\gamma: z_1 \in \bar \gamma,\gamma \not\subset B_{z_1,d_{z_1}}, \gamma \subset D)
-\mu^{*}(\gamma: z_1 \in \bar \gamma,\gamma \not\subset B_{z_1, \rho}, 
\gamma \subset D, z_2 \notin \overline \gamma) \notag
%&&  \quad\longrightarrow
%\mu^{*}(\gamma: z_1 \in \bar \gamma,\gamma \not\subset B_{z_1,d_{z_1}}, \gamma \subset D)
%-\alpha_{\neg B_{{\bf 0},1}}({\bf 0}|{\bf 1})=:\Psi^*(z_1,D),\label{e:dpsild}
\end{eqnarray}
where %the limit in the last expression is taken as ${z_2 \to z_1} $, and 
the last equality follows from a direct application of \eqref{e:soul}.
Letting $B_{z_1}(D)$ denote the largest open disc centered at $z_1$ and contained in $D$, with radius $r_{z_1}(D)$, and using the translation,
rotation and scale invariance of $\mu^*$, the last term of the last line of \eqref{last-line} %the previous equation
can be written as
\begin{eqnarray*}
\mu^{*}(\gamma: z_1 \in \bar \gamma,\gamma \not\subset B_{z_1, \rho}, \gamma \subset D, z_2 \notin \overline \gamma) & = &
\mu^{*}(\gamma: z_1 \in \bar \gamma,\gamma \not\subset B_{z_1, \rho}, \gamma \subset B_{z_1}(D), z_2 \notin \overline \gamma) \\
& + & \mu^{*}(\gamma: z_1 \in \bar \gamma,\gamma \not\subset B_{z_1}(D), \gamma \subset D, z_2 \notin \overline \gamma) \\
& = & \mu^{*}(\gamma: {\bf 0} \in \bar \gamma,\gamma \not\subset B_{{\bf 0}, 1}, \gamma \subset B_{{\bf 0}, r_{z_1}(D)/\rho},
{\bf 1} \notin \overline \gamma) \\
& + & \mu^{*}(\gamma: z_1 \in \bar \gamma,\gamma \not\subset B_{z_1}(D), \gamma \subset D, z_2 \notin \overline \gamma).
\end{eqnarray*}
To show the existence of the limit, note that
\begin{equation} \label{e:limit1}
\lim_{\rho \to 0} \mu^{*}(\gamma: {\bf 0} \in \bar \gamma,\gamma \not\subset B_{{\bf 0}, 1}, \gamma \subset B_{{\bf 0}, r_{z_1}(D)/\rho},
{\bf 1} \notin \overline \gamma) = \alpha_{\neg B_{{\bf 0},1}}({\bf 0}|{\bf 1}),
\end{equation}
where $\alpha_{\neg B_{{\bf 0},1}}({\bf 0}|{\bf 1})<\infty$ because of the thinness of $\mu^*$. Using
again the translation, rotation and scale invariance of $\mu^*$,
\begin{eqnarray}
\mu^{*}(\gamma: z_1 \in \bar \gamma,\gamma \not\subset B_{z_1}(D), \gamma \subset D, z_2 \notin \overline \gamma) & \leq &
\mu^{*}(\gamma: z_1 \in \bar \gamma,\gamma \not\subset B_{z_1}(D), z_2 \notin \overline \gamma) \notag \\
& = & \mu^{*}(\gamma: {\bf 0} \in \bar \gamma,\gamma \not\subset B_{{\bf 0}, r_{z_1}(D)/\rho}, {\bf 1} \notin \overline \gamma) \notag \\
& \stackrel{\rho \to 0}{\longrightarrow} & 0 \label{e:limit2}
\end{eqnarray}
again by the thinness of $\mu^*$.

To prove the continuity of $\Psi^*$ in $z_1 \in D$, note that \eqref{e:limit1} and \eqref{e:limit2} imply that
\begin{eqnarray}
\Psi^*(z_1,D) & := & \lim_{z_2 \to z_1} (\alpha^*_D(z_1, z_2) - \alpha^*_{|z_1-z_2|,d_{z_1}}(z_1)) \nonumber \\
& = & \mu^{*}(\gamma: z_1 \in \bar \gamma,\gamma \not\subset B_{z_1,d_{z_1}}, \gamma \subset D) - \alpha_{\neg B_{{\bf 0},1}}({\bf 0}|{\bf 1}) \label{eq:Psi-alpha}
\end{eqnarray}
and that, for any $z_2 \in D$,
\begin{eqnarray*}
\{ \gamma: z_1 \in \bar \gamma,\gamma \not\subset B_{z_1,d_{z_1}}, \gamma \subset D \} & = & \{ \gamma: z_1 \in \bar \gamma, z_2 \in \bar\gamma, \gamma \not\subset B_{z_1,d_{z_1}}, \gamma \not\subset B_{z_2,d_{z_2}}, \gamma \subset D \} \\
& \cup & \{ \gamma: z_1 \in \bar \gamma, z_2 \not\in \bar\gamma, \gamma \not\subset B_{z_1,d_{z_1}}, \gamma \subset D \} \\
& \cup & \{ \gamma: z_1 \in \bar \gamma, z_2 \in \bar\gamma, \gamma \not\subset B_{z_1,d_{z_1}}, \gamma \subset B_{z_2,d_{z_2}} \}.
\end{eqnarray*}
Therefore,
\begin{eqnarray}
&& \Psi^*(z_1,D) - \Psi^*(z_2,D) \notag \\
&& = \mu^{*}(\gamma: z_1 \in \bar \gamma,\gamma \not\subset B_{z_1,d_{z_1}}, \gamma \subset D) - \mu^{*}(\gamma: z_2 \in \bar \gamma,\gamma \not\subset B_{z_2,d_{z_2}}, \gamma \subset D) \notag \\
&& = \mu^*(\gamma: z_1 \in \bar \gamma, z_2 \not\in \bar\gamma, \gamma \not\subset B_{z_1,d_{z_1}}, \gamma \subset D) + \mu^*(\gamma: z_1,z_2 \in \bar \gamma, \gamma \not\subset B_{z_1,d_{z_1}}, \gamma \subset B_{z_2,d_{z_2}}) \notag \\
&& - \mu^*(\gamma: z_2 \in \bar \gamma, z_1 \not\in \bar\gamma, \gamma \not\subset B_{z_2,d_{z_2}}, \gamma \subset D) - \mu^*(\gamma: z_1,z_2 \in \bar \gamma, \gamma \not\subset B_{z_2,d_{z_2}}, \gamma \subset B_{z_1,d_{z_1}}). \notag
\end{eqnarray}
The first and third terms on the right hand side of the last equality tend to zero as $z_2 \to z_1$ because of the thinness of $\mu^*$, as in \eqref{e:limit2}.
The second and fourth terms can be treated in the same way, and we focus only on the second. Using the elementary fact that
$d_{z_2} \leq d_{z_1} + |z_1-z_2|$ and either \eqref{alpha5} or \eqref{alpha8}, we have that
\begin{eqnarray}
&& \mu(\gamma: z_1,z_2 \in \bar \gamma, \gamma \not\subset B_{z_1,d_{z_1}}, \gamma \subset B_{z_2,d_{z_2}}) \notag \\
&& \quad \leq \mu(\gamma: z_1 \in \bar \gamma, \gamma \not\subset B_{z_1,d_{z_1}}, \gamma \subset B_{z_1,d_{z_1}+2|z_1-z_2|}) \notag \\
&& \quad \leq \pi \log\Big(1+2\frac{|z_1-z_2|}{d_{z_1}}\Big) \longrightarrow 0 \quad \text{ as } z_2 \to z_1.
\end{eqnarray}
This concludes the proof of the theorem.
\end{proof}

\medskip

\begin{proof}[Proof of Lemma \ref{l:brexit}]
We let $\rho = |z_1-z_2|$ and write
\begin{eqnarray}
&&\alpha^m_D(z_1, z_2) - \alpha^{loop}_{\rho,d_{z_1}}(z_1) = \notag \\
&& \quad \alpha^{m}_{D}(z_1,z_2) - \alpha^{loop}_{D}(z_1,z_2) + \alpha^{loop}_{D}(z_1,z_2) - \alpha^{loop}_{\rho,d_{z_1}}(z_1). \label{e:two-terms}
\end{eqnarray}
When $z_2$ tends to $z_1$, the third and fourth terms on the right hand side of \eqref{e:two-terms}, when combined, converge to $\Psi^{loop}(z_1,D)$
by \eqref{ThScaleInvTwoPoint}. To deal with the first two terms, we will use property \eqref{MassiveAlphaEstimate3} of the measure $\hat\mu$ defined in \eqref{e:muhat}.
Letting $A_n = \{ \gamma: z_1 \in \bar\gamma, \gamma \subset D, B_{z_1,1/n} \subset \bar\gamma \}$, we have
\begin{equation*}
\{ \gamma: z_1 \in \bar\gamma, \gamma \subset D \} = \cup_{n=1}^{\infty} A_n \cup \{ \gamma: z_1 \in \partial\bar\gamma, \gamma \subset D \},
\end{equation*}
where $\partial\bar\gamma$ denotes the boundary of the hull of $\gamma$.
Since $\hat\mu(\gamma: z_1 \in \partial\bar\gamma, \gamma \subset D)=0$ and $\{ A_n \}_{n=1}^{\infty}$ is an increasing sequence of sets,
$\hat\mu(\gamma: z_1 \in \bar\gamma, \gamma \subset D) = \lim_{n \to \infty} \hat\mu(A_n)$.

For any $D \ni z_2 \neq z_1$, let $N(z_1,z_2)$ denote the largest integer $n$ such that $z_2 \in B_{z_1,1/n}$; then
$ \{ \gamma: z_1,z_2 \in \bar\gamma, \gamma \subset D \} \supseteq A_{N(z_1,z_2)}$, which implies
\begin{equation*}
\liminf_{z_2 \to z_1} \hat\mu(\gamma: z_1,z_2 \in \bar\gamma, \gamma \subset D) \geq \lim_{z_2 \to z_1} \hat\mu(A_{N(z_1,z_2)})
= \hat\mu(\gamma: z_1 \in \bar\gamma, \gamma \subset D).
\end{equation*}
On the other hand, one trivially has that, for every $D \ni z_2 \neq z_1$,
$\hat\mu(\gamma: z_1,z_2 \in \bar\gamma, \gamma \subset D) \leq \hat\mu(\gamma: z_1 \in \bar\gamma, \gamma \subset D)$,
which implies
\begin{equation}
\limsup_{z_2 \to z_1} \hat\mu(\gamma: z_1,z_2 \in \bar\gamma, \gamma \subset D) \leq \hat\mu(\gamma: z_1 \in \bar\gamma, \gamma \subset D).
\end{equation}
Hence, we conclude that
\begin{eqnarray}
\lim_{z_2 \to z_1} \big( \alpha^{m}_{D}(z_1,z_2) - \alpha^{loop}_{D}(z_1,z_2) \big) & = & - \lim_{z_2 \to z_1} \hat\mu(\gamma: z_1,z_2 \in \bar\gamma, \gamma \subset D) \nonumber \\
& = & - \hat\mu(\gamma: z_1 \in \bar\gamma, \gamma \subset D).
\end{eqnarray}
and
\begin{equation}
\lim_{z_2 \to z_1} \big( \alpha^m_D(z_1, z_2) - \alpha^{loop}_{\rho,d_{z_1}}(z_1) \big)
= \Psi^{loop}(z_1,D) - \hat\mu(\gamma: z_1 \in \bar\gamma, \gamma \subset D).
\end{equation}

To prove the continuity of $\Psi^{m}(z,D)$ in $z$, since we already know that $\Psi^{loop}(z,D)$ is continuous in $z$, it suffices to prove continuity of
$\hat\mu(\gamma: z \in \bar\gamma, \gamma \subset D)$ in $z$.
To that end, let $\hat\alpha_{0,D}(z) := \hat\mu(\gamma: z \in \bar\gamma, \gamma \subset D)$ and
$\hat\alpha_{0,D}(z|w) := \hat\mu(\gamma: z \in \bar\gamma, w \notin \bar\gamma, \gamma \subset D)$,
and note that
\begin{eqnarray}
|\hat\alpha_{0,D}(z_1) - \hat\alpha_{0,D}(z_2)| & = & |\hat\alpha_{0,D}(z_1|z_2) - \hat\alpha_{0,D}(z_2|z_1)| \nonumber \\
& \leq & \hat\alpha_{0,D}(z_1|z_2) + \hat\alpha_{0,D}(z_2|z_1).
\end{eqnarray}
For any $n$ such that $B_{z_1,1/n} \subset D$, let $O_n = \{ \gamma: z_1 \in \bar\gamma, \gamma \subset B_{z_1,1/n} \}$.
Since $\{ O_n \}_{n=1}^{\infty}$ is a decreasing sequence of sets and
$\hat\mu (\cap_{n=1}^{\infty} O_n) = \hat\mu(\gamma: z_1 \in \bar\gamma, \diam(\gamma)=0)=0$,
$\lim_{n \to \infty} \hat\mu(O_n) = 0$.
For any $\varepsilon>0$, choosing first $n$ so large that $\hat\mu(O_n) < \varepsilon/2$ and then $z_2$ sufficiently close to $z_1$ so that, by thinness,
$\mu^{loop} \Big( {\bf 0} \in \bar\gamma, {\bf 1} \notin \bar\gamma, \gamma \not\subset B_{{\bf 0},\frac{1}{n|z_1-z_2|}} \Big) < \varepsilon/2$,
we have that
\begin{eqnarray} \nonumber
\hat\alpha_{0,D}(z_1|z_2) & = & \hat\mu(\gamma: z_1 \in \bar\gamma, z_2 \notin \bar\gamma, \gamma \not\subset B_{z_1,1/n}, \gamma \subset D)
+ \hat\mu(\gamma: z_1 \in \bar\gamma, z_2 \notin \bar\gamma, \gamma \subset B_{z_1,1/n}) \nonumber \\
& \leq & \mu^{loop} \Big( {\bf 0} \in \bar\gamma, {\bf 1} \notin \bar\gamma, \gamma \not\subset B_{{\bf 0},\frac{1}{n|z_1-z_2|}} \Big) + \hat\mu(O_n) < \varepsilon.
\end{eqnarray}
The term $\hat\alpha_{D}(z_2|z_1)$ can be treated in exactly the same way, which concludes the proof.
\end{proof}

\medskip

\begin{proof}[Proof of Lemma \ref{l:cont}]
First of all, we observe that, in view of \eqref{ThScaleInvTwoPoint} and \eqref{ThScaleInvTwoPointMassive}, using \eqref{alpha5}, we have that,
for $* = loop, m$,
\begin{eqnarray}
\lim_{z_2 \to z_1} g^{*}(z_1,z_2) & = & \lim\limits_{z_2 \to z_1} \Big(\alpha_D^*(z_1,z_2) - \frac15 \log^+{\frac{1}{|z_1-z_2|}}\Big) \nonumber \\
& = & \frac15 \log{d_{z_1}} + \Psi^*(z_1, D).
\end{eqnarray}
For the disk model, using \eqref{alpha8}, we have 
\begin{eqnarray}
\lim_{z_2 \to z_1} g^{disk}(z_1,z_2) & = & \lim\limits_{z_2 \to z_1} \Big(\alpha_D^{disk}(z_1,z_2)- \pi \log^+{\frac{1}{|z_1-z_2|}}\Big) \nonumber \\
& = & \pi \log{d_{z_1}} + \Psi^{disk}(z_1, D).
\end{eqnarray}

The continuity of $\alpha_D^{*}(z_1,z_2)$ in its two arguments, for $z_1 \neq z_2$, can be inferred from the following upper bound:
\begin{eqnarray*}
&&| \alpha_D^{*}(z_1, z_2) -  \alpha_D^{*}(w_1, w_2) | \\ 
&& \quad = | \mu^{*}(z_1,z_2,w_1 \in \bar\gamma, w_2 \notin \bar\gamma, \gamma \subset D)
+ \mu^{*}(z_1,z_2,w_2 \in \bar\gamma, w_1 \notin \bar\gamma, \gamma \subset D) \\
&& \quad + \mu^{*}(z_1,z_2 \in \bar\gamma, w_1,w_2 \notin \bar\gamma, \gamma \subset D)
- \mu^{*}(z_1,w_1,w_2 \in \bar\gamma, z_2 \notin \bar\gamma, \gamma \subset D) \\
&& \quad - \mu^{*}(z_2,w_1,w_2 \in \bar\gamma, z_1 \notin \bar\gamma, \gamma \subset D)
- \mu^{*}(w_1,w_2 \in \bar\gamma, z_1,z_2 \notin \bar\gamma, \gamma \subset D) | \\
&& \quad \leq \mu^{*}(z_1,z_2,w_1 \in \bar\gamma, w_2 \notin \bar\gamma, \gamma \subset D)
+ \mu^{*}(z_1,z_2,w_2 \in \bar\gamma, w_1 \notin \bar\gamma, \gamma \subset D) \\
&& \quad + \mu^{*}(z_1,z_2 \in \bar\gamma, w_1,w_2 \notin \bar\gamma, \gamma \subset D)
+ \mu^{*}(z_1,w_1,w_2 \in \bar\gamma, z_2 \notin \bar\gamma, \gamma \subset D) \\
&& \quad + \mu^{*}(z_2,w_1,w_2 \in \bar\gamma, z_1 \notin \bar\gamma, \gamma \subset D)
+ \mu^{*}(w_1,w_2 \in \bar\gamma, z_1,z_2 \notin \bar\gamma, \gamma \subset D).
\end{eqnarray*}
All terms in the upper bound tend to zero as $w_1 \to z_1$ and $w_2 \to z_2$. We only show that this is the case for the first term, since all the others
can be treated analogously. Using the translation, rotation and scale invariance, of $\mu^{loop}$, we have that
\begin{eqnarray}
\mu^{loop}(z_1,z_2,w_1 \in \bar\gamma, w_2 \notin \bar\gamma, \gamma \subset D) & \leq & \mu^{loop}(z_1,z_2 \in \bar\gamma, w_2 \notin \bar\gamma, \gamma \subset D) \\
& \leq & \mu^{loop}(z_2 \in \bar\gamma, w_2 \notin \bar\gamma, \gamma \not\subset B_{z_2,|z_2-z_1|}) \nonumber \\
& = & \mu^{loop} \Big( {\bf 0} \in \bar\gamma, {\bf 1} \notin \bar\gamma, \gamma \not\subset B_{{\bf 0},\frac{|z_2-z_1|}{|z_2-w_2|}} \Big) \stackrel{w_2 \to z_2}{\longrightarrow} 0, \nonumber
\end{eqnarray}
where the limit follows from the thinness of $\mu^{loop}$.
For $\mu^{m}$, it suffices to note that
\begin{equation}
\mu^{m}(z_1,z_2,w_1 \in \bar\gamma, w_2 \notin \bar\gamma, \gamma \subset D) \leq \mu^{loop}(z_1,z_2,w_1 \in \bar\gamma, w_2 \notin \bar\gamma, \gamma \subset D).
\end{equation}
For $\mu^{disk}$, we can use the same argument as for $\mu^{loop}$, given that $\mu^{disk}$ is thin, as show in Lemma \ref{lemma:disk-thin}.
%
%provided that we can prove that $\mu^{disk}$ is thin. The thinness of $\mu^{disk}$ is a consequence
%of the following calculation, which follows from elementary geometric and trigonometric considerations:
%\begin{eqnarray*}
%&& \lim_{R \to \infty} \alpha^{disk}_{\mathbb R^2} \Big( {\bf 0} \in \bar\gamma, {\bf 1} \notin \bar\gamma, \diam(\gamma)>R \Big) \\
%&& \quad = \lim_{R \to \infty} \int_{R}^{\infty} \Big( 2 r^2 \sin^{-1}(1/2r) + \sqrt{r^2-1/4} \Big) \frac{dr}{r^3} = 0.
%\end{eqnarray*}

It remains to check that $g^*(z,w) \to g^*(z_1,z_1)$ as $z,w \to z_1$. For $*=loop$, using \eqref{alpha5}, we have that
\begin{eqnarray}
&& \Big| g^{loop}(z_1,z_1) - g^{loop}(z,w) \Big| = \Big| \frac15 \log d_{z_1} + \Psi^{loop}(z_1,D) - \alpha^{loop}_D(z,w) + \frac15 \log^+\frac{1}{|z-w|} \Big| \nonumber \\
&& = \Big| \frac15 \log d_{z_1} + \Psi^{loop}(z_1,D) - \big(\alpha^{loop}_D(z,w) - \alpha^{loop}_{|z-w|,d_z}(z) \big) - \alpha^{loop}_{|z-w|,d_z}(z) + \frac15 \log^+\frac{1}{|z-w|} \Big| \nonumber \\
&& \leq \Big| \frac15 \log d_{z_1} - \frac15 \log d_z \Big| + \Big| \Psi^{loop}(z_1,D) - \big(\alpha^{loop}_D(z,w) - \alpha^{loop}_{|z-w|,d_z}(z) \big) \Big| \to 0
\end{eqnarray}
as $z,w \to z_1$ by the continuity of $d_z$ and Theorem \ref{t:2p1}.
For $*=disk$ one can do exactly the same calculation using \eqref{alpha8} instead of \eqref{alpha5}.
For $*=m$, using again \eqref{alpha5}, one has that
\begin{eqnarray*}
&& \Big| g^{m}(z_1,z_1) - g^{m}(z,w) \Big| = \Big| \frac15 \log d_{z_1} + \Psi^{m}(z_1,D) - \alpha^{m}_D(z,w) + \frac15 \log^+\frac{1}{|z-w|} \Big| \\
&& = \Big| \frac15 \log d_{z_1} + \Psi^{m}(z_1,D) - \big(\alpha^{m}_D(z,w) - \alpha^{loop}_{|z-w|,d_z}(z) \big) - \alpha^{loop}_{|z-w|,d_z}(z) + \frac15 \log^+\frac{1}{|z-w|} \Big| \\
&& \leq \Big| \frac15 \log d_{z_1} - \frac15 \log d_z \Big| + \Big| \Psi^{m}(z_1,D) - \big(\alpha^{m}_D(z,w) - \alpha^{loop}_{|z-w|,d_z}(z) \big) \Big| \to 0
\end{eqnarray*}
as $z,w \to z_1$ by the continuity of $d_z$ and Lemma \ref{l:brexit}. With this, the proof of the lemma is concluded.
\end{proof}

%The analogous results holds for a general domain $D$:
%\begin{corollary} \label{finiteness}
%Let $D$ be a bounded, simply connected domain of $\mathbb C^2$ with a
%$C^1$ boundary; for each $z \in D$, let $d_z= dist(z, \delta D)$; let $\varphi$
%be a smooth test function; 
%and let
%$a,b ,c \in (0,1)$. Then
% \begin{eqnarray*}
% \iint_{\mathbb D^2} \frac{1}{(\|z-t\|)^a}
%  \frac{1}{d_z^b}  \frac{1}{m_t^c} \varphi(z) \varphi(t) \ dz dt < \infty
% \end{eqnarray*} 
%
%
%\end{corollary}
%\begin{proof}
%{\red I did not write this proof too carefully...}
%Using the conformal equivalence $f: \mathbb D \to D$ as in \cite{Spin_loops_berg_camia_lis}, proof of Proposition 5.3, 
%we have 
%\begin{eqnarray*}
% && \iint_{\mathbb D^2} \frac{1}{(\|z-t\|)^a}
%  \frac{1}{m_{f(z)}^b}  \frac{1}{m_{f(t)}^c} \varphi(f(z)) \varphi(f(t)) 
%  |f'(z)|^{-a-b} |f'(t)|^{-a-c}\ dz dt \\
%  && \quad \leq 4^{2a+b+c} \|f'\|^{2a+b+c}_{L^{\infty}(\mathbb D)}
%  \|\varphi\|^{2}_{L^{\infty}(\mathbb D)}
%   \iint_{\mathbb D^2}  \frac{1}{(\|z-t\|)^a}
%  \frac{1}{d_z^b}  \frac{1}{m_t^c}  \ dz dt 
%  \\
%  && \quad \leq 4^{2a+b+c} \|f'\|^{2a+b+c}_{L^{\infty}(\mathbb D)}
%  \|\varphi\|^{2}_{L^{\infty}(\mathbb D)}
%   \iint_{\mathbb D^2} \frac{1}{(\|z-t\|)^a}
%  \frac{1}{(1-\|z\|)^b}  \frac{1}{(1-\|t\|)^c} \ dz dt  < \infty
% \end{eqnarray*} 
% by Lemma \ref{finitenessunitdisk}.
%\end{proof}

\subsection{One-point functions} \label{sec:one-point-functions}

\begin{lemma} \label{lemma:one-point-functions}
Let $D$ be a bounded, simply-connected domain and, for $z \in D$, let $d_z := \text{dist}(z,\partial D)$. Then, we have the following relations, where $*=loop$ or $disk$,
\begin{eqnarray}
\big\langle V^*_{\lambda,\beta}(z) \big\rangle & = & d_{z}^{-2\Delta^*_{\lambda,\beta}} e^{-\lambda \alpha^*_{d_{z},D}(z) (1-\cos\beta)} \label{eq:one_point_function_loop} \\
\langle V^{m}_{\lambda,\beta}(z) \rangle & = & d_{z}^{-2\Delta^{loop}_{\lambda,\beta}} e^{-\lambda (1 - \cos\beta) \alpha^m_{d_z,D}(z)} e^{\lambda(1-\cos\beta)\hat\mu^{loop}_{D}(\gamma: z \in \bar\gamma, \diam(\gamma) \leq d_z)} \label{eq:massive_one_point_function_loop} \\
\langle W^{*}_{\xi}(z) \rangle & = & d_{z}^{-2\Delta^{*}_{\xi}} e^{-\frac{\xi^2}{2} \alpha^*_{d,D}(z)} \label{eq:one_point_function_gauss} \\
\langle W^{m}_{\xi}(z) \rangle & = & d_{z}^{-2\Delta^{loop}_{\xi}} e^{-\frac{\xi^2}{2} \alpha^m_{d_z,D}(z)} e^{\frac{\xi^2}{2} \hat\mu^{loop}_{D}(\gamma: z \in \bar\gamma, \diam(\gamma) \leq d_z)}. \label{eq:massive_one_point_function_gauss}
\end{eqnarray}
\end{lemma}

\begin{proof}
In the notation of this paper, Equation~(4.2) of \cite{camia2016conformal}, when restricted to the case $n=1$, gives
\be \label{eq:CGKbis}
\big\langle V^{\delta}_{\lambda,\beta}(z) \big\rangle = \big\langle e^{i \beta N_{\lambda}(h^{\delta}_z)} \big\rangle = e^{-\lambda (1-\cos\beta) \alpha^{loop}_{\delta,D}(z)}.
\ee
Since the proof of Equation~(4.2) of \cite{camia2016conformal} uses only the Poissonian nature of the field $V^{loop}_{\lambda,\beta}$, analogs of Equation~\eqref{eq:CGKbis}
are also valid for $V^{disk}_{\lambda,\beta}$ and $V^{m}_{\lambda,\beta}$. Using this observation and \eqref{alpha5} and \eqref{alpha8}, we have that, for $*=loop, disk$,
\begin{eqnarray}
\big\langle V^*_{\lambda,\beta}(z) \big\rangle & = & \lim_{\delta \to 0} \delta^{-2\Delta^*_{\lambda,\beta}} \big\langle e^{i \beta N_{\lambda}(h^{\delta}_z)} \big\rangle^*_D \\
& = & \lim_{\delta \to 0} \delta^{-2\Delta^*_{\lambda,\beta}} e^{-\lambda (1-\cos\beta) \big( \alpha^*_{\delta,d_{z}}(z) + \alpha^*_{d_{z},D}(z) \big)} \\
& = & d_{z}^{-2\Delta^*_{\lambda,\beta}} e^{-\lambda (1-\cos\beta) \alpha^*_{d_{z},D}(z)},
%& \leq & d_{z,D}^{-2\Delta^*_{\lambda,\beta}}, \label{eq:integrability_loop_disk}
\end{eqnarray}
as desired. For $*=m$, using the same argument as above as well as \eqref{MassiveAlphaEstimate2} and the fact that $\Delta^m_{\lambda,\beta}=\Delta^{loop}_{\lambda,\beta}$, we obtain
\begin{eqnarray}
\langle V^{m}_{\lambda,\beta}(z) \rangle
& = & e^{-\lambda (1 - \cos\beta) \alpha^m_{d_z,\D}(z)} \lim_{\delta \to 0} \frac{e^{-\lambda (1 - \cos\beta) \alpha^{m}_{\delta,d_z}(z)}}{\delta^{-2\Delta^{m}_{\lambda,\beta}}} \\
& = & e^{-\lambda (1 - \cos\beta) \alpha^m_{d_z,\D}(z)} d_{z}^{-2\Delta^{loop}_{\lambda,\beta}} e^{\lambda(1-\cos\beta)\hat\mu^{loop}_{\D}(\gamma: z \in \bar\gamma, \diam(\gamma) \leq d_z)}.
\end{eqnarray}

For the Gaussian case, the definition of $h^{\delta}_z$ (see Section \ref{ss:GaussianLayering}) and, for instance, \cite[Example 5.3.6 (i)]{Gio_murad_taqqu}) immediately imply that
\be
\big\langle W^{\delta}_{\xi}(z) \big\rangle^*_D = e^{-\frac{\xi^2}{2} \int_A (h^{\delta}_z(x))^2 \nu(dx)} = e^{-\frac{\xi^2}{2} \alpha^*_{\delta,D}(z)},
\ee
where $*$ can be taken to be equal to $loop, disk$ or $m$. Hence, the same analysis as above leads to the desired results.
\end{proof}

\begin{corollary} \label{cor:limits}
For $*=loop,disk,m$, as $\lambda \to \infty$ and $\beta \to 0$ with $\lambda\beta^2 \to \xi^2$, we have that, for every $z \in D$,
\be \label{limit1}
\big\langle V^*_{\lambda,\beta}(z) \big\rangle \to \langle W^{*}_{\xi}(z) \rangle.
\ee
\end{corollary}

\begin{proof}
The corollary follows immediately from Lemma~\ref{lemma:one-point-functions}, using the fact that $\Delta^{loop}_{\lambda,\beta} \to \Delta^{loop}_{\xi}$
and $\lambda(1 - \cos\beta) \to \xi^2/2$ as $\lambda \to \infty, \beta \to 0$ with $\lambda\beta^2 \to \xi^2$.
\end{proof}

\subsection{Existence and conformal covariance of layering fields} \label{sec:conf-inv}

In this section, we prove the existence and the conformal covariance of the Poisson and Gaussian layering fields.
This can be seen as an extension of a result of \cite{camia2016conformal}, where existence and conformal covariance
are proved for the $n$-point functions of Poisson layering fields. We start with the existence of Poisson layering fields
in bounded domains with a $C^1$ boundary.

\begin{theorem} \label{ExistenceLayeringFieldBoundedDomains}
If $D$ is a bounded simply connected domain with a $C^1$ boundary, then the conlusions of Theorem \ref{ExistenceLayeringField} hold.
\end{theorem}

\begin{proof}
By an application of Theorem \ref{convergence}, it suffices to check conditions \eqref{two_point_function_exists}, \eqref{two_point_function}, and \eqref{2.3}.
This can be done using elementary properties of Poisson point processes and arguments analogous to those in the proofs of Theorem~4.1
of \cite{camia2016conformal} and Proposition~5.3 of \cite{Spin_loops_berg_camia_lis}. (Note, however, that the notation in those papers is different.
For instance, what we would denote $\alpha^{loop}_{\delta,D}(z|w)$ in this paper would be denoted $\alpha_{\delta,D}(z)$ in \cite{camia2016conformal}.)
Using the definition of $V_{\lambda, \beta}^{\delta}$, we have that
\begin{eqnarray}
&& 0 \leq \big\langle V_{\lambda, \beta}^{\delta}(z) \overline{V_{\lambda, \beta}^{\delta'}(w)} \big\rangle
= \big\langle e^{i \beta N_{\lambda}(h^{\delta}_z) - i \beta N_{\lambda}(h^{\delta'}_w)} \big\rangle \\
&& \quad = e^{-\lambda \alpha^{*}_{\delta,D}(z|w)(1-\cos\beta)} e^{-\lambda \alpha^{*}_{\delta',D}(w|z)(1-\cos\beta)}
e^{-\lambda \alpha^{*}_{\delta',\delta}(z,w)(1-\cos\beta)} \\
&& \quad \leq e^{-(\lambda \alpha^{*}_{\delta,D}(z|w)(1-\cos\beta) + \lambda \alpha^{*}_{\delta',D}(w|z)(1-\cos\beta))},
\end{eqnarray}
which verifies condition \eqref{two_point_function} with $\tau^{\delta}_{z,w} = \lambda \alpha^{*}_{\delta,D}(z|w)(1-\cos\beta)$.
Moreover, if $\delta' \leq \delta < \min(d_{z,w},d_{w,z})$, noting that $\alpha^{*}_{\delta',\delta}(z,w)=0$ when $\delta<|z-w|$, we can write
\begin{eqnarray*}
&& \big\langle V_{\lambda, \beta}^{\delta}(z) \overline{V_{\lambda, \beta}^{\delta'}(w)} \big\rangle
= e^{-\lambda \alpha^{*}_{\delta,D}(z|w)(1-\cos\beta)} e^{-\lambda \alpha^{*}_{\delta',D}(w|z)(1-\cos\beta)}
e^{-\lambda \alpha^{*}_{\delta',\delta}(z,w)(1-\cos\beta)} \\
&& \quad = e^{-\lambda \alpha^{*}_{\delta,d_{z,w}}(z)(1-\cos\beta)} e^{-\lambda \alpha^{*}_{d_{z,w},D}(z|w)(1-\cos\beta)} 
e^{-\lambda \alpha^{*}_{\delta',d_{w,z}}(w)(1-\cos\beta)} e^{-\lambda \alpha^{*}_{d_{w,z},D}(w|z)(1-\cos\beta)}.
\end{eqnarray*}

For $*=loop, disk$, conditions \eqref{two_point_function_exists} and \eqref{2.3} now follow easily from \eqref{alpha5} and \eqref{alpha8}. For $*=loop$, \eqref{alpha5} gives
\begin{equation} \label{eq:limit-loop} 
\lim_{\delta \to 0} \frac{e^{- \lambda \alpha^{loop}_{\delta,d_{z,w}}(z)(1-\cos\beta) }}{\delta^{2\Delta^{loop}_{\lambda,\beta}}} = \Big( \frac{1}{d_{z,w}} \Big)^{\frac{\lambda}{5}(1-\cos\beta)},
\end{equation}
which shows that
\be \notag
\lim_{\delta,\delta' \to 0} (\delta\delta')^{-2\Delta^{loop}_{\lambda,\beta}} \big\langle V_{\lambda, \beta}^{\delta}(z) \overline{V_{\lambda, \beta}^{\delta'}(w)} \big\rangle
= \Big( \frac{1}{d_{z,w}d_{w,z}} \Big)^{\frac{\lambda}{5}(1-\cos\beta)} e^{-\lambda (1-\cos\beta) \big(\alpha^{loop}_{d_{z,w},D}(z|w) + \alpha^{loop}_{d_{w,z},D}(w|z)\big)},
\ee
and verifies \eqref{two_point_function_exists}, and, for $\delta<d_{z,w}$,
\be \notag
e^{- \lambda \alpha^{loop}_{\delta,D}(z|w)(1-\cos\beta)} = e^{-\lambda \alpha^{loop}_{\delta,d_{z,w}}(z)(1-\cos\beta)} e^{-\lambda \alpha^{loop}_{d_{z,w},D}(z|w)(1-\cos\beta)} \leq \Big( \frac{d_{z,w}}{\delta} \Big)^{-\frac{\lambda}{5}(1-\cos\beta)},
\ee
which verifies \eqref{2.3}.
For $*=disk$, \eqref{alpha8} gives analogous results with the quantity $\frac{\lambda}{5}(1-\cos\beta)$ replaced by $\lambda\pi(1-\cos\beta)$.
For $*=m$, using \eqref{MassiveAlphaEstimate2} and the notation introduced right after Lemma \ref{MassiveOnePointFunctionEstimate}, we have that
\begin{eqnarray} \label{eq:limit-loop-mass}
\lim_{\delta \to 0} \frac{e^{- \lambda \alpha^{m}_{\delta,d_{z,w}}(z)(1-\cos\beta) }}{\delta^{2\Delta^{m}_{\lambda,\beta}}} 
& = & \lim_{\delta \to 0} \frac{e^{-\lambda \alpha^{loop}_{\delta,d_{z,w}}(z)(1-\cos\beta)}}{\delta^{\frac{\lambda}{5}(1-\cos\beta)}} \,
e^{ \lambda(1-\cos\beta) \big( \alpha^{loop}_{\delta,d_{z,w}}(z) - \alpha^{m}_{\delta,d_{z,w}}(z) \big) } \notag \\
& = & \Big( \frac{1}{d_{z,w}} \Big)^{\frac{\lambda}{5}(1-\cos\beta)} e^{\lambda (1-\cos\beta) \hat\alpha_{0,d_{z,w}}(z)}
\end{eqnarray}
and
\begin{eqnarray}
e^{- \lambda \alpha^{m}_{\delta,d_{z,w}}(z)(1-\cos\beta) }
& = & \Big( \frac{d_{z,w}}{\delta} \Big)^{-\frac{\lambda}{5}(1-\cos\beta)} e^{ \lambda(1-\cos\beta) \big( \alpha^{loop}_{\delta,d_{z,w}}(z) - \alpha^{m}_{\delta,d_{z,w}}(z) \big) } \notag \\
& \leq & c(D) \Big( \frac{d_{z,w}}{\delta} \Big)^{-\frac{\lambda}{5}(1-\cos\beta)},
\end{eqnarray}
where $c(D) = e^{ \lambda(1-\cos\beta) \sup_{u \in D}\hat\alpha_{0,D}(u) }$ and
the last inequality follows from the fact that $\alpha^{loop}_{\delta,d_{z,w}}(z) - \alpha^{m}_{\delta,d_{z,w}}(z)$ is increasing as $\delta \downarrow 0$ and therefore
\be \nonumber
\alpha^{loop}_{\delta,d_{z,w}}(z) - \alpha^{m}_{\delta,d_{z,w}}(z) \leq \lim_{\delta \to 0} \big( \alpha^{loop}_{\delta,d_{z,w}}(z) - \alpha^{m}_{\delta,d_{z,w}}(z) \big)
= \hat\alpha_{0,d_{z,w}}(z) \leq \sup_{u \in D} \hat\alpha_{0,D}(u) < \infty,
\ee
where the last inequality follows from \eqref{MassiveAlphaEstimate3}.
\end{proof}

The next result concerns the existence of Gaussian layering fields in bounded domains with a $C^1$ boundary.

\begin{theorem} \label{ExistenceGaussianLayeringFieldBoundedDomains}
If $D$ is a bounded simply connected domain with a $C^1$ boundary, then the conclusions of Theorem \ref{ExistenceGaussianLayeringField} hold.
\end{theorem}

\begin{proof}
By Theorem~\ref{convergence}, it suffices to check \eqref{two_point_function_exists}, \eqref{two_point_function}, and \eqref{2.3}.
Using the fact that
\be \label{eq:covariance}
\text{Cov}(G(h^{\delta}_z)G(h^{\delta}_w)) = \alpha^{*}_{\delta,D}(z,w), 
\ee
which follows from the definition of the Gaussian process $G$ (see \eqref{Gauss} and the discussion right after), we have that
\begin{eqnarray}
\langle W_{\xi}^{\delta}(z) \overline{W_{\xi}^{\delta'}(w)} \rangle^*_D & = & \Big\langle e^{i \xi \left[ G(A_{\delta,D}(z)) - G(A_{\delta',D}(w)) \right]} \Big\rangle^*_D \\
& = & e^{-\frac{\xi^2}{2} \text{Var}\big( G(A_{\delta,D}(z)) - G(A_{\delta',D}(w)) \big)} \\
& = & e^{-\frac{\xi^2}{2} \big( \alpha^{*}_{\delta,D}(z) + \alpha^{*}_{\delta',D}(w) + 2 \alpha^{*}_{\delta,D}(z,w) \big)} \\
& \leq & e^{-\frac{\xi^2}{2} \big( \alpha^{*}_{\delta,D}(z) + \alpha^{*}_{\delta',D}(w) \big)},
\end{eqnarray}
which verifies condition \eqref{two_point_function} with $\tau^{\delta}_{z,w} = \frac{\xi^2}{2} \alpha^{*}_{\delta,D}(z)$.
Conditions \eqref{two_point_function_exists} and \eqref{2.3} can now be checked as in the proof of Theorem \ref{ExistenceLayeringField}.

To prove the last statement, we will use Theorem 1.1 of \cite{GMC_imaginary}. To that end, we write
\be \label{eq:normalized}
\delta^{-2\Delta_\xi^*} W_{\xi}^\delta(z) = \delta^{-2\Delta_\xi^*}  e^{i \xi G(h_z^\delta)} = e^{i \xi G(h_z^\delta) - 2 \Delta^{*}_{\xi} \log\delta}.
\ee
For $*=loop$, using \eqref{alpha5}, we have that
\be \label{eq:variance}
\E(G(h^{\delta}_{z})^2) = \alpha^{loop}_{\delta,D}(z) = \alpha^{loop}_{\delta,d_{z}}(z) + \alpha^{loop}_{d_{z},D}(z)
= \frac15 \log\frac{d_{z}}{\delta} + \alpha^{loop}_{d_{z},D}(z).
\ee
Since $\Delta^{loop}_{\xi} = \xi^2/20$, using \eqref{eq:normalized}, we can write
\be
\delta^{-2\Delta_\xi^{loop}} W_{\xi}^\delta(z) = e^{i \xi G(h_z^\delta) - \frac{\xi^2}{10} \log\delta}
\ee
while, using \eqref{eq:variance}, we have that
\be
e^{i \xi G(h_z^\delta) + \frac{\xi^2}{2} \E(G(h^{\delta}_{z})^2)} = e^{\frac{\xi^2}{2} \big( \frac15\log d_{z} + \alpha^{loop}_{d_{z},D}(z) \big)} e^{i \xi G(h_z^\delta) - \frac{\xi^2}{10} \log\delta}.
\ee
Therefore,
\begin{eqnarray}
\delta^{-2\Delta_\xi^{loop}} W_{\xi}^\delta(z) & = & e^{-\frac{\xi^2}{2} \big( \frac15\log d_{z} + \alpha^{loop}_{d_{z},D}(z) \big)}
e^{i \xi G(h_z^\delta) + \frac{\xi^2}{2} \E(G(h^{\delta}_{z})^2)} \\
& =& e^{-\frac{\xi^2}{2} \big( \frac15\log d_{z} + \alpha^{loop}_{d_{z},D}(z) \big)} e^{i \frac{\xi}{\sqrt 5} X^{\delta}(z) + \frac{\xi^2}{10} \E(X^{\delta}(z)^2)},
\end{eqnarray}
where $X^{\delta}(z) := \sqrt{5} G(h^{\delta}_{z})$.
It is a consequence of \eqref{eq:cov-conv}, as discussed right after \eqref{eq:cov-conv}, that the field ${\mathcal G}^{\delta} = \{ G(h^{\delta}_{z}) \}_{z \in D}$
converges to ${\mathcal G}^{*} = \{ G(h^{0}_{z}) \}_{z \in D}$, as $\delta \to 0$, in second mean in the Sobolev space $\mathcal{H}^{-\alpha}(D)$ for every $\alpha>3/2$.
This implies that $X^{\delta}$ converges to a Gaussian field $X$ with covariance
\be
K^{X}(z,w) = \log^{+}\frac{1}{|z-w|} + 5 \, g^{loop}(z,w),
\ee
where $g^{loop}$ is bounded, integrable and continuous, as a consequence of Lemma \ref{l:cont}.
In order to use Theorem 1.1. of \cite{GMC_imaginary}, we need to check that $X^{\delta}$ is a \emph{standard approximation} of $X$ in the sense of 
Definition 2.7 of \cite{GMC_imaginary}. Condition (i) in that definition is clearly satisfied since
\be \nonumber
\lim_{\delta,\delta' \to 0} \E(X^{\delta}(z) X^{\delta'}(w)) = 5 \lim_{\delta,\delta' \to 0} \alpha^{loop}_{\max{(\delta,\delta')},D}(z,w) = \log^{+}\frac{1}{|z-w|} + 5 \, g^{loop}(z,w).
\ee
Condition (ii), in our setting, can be written as follows.
There exists a decreasing function $c(\delta)$ such that $c(\delta)>0$ for each $\delta>0$, $\lim_{\delta \to 0} c(\delta) = 0$ and, for every compact $K \subset D$,
\begin{eqnarray} \label{eq:condition2}
&& \sup_{1 \geq \delta>0} \sup_{z,w \in K} \Big| \E(X^{\delta}(z) X^{\delta}(w)) - \log\frac{1}{\max(c(\delta),|z-w|)} \Big| \nonumber \\
&& \quad = \sup_{1 \geq \delta>0} \sup_{z,w \in K} \Big| 5 \, \alpha^{loop}_{\delta,D}(z,w) - \log\frac{1}{\max(c(\delta),|z-w|)} \Big| < \infty.
\end{eqnarray}
To verify this condition, we take $c(\delta)=\delta$ and define the function
\begin{equation} \label{functionF}
F_K(z,w;\delta) := \Big| 5 \, \alpha^{loop}_{\delta,D}(z,w) - \log\frac{1}{\max(\delta,|z-w|)} \Big|
\end{equation}
for all $z,w \in K$ and all $1 \geq \delta > 0$, as well as for $\delta=0$, provided that $z \neq w$.
We observe that $F_K$ is bounded for $\delta>0$ and that, when $\delta=0$, it can be extended continuously by setting, for each $z \in K$,
\begin{eqnarray}
F_K(z,z;0) & := & \lim_{\delta \to 0} F_K(z,z;\delta) = \lim_{\delta \to 0} \Big| 5 \, \alpha^{loop}_{\delta,D}(z) - \log\frac{1}{\delta} \Big| \notag \\
& = & \lim_{\delta \to 0} \Big| \log\frac{d_{z}}{\delta} + 5 \, \mu^{loop}_{D}(\gamma: z \in \bar\gamma, \diam(\gamma) > d_{z}) - \log\frac{1}{\delta} \Big| \notag \\
& = & \log d_{z} + 5 \, \mu^{loop}_{D}(\gamma: z \in \bar\gamma, \diam(\gamma) > d_{z}).
\end{eqnarray}
It is now clear that
\begin{eqnarray} \notag
&& \sup_{1 \geq \delta>0} \sup_{z,w \in K} \Big| 5 \, \alpha^{loop}_{\delta,D}(z,w) - \log\frac{1}{\max(c(\delta),|z-w|)} \Big| \\
&& = \sup_{z,w \in K, 0 \leq \delta \leq 1} F_{K}(z,w;\delta) < \infty.
\end{eqnarray}
Condition (iii), in our setting, can be written as follows:
\begin{eqnarray} \label{eq:condition3}
&& \sup_{1 \geq \delta>0} \sup_{z,w \in D} \Big( \E(X^{\delta}(z) X^{\delta}(w)) - \log\frac{1}{|z-w|} \Big) \notag \\
&& = \sup_{1 \geq \delta>0} \sup_{z,w \in D} \Big( 5 \, \alpha^{loop}_{\delta,D}(z,w) - \log\frac{1}{|z-w|} \Big) < \infty.
\end{eqnarray}
To verify this condition, we define the function
\begin{equation} \label{functionG}
H(z,w;\delta) := \Big( 5 \, \alpha^{loop}_{\delta,D}(z,w) - \log\frac{1}{|z-w|} \Big)
\end{equation}
for all $0 \leq \delta \leq 1$ and all $z,w \in D$ provided $z \neq w$.
We note that, for all $\delta>0$ and every $z \in D$, $\lim_{w \to z} H(z,w;\delta) = -\infty$.
For $\delta=0$, in view of \eqref{ThScaleInvTwoPoint} and \eqref{ThScaleInvTwoPointMassive}, using \eqref{alpha5},
we can extend $H$ continuously by setting, for every $z \in D$,
\begin{eqnarray}
H(z,z;0) & := & \lim_{w \to z} H(z,w;0) = \lim_{w \to z} \Big( 5 \, \alpha^{loop}_{D}(z,w) - \log\frac{1}{|z-w|} \Big) \notag \\
& = & \log d_{z} + 5 \, \Psi^{loop}(z,D).
\end{eqnarray}
Using \eqref{eq:Psi-alpha} in the appendix, we see that
\be
\lim_{z \to \partial D} \Psi^{loop}(z,D) = - \alpha_{\neg B_{{\bf 0},1}}({\bf 0}|{\bf 1}) < \infty.
\ee
Therefore
\be
\sup_{1 \geq \delta>0} \sup_{z,w \in D} \Big( 5 \, \alpha^{loop}_{\delta,D}(z,w) - \log\frac{1}{|z-w|} \Big)
= \sup_{z,w \in D, 0 \leq \delta \leq 1} H(z,w;\delta) < \infty.
\ee
To conclude the proof for $*=loop$, we note that, in order to apply Theorem 1.1 of \cite{GMC_imaginary}, we also need the condition $\xi/\sqrt 5<\sqrt 2$,
which is equivalent to $\Delta^{loop}_{\xi}<1/2$. Applying Theorem 1.1 of \cite{GMC_imaginary}, we can now conclude that
\be
e^{i \xi G(h_z^\delta) + \frac{\xi^2}{2} \E(G(h^{\delta}_{z})^2)} \stackrel{\delta \to 0}{\longrightarrow} M^{loop}_{\xi} = e^{i \frac{\xi}{\sqrt 5} X} = e^{i \xi G(h^{0})},
\ee
where the convergence is in probability in the Sobolev space $\mathcal{H}^{-\alpha}$ for any $\alpha>1$.

The proof for $*=disk$ is exactly analogous, using $\Delta^{disc}_{\xi} = \pi\xi^2/4$, and \eqref{alpha8} in place of \eqref{alpha5}.

For $*=m$, using \eqref{eq:variance}, we can write
\be \label{eq:variance-massive}
\E(G(h^{\delta}_{z})^2) = \alpha^{m}_{\delta,D}(z) = -\frac15\log\delta + \frac15\log d_{z} + \alpha^{loop}_{\delta,d_{z}}(z)
- (\alpha^{loop}_{\delta,D}(z) - \alpha^{m}_{\delta,D}(z)).
\ee
The rest of the proof is analogous to the proof of the case $*=loop$, and the relation $\Theta^{m}_{D}(z) = \Theta^{loop}_{D}(z) - \hat\alpha_{0,D}(z)$
is a consequence of \eqref{eq:variance-massive} and of $\lim_{\delta \to 0}(\alpha^{loop}_{\delta,D}(z) - \alpha^{m}_{\delta,D}(z)) = \hat\alpha_{0,D}(z)$.
\end{proof}

In the rest of this section, we show the existence of layering fields for general domains $D$ conformally equivalent to the unit disk $\D$ and prove their
conformal covariance. Let $f: {\mathbb D} \to D$ be a conformal map.
In view of our use of the notation $U_D(\varphi) = \int_{D} U_D(z) \varphi(z) dz$, the conformal covariance proved below can be (formally) expressed
in the following way:
\be \label{eq:conformal_covariance_loop}
V^{loop}_{\lambda,\beta,D}(z) dz \stackrel{d}{=} |f'(w)|^{2 - 2 \Delta^{loop}_{\lambda,\beta}} V^{loop}_{\lambda,\beta,{\mathbb D}}(w) dw,
\ee
\be
V^{\tilde m}_{\lambda,\beta,D}(z) dz \stackrel{d}{=} |f'(w)|^{2 - 2 \Delta^{loop}_{\lambda,\beta}} V^{m}_{\lambda,\beta,{\mathbb D}}(w) dw,
\ee
\be \label{eq:conformal_covariance_loop_gauss}
W^{loop}_{\xi,D}(z) dz \stackrel{d}{=} |f'(w)|^{2 - 2 \Delta^{loop}_{\xi}} W^{loop}_{\xi,{\mathbb D}}(w) dw,
\ee
\be
W^{\tilde m}_{\xi,D}(z) dz \stackrel{d}{=} |f'(w)|^{2 - 2 \Delta^{loop}_{\xi}} W^{m}_{\xi,{\mathbb D}}(w) dw,
\ee
where $\stackrel{d}{=}$ denotes equality in distribution, $w \in \D$, $m:\D \to {\mathbb R}^+$ is a mass function on $\D$, $z=f(w) \in D$, and $\tilde m(z) = |f'(w)|^{-1} m(w)$.

\medskip

\begin{remark} \label{remark:conf_inv} {\rm
For the disk model, an analog of \eqref{eq:conformal_covariance_loop} and of Theorem \ref{thm:conformal_covariance} below hold when $D=\D$ and $f$ is a M\"obius transformation.
The proof is the same as that of Theorem \ref{thm:conformal_covariance}.
}
\end{remark}

The following statement is the main result of the section.

\begin{theorem}\label{thm:conformal_covariance}
Let $f: {\mathbb D} \to D$ be a conformal map from the unit disc $\mathbb D$ to a domain $D$ conformally equivalent to $\D$.
Let $m: \D \to {\mathbb R}^+$ be a mass function on $\D$ and define a mass function on $D$ via $\tilde m(f(w)) = |f'(w)|^{-1} m (w)$ for each $z=f(w) \in D$.
Then, for every $\varphi \in C^{\infty}_{0}(D)$, $U_D(\varphi)$ exists as a limit in distribution of
$U^{\delta}_D(\varphi)$ for $U_D = V^{loop}_{\lambda,\beta,D}, V^{m}_{\lambda,\beta,D}, W^{loop}_{\xi,D}$, and $W^{m}_{\xi,D}$. Moreover,
\be
\int_{D} V^{loop}_{\lambda,\beta,{D}}(z) \varphi(z) dz
\stackrel{d}{=} \int_{\D} |f'(w)|^{2 - 2 \Delta^{loop}_{\lambda,\beta}} V^{loop}_{\lambda,\beta,\D}(w) \varphi \circ f(w) dw,
\ee
\be
\int_{D} V^{\tilde m}_{\lambda,\beta,{D}}(z) \varphi(z) dz
\stackrel{d}{=} \int_{\D} |f'(w)|^{2 - 2 \Delta^{loop}_{\lambda,\beta}} V^m_{\lambda,\beta,\D}(w) \varphi \circ f(w) dw,
\ee
\be
\int_{D} W^{loop}_{\xi,{D}}(z) \varphi(z) dz
\stackrel{d}{=} \int_{\D} |f'(w)|^{2 - 2 \Delta^{loop}_{\xi}} W^{loop}_{\xi,\D}(w) \varphi \circ f(w) dw,
\ee
\be
\int_{D} W^{\tilde m}_{\xi,{D}}(z) \varphi(z) dz
\stackrel{d}{=} \int_{\D} |f'(w)|^{2 - 2 \Delta^{loop}_{\xi}} W^m_{\xi,\D}(w) \varphi \circ f(w) dw.
\ee
\end{theorem}

\begin{proof}
In this proof, $*=loop$ or $m$ and we will use the notation $d_{z,D}$ instead of $d_z$ to denote $\dist(z,\partial D)$ to avoid confusion.
We first deal with the Poisson layering fields $V^{loop}_{\lambda,\beta,{D}}$ and $V^{m}_{\lambda,\beta,{D}}$.
Let $N^{B_{z,\delta}}_{\lambda,D}(z)$ denote the random variable corresponding to $\sum_{\gamma} \epsilon_{\gamma}$ where the sum is over loops $\gamma$ from a
Brownian loop soup winding around $z$, contained in $D$, but not completely contained in $B_{z,\delta}$, where each $\epsilon_{\gamma}$ is $1$ or $-1$ with equal probability.
Then, we can write
\be \label{eq:splitting}
N^{\delta}_{\lambda,D}(z) \stackrel{d}{=} N^{\delta}_{\lambda,{B_{z,\delta}}}(z) + N^{B_{z,\delta}}_{\lambda,D}(z),
\ee
where $N^{\delta}_{\lambda,{B_{z,\delta}}}(z) \stackrel{d}{=} N^{1}_{\lambda,{B_{0,1}}}(0)$ is independent of $\delta$ and $z$ by the scale and translation invariance of the Brownian loop soup. (Note that in \eqref{eq:splitting} and later in this proof, the various $N$ variables are independent.)

Assume that $z \in D$ is such that $d_{z,D} > \delta$ and $d_{f^{-1}(z),\D} > |(f^{-1})'(z)| \delta$, and let $w=f^{-1}(z)$, $s=|(f^{-1})'(z)|=1/|f'(w)|$;
then, for $*=loop$, the conformal invariance of the Brownian loop soup implies that
\begin{eqnarray}
N^{B_{z,\delta}}_{\lambda,D}(z) & \stackrel{d}{=} & N^{f^{-1}(B_{z,\delta})}_{\lambda,\D}(w) \\
& \stackrel{d}{=} & N^{B_{w,s\delta}}_{\lambda,\D}(w) + N^{f^{-1}(B_{z,\delta})}_{\lambda,B_{w,s\delta}}(w) - N^{B_{w,s\delta}}_{\lambda,f^{-1}(B_{z,\delta})}(w).
\end{eqnarray}
By an application of Lemma~4.2 of \cite{camia2016conformal}, the intensity measures of the Poisson random variables (i.e., the number of loops from the Brownian loop soup)
corresponding to $N^{f^{-1}(B_{z,\delta})}_{\lambda,B_{w,s\delta}}(w)$ and $N^{B_{w,s\delta}}_{\lambda,f^{-1}(B_{z,\delta})}(w)$ tend to zero as $\delta \to 0$.
Moreover,
\begin{eqnarray}
N^{\delta}_{\lambda,{B_{z,\delta}}}(z) + N^{B_{w,s\delta}}_{\lambda,\D}(w) & \stackrel{d}{=} & N^{s\delta}_{\lambda,{B_{w,s\delta}}}(w) + N^{B_{w,s\delta}}_{\lambda,\D}(w) \\
& \stackrel{d}{=} & N^{s\delta}_{\lambda,\D}(w).
\end{eqnarray}

For $*=m$, the equations above still hold with the following stipulation: if the loop soup in $\mathbb D$ has mass function $m(z)$, then the loop soup in $D=f(\D)$ has mass
$\tilde m(z) = |f'(w)|^{-1} m(w)$. This follows from the conformal covariance of the massive loop soup, as discussed in Section~2.3 of \cite{camia_notes}
(see, in particular, Equation~2.13 there).

Now consider a test function $\varphi \in C^{\infty}_{0}(D)$. The considerations above apply to all $z \in \mathrm{supp}(\varphi)$ provided that
$\delta < \mathrm{dist}(\mathrm{supp}(\varphi), \partial D)$. Hence, we have that
\begin{eqnarray}
\int_{D} V^{*}_{\lambda,\beta,D}(z) \varphi(z) dz & := & \lim_{\delta \to 0} \int_{D} \frac{e^{i \beta N^{\delta}_{\lambda,D}(z)}}{\delta^{2\Delta^*_{\lambda,\beta}}} \varphi(z) \, dz \\
& \stackrel{d}{=} & \lim_{\delta \to 0} \int_{\D} \frac{e^{i \beta N^{s\delta}_{\lambda,\D}(w)}}{\delta^{2\Delta^*_{\lambda,\beta}}} \varphi \circ f(w) \, df(w) \\
& \stackrel{d}{=} & \lim_{\delta \to 0} \int_{\D} \frac{e^{i \beta N^{s\delta}_{\lambda,\D}(w)}}{(s\delta)^{2\Delta^*_{\lambda,\beta}}} s^{2\Delta^*_{\lambda,\beta} - 2} \varphi \circ f(w) \, dw \label{limit_exists} \\
& \stackrel{d}{=} & \int_{\D} V^{*}_{\lambda,\beta,\D}(w) |f'(w)|^{2 - 2\Delta^*_{\lambda,\beta}} \varphi \circ f(w) \, dw,
\end{eqnarray}
where the limits are in distribution and the existence of the limit in \eqref{limit_exists} is guaranteed by Theorem \ref{ExistenceLayeringFieldBoundedDomains}.

The proof for the Gaussian layering fields $W^{loop}_{\xi,D}$ and $W^{m}_{\xi,D}$ is similar; we sketch the key argument for completeness.
Assume again that $z \in D$ is such that $d_{z,D} > \delta$ and $d_{f^{-1}(z),\D} > |(f^{-1})'(z)| \delta$, let $w=f^{-1}(z)$, $s=|(f^{-1})'(z)|=1/|f'(w)|$,
and recall the notation 
\be
A_{\delta,V}(z) = \{ \gamma: z \in \overline \gamma, \delta \leq \diam(\gamma), \gamma \subseteq V \})
\ee
and
\be
A_{U,V}(z) =  \{ \gamma: z \in \overline \gamma, \gamma \not \subseteq U, \gamma \subseteq V \}.
\ee
With this notation, we can write $A_{\delta,D}(z)$ and $A_{s\delta,\D}(w)$ as a disjoint unions, as follows:
\be
A_{\delta,D}(z) = A_{\delta,B_{z,\delta}}(z) \cup A_{B_{z,\delta},D}(z).
\ee
and
\be
A_{s\delta,\D}(w) = A_{w,B_{w,s\delta}}(w) \cup A_{B_{w,s\delta},\D}(w).
\ee
The conformal invariance of $\mu^{loop}$ implies that
\be
\mu^{loop}(A_{\delta,B_{z,\delta}}(z)) = \mu^{loop}(A_{s\delta,B_{w,s\delta}}(w)) = \mu^{loop}(A_{1,B_{{\bf 0},1}}({\bf 0}))
\ee
and
\be
\mu^{loop}(A_{B_{z,\delta},D}(z)) = \mu^{loop}(A_{f^{-1}(B_{z,\delta}),\D}(w)),
\ee
so that
\be
\mu^{loop}(A_{\delta,D}(z)) = \mu^{loop}(A_{1,B_{{\bf 0},1}}({\bf 0})) + \mu^{loop}(A_{f^{-1}(B_{z,\delta}),\D}(w)).
\ee
Comparing this with
\be
\mu^{loop}(A_{s\delta,\D}(w)) = \mu^{loop}(A_{1,B_{{\bf 0},1}}({\bf 0})) + \mu^{loop}(A_{B_{w,s\delta},\D}(w)),
\ee
we see that, as $\delta \to 0$,
\be \notag
| \mu^{loop}(A_{\delta,D}(z)) - \mu^{loop}(A_{s\delta,\D}(w)) | = | \mu^{loop}(A_{f^{-1}(B_{z,\delta}),\D}(w)) - \mu^{loop}(A_{B_{w,s\delta},\D}(w)) | \to 0
\ee
by Lemma~4.2 of \cite{camia2016conformal}.

Now consider a test function $\varphi \in C^{\infty}_{0}(D)$. The considerations above apply to all $z \in \mathrm{supp}(\varphi)$ provided that
$\delta < \mathrm{dist}(\mathrm{supp}(\varphi), \partial D)$. Hence, we have that
\begin{eqnarray}
\int_{D} W^*_{\xi,D}(z) \varphi(z) dz & := & \lim_{\delta \to 0} \int_D \frac{e^{i \xi G(h^{\delta}_z)}}{\delta^{2\Delta^*_{\xi}}} \varphi(z) \, dz \\
& = & \lim_{\delta \to 0} \int_{D} \frac{e^{i \xi G(A_{\delta,D}(z))}}{\delta^{2\Delta^*_{\xi}}} \varphi(z) \, dz \\
& \stackrel{d}{=} & \lim_{\delta \to 0} \int_{\D} \frac{e^{i \xi G(A_{s\delta,\D}(w))}}{\delta^{2\Delta^*_{\lambda,\beta}}} \varphi \circ f(w) \, df(w) \\
& \stackrel{d}{=} & \lim_{\delta \to 0} \int_{\D} \frac{e^{i \xi G(A_{s\delta,\D}(w))}}{(s\delta)^{2\Delta^*_{\lambda,\beta}}} s^{2\Delta^*_{\xi} - 2} \varphi \circ f(w) \, dw \label{limit_exists_gauss} \\
& \stackrel{d}{=} & \int_{\D} W^{*}_{\xi,\D}(w) |f'(w)|^{2 - 2\Delta^*_{\lambda,\beta}} \varphi \circ f(w) \, dw,
\end{eqnarray}
where the limits are in distribution and the existence of the limit in \eqref{limit_exists_gauss} is guaranteed by Theorem \ref{ExistenceGaussianLayeringFieldBoundedDomains}.
\end{proof}

\subsection{Tightness and uniqueness of layering fields} \label{sec:tightness-uniqueness}

In this section, we prove two results. Lemma \ref{lemma:tightness} shows tightness of the Poisson layering field $V^*_{\lambda,\beta}$ as $\lambda \to \infty, \beta \to 0$ with $\lambda\beta^2 \to \xi^2$
when $\xi^2$ is contained in an appropriate range. Lemma \ref{l:fidi} then establishes that finite dimensional distributions are enough to characterize the limit. For bounded simply connected domains
with a $C^1$ boundary, these two results allow us to lift the convergence in Theorem \ref{p:cil} to convergence in distribution in the metric space $\mathcal{H}^{-\alpha}$ endowed
with the distance induced by the Sobolev norm, as stated in Theorem \ref{MainTheorem}.

\begin{lemma} \label{lemma:tightness}
Take $\lambda>0$ and $\beta \in [0,2\pi)$ and let $*=loop,m$. If $D$ is a bounded simply connected domain with a $C^1$ boundary,
as $\lambda \to \infty, \beta \to 0$ with $\lambda\beta^2 \to \xi^2<5$, $V^*_{\lambda,\beta}$ is tight in the Sobolev space $\mathcal{H}^{-\alpha}$ for any $\alpha>3/2$.
The same conclusion holds for $*=disk$ if $\xi^2<1/\pi$.
\end{lemma}

\begin{proof}
Using the same notation as in the second half of the proof of Theorem \ref{convergence} and a similar argument, we obtain the bound
\begin{equation}
\langle\|V^*_{\lambda,\beta}\|_{\mathcal{H}^{-\alpha}}^2\rangle_D 
\leq \left(\int_{D}\int_{D}\big|\big\langle V^*_{\lambda,\beta}(z) \overline{V^*_{\lambda,\beta}(w)} \big\rangle_D \big| dz dw \right)\sum_{i}\frac{c^2}{\lambda_i^{\alpha-\frac{1}{2}}}. 
\end{equation}
Arguments analogous to those used in the proof of Theorem \ref{ExistenceLayeringFieldBoundedDomains} show that
\begin{equation}
\big|\big\langle V^*_{\lambda,\beta}(z) \overline{V^*_{\lambda,\beta}(w)} \big\rangle_D \big| \leq c(D) \Big( \frac{1}{d_{z,w}d_{w,z}} \Big)^{2\Delta^*_{\lambda,\beta}},
\end{equation}
for some constant $c(D)<\infty$ that depends only on the domain $D$.
Hence, expanding $\Delta^*_{\lambda,\beta}$ to second order in $\beta$, the argument used to prove \eqref{convergent_integral} shows that
$\langle\|V^*_{\lambda,\beta}\|_{\mathcal{H}^{-\alpha}}^2\rangle_D$ stays bounded for any $\alpha>3/2$ as $\lambda \to \infty, \beta \to 0$,
provided that $\lambda\beta^2 \to \xi^2<5$ when $*=loop,m$ and $\lambda\beta^2 \to \xi^2<1/\pi$ when $*=disk$.
Tightness follows from a standard application of Chebyshev's inequality.
\end{proof}

\begin{lemma}\label{l:fidi} Fix a bounded simply connected domain $D\subset \R^2$, as well as $\alpha >0$. Assume that $V$ and $W$ are two random elements with values in $\mathcal{H}^{-\alpha}$ and suppose that, for every $\varphi\in C^\infty_0(D)$, the random variables $V(\varphi)$ and $W(\varphi)$ have the same distribution. Then, $\mathbb{P}_V = \mathbb{P}_W$, where $\mathbb{P}_V$ and $\mathbb{P}_W$ indicate the probability masures induced by $V$ and $W$, respectively, on $(\mathcal{H}^{-\alpha}, \mathscr B)$, where $\mathscr B$ is the Borel $\sigma$-field associated with the norm $\|\cdot \|_{\mathcal{H}^{-\alpha}}$.
\end{lemma}

\begin{proof} Throughout the proof, we use the symbol $\stackrel{d}{=}$ to denote equality in distribution between random elements. Since $V$ and $W$ are random continuous linear functionals on $\mathcal{H}_0^\alpha$, the assumptions in the statement imply immediately that, for every $\varphi_1,...,\varphi_m\in C^\infty_0(D)$, $(V( \varphi_1) ,...,V(\varphi_m)) \stackrel{d}{=} (W( \varphi_1) ,...,W(\varphi_m))$. Exploiting the fact that (by definition) $C_0^{\infty}$ is dense in $\mathcal{H}_0^\alpha$, one deduces that, if $f_1,...,f_m \in \mathcal{H}_0^\alpha$, then there exist sequences $\{\varphi_i^n \in C_0^{\infty} : n\geq 1\}$, $i=1,...,m$, such that $\|  \varphi_i^n - f_i\|_{\mathcal{H}_0^\alpha} \to 0$, as $n\to \infty$, yielding that, $\P$-a.s., $V(\varphi_i^n ) \to V(f_i)$ and $W(\varphi_i^n ) \to W(f_i)$ for $i=1,...,m$. Since almost sure convergence preserves equality in distribution (e.g., by dominated convergence) one infers that $(V( f_1) ,...,V(f_m)) \stackrel{d}{=} (W( f_1) ,...,W(f_m))$ for all $f_1,...,f_m \in \mathcal{H}_0^\alpha$. Now, for $M = 1, 2,...$ and $f = \sum_{i=1}^\infty a_iu_i \in \mathcal{H}_0^\alpha$, set $f^{(M)} := \sum_{i=1}^M a_i u_i$, and define
$$
\| h\|^{(M)}_{\mathcal{H}^{-\alpha}} := \sup_{\| f\|_{\mathcal{H}_0^\alpha} \leq 1}| h(f^{(M)})|, \quad h\in \mathcal{H}^{-\alpha}.
$$
One has that $\| h\|^{(M)}_{\mathcal{H}^{-\alpha}}\leq \| h\|^{(M+1)}_{\mathcal{H}^{-\alpha}}\leq \| h\|_{\mathcal{H}^{-\alpha}}$, which shows that $\lim_{M\to \infty}  \| h\|^{(M)}_{\mathcal{H}^{-\alpha}}$ exists and is bounded above by $\|h\|_{\mathcal{H}^{-\alpha}}$. Moreover, for every $f\in \mathcal{H}_0^\alpha$, $\| f^{(M)}- f\|_{\mathcal{H}^\alpha}\to 0$, as $M\to \infty$ and consequently $h(f^{(M)}) \to h(f)$ for every $h\in \mathcal{H}^{-\alpha}$. This yields that, for every $f\in \mathcal{H}_0^\alpha$ such that $\|f\|_{\mathcal{H}_0^\alpha}\leq 1$ and every $h\in \mathcal{H}^{-\alpha}$, $| h(f)|  = \lim_{M\to \infty} |h(f^{(M)})| \leq \lim_{M\to \infty}  \| h\|^{(M)}_{\mathcal{H}^{-\alpha}}$, which shows that $\|h\|_{\mathcal{H}^{-\alpha}} \leq \lim_{M\to \infty}  \| h\|^{(M)}_{\mathcal{H}^{-\alpha}}$ and leads to the conclusion that $\|h\|_{\mathcal{H}^{-\alpha}} = \lim_{M\to \infty}  \| h\|^{(M)}_{\mathcal{H}^{-\alpha}}$. To conclude the proof, fix $h_1,...,h_d \in \mathcal{H}^{-\alpha}$, as well as $\gamma_1,...,\gamma_d\in \R$: the previous discussion implies that, $\P$-a.s.,
$$
\sum_{i=1}^d \gamma_i \| V - h_i\|^{(M)}_{\mathcal{H}^{-\alpha}}  \longrightarrow \sum_{i=1}^d \gamma_i \| V - h_i\|_{\mathcal{H}^{-\alpha}} \quad \text{as } M\to\infty.
$$
Since each of the random variables $\sum_{i=1}^d \gamma_i \| V - h_i\|^{(M)}_{\mathcal{H}^{-\alpha}}$ only depends on the vector $(V(u_1),...,V(u_M))\stackrel{d}{=}(W(u_1),...,W(u_M))$ (where equality in distribution follows from the first part of the proof), we deduce that $\sum_{i=1}^d \gamma_i \| V - h_i\|^{(M)}_{\mathcal{H}^{-\alpha}}\stackrel{d}{=}\sum_{i=1}^d \gamma_i \| W - h_i\|^{(M)}_{\mathcal{H}^{-\alpha}}$ for every $M$, which in turn implies that $$\big( \| V - h_1\|_{\mathcal{H}^{-\alpha}},...,  \| V - h_d\|_{\mathcal{H}^{-\alpha}}\big) \stackrel{d}{=} \big( \| W - h_1\|_{\mathcal{H}^{-\alpha}},...,  \| W - h_d\|_{\mathcal{H}^{-\alpha}}\big).$$ The conclusion follows from the fact that $\mathscr{B}$ is generated by finite intersections of sets of the type
$$
\big\{ g \in \mathcal{H}^{-\alpha} : \| g-h\|_{\mathcal{H}^{-\alpha}} < t\big\}, \quad h\in \mathcal{H}^{-\alpha}, \,  t>0,
$$
combined with a monotone class argument.
\end{proof}


\begin{thebibliography}{10}

\bibitem{aizenman} M.~Aizenman. Geometric Analysis of $\phi^4$ Fields and Ising Models. Parts I and II,
\emph{Comm. Math. Phys.} { 86}: 1--48, 1982.
	
	\bibitem{conformal_weldings}
	K. Astala, P. Jones, A. Kupiainen, and E. Saksman.
	\newblock  Random conformal weldings.
	\newblock {\em Acta Math.} 207 , no. 2: 203--254, 2011.
	
	\bibitem{BPZ_NP}
	A.A. Belavin, A.M. Polyakov, and A.B. Zamolodchikov.
	\newblock Infinite conformal symmetry in two-dimensional quantum field theory.
	\newblock {\em Nuclear Phys. B}, 241(2):333--380, 1984.
	
	\bibitem{BPZ_JSP}
	A.A. Belavin, A.M. Polyakov, and A.B. Zamolodchikov.
	\newblock Infinite conformal symmetry of critical fluctuations in two
	dimensions.
	\newblock {\em J. Statist. Phys.}, 34(5-6):763--774, 1984.
	
	\bibitem{GMC_berestycki}
	N. Berestycki.
	\newblock An elementary approach to {G}aussian multiplicative chaos.
	\newblock {\em Electron. Commun. Probab.}, 22:Paper No. 27, 12, 2017.
	
	\bibitem{GMC_GUE}
	N. Berestycki, C. Webb, and M.~D. Wong.
	\newblock Random {H}ermitian matrices and {G}aussian multiplicative chaos.
	\newblock {\em Probab. Theory Related Fields}, 172(1-2):103--189, 2018.
	
	\bibitem{BJS_Book}
         L. Bers, F. John, and M. Schechter. {\it Partial differential equations}. Lectures in Applied Mathematics 3A, American Mathematical Society, 1964.

	 \bibitem{growth}
	  A. Borodin and P. Ferrari
	  \newblock  Anisotropic growth of random surfaces in 2+1 dimensions.
	  \newblock {\em Comm.Math. Phys}. 325, no. 2: 603–-684, 2014.

	\bibitem{erik_federico}
	E.I. Broman and F. Camia.
	\newblock Universal behavior of connectivity properties in fractal percolation
	models.
	\newblock {\em Electron. J. Probab.}, 15:1394--1414, 2010.

\bibitem{bfso} D.C.~Brydges, J.~Fr\"ohlich, A.D.~Sokal,
The random-walk representation of classical spin systems and correlation inequalities. II. The skeleton inequalities,
\emph{Comm. Math. Phys.} {\bf 91}: 117--139 1983.

\bibitem{bfsp} D.C.~Brydges, J.~Fr\"ohlich, T.~Spencer,
The random walk representation of classical spin systems and correlation inequalities,
\emph{Comm. Math. Phys.}, { 83}: 123--150, 1982.
	
	\bibitem{Spin_loops_berg_camia_lis}
	T. van~de Brug, F. Camia, and M. Lis.
	\newblock Spin systems from loop soups.
	\newblock {\em Electron. J. Probab.}, 23:17 pp., 2018.
	
	\bibitem{camia_notes}
	F. Camia.
	\newblock Scaling Limits, Brownian Loops and Conformal Fields.
	\newblock In {\em Advances in Disordered Systems, Random Processes and Some Applications}: 205--269. Cambridge Univ. Press, Cambridge, 2017.
%	\newblock {\em arXiv preprint arXiv:1501.04861}, 2015.

	\bibitem{CFGL2019}
	F. Camia, V.F. Fiot, A. Gandolfi, and M. Kleban.
	\newblock Exact Correlation Functions in the Brownian Loop Soup.
	\newblock {\em Preprint arXiv:1912.00973}, 2019.
	

	\bibitem{camia2016conformal}
	F. Camia, A. Gandolfi, and M. Kleban.
	\newblock Conformal correlation functions in the {B}rownian loop soup.
	\newblock {\em Nuclear Physics B}, 902:483--507, 2016.

	\bibitem{CGN2015}
	F. Camia, C. Garban, and C.M. Newman.
	\newblock Planar Ising magnetization field I. Uniqueness of the critical scaling limit.
	\newblock {\em Ann. Probab.} {\bf 43}, 528--571, 2015.

	\bibitem{CGN2016}
	F. Camia, C. Garban, and C.M. Newman.
	\newblock Planar Ising magnetization field II. Properties of the critical and near-critical scaling limits.
	\newblock {\em Ann. Inst. H. Poincar\'e Probab. Statist.} {\bf 52}, 146--161, 2016.
	
	\bibitem{disorder}
	D. Carpentier and P. Le Doussal
	\newblock  Glass  transition  of  a  particle  in  a  random  potential,  front selection in nonlinear RG and entropic phenomena in Liouville and Sinh-Gordon models. 
	\newblock {\emph Phys.Rev. E} 63:026110, 2001.
	
	\bibitem{LQG_Riemann}
	 F. David, A. Kupiainen, R. Rhodes, and V. Vargas
	 \newblock  Liouville Quantum Gravity on the Riemann sphere.
	 \newblock {\emph Commun. Math. Phys.} 342 (3): 869--907, 2016.
	
	\bibitem{CFT}
	P. Di~Francesco, P. Mathieu, and D. S\'{e}n\'{e}chal.
	\newblock {\em Conformal field theory}.
	\newblock Graduate Texts in Contemporary Physics. Springer-Verlag, New York,
	1997.
	
\bibitem{dynkin1} E.B. Dynkin, Markov processes as a tool in field theory,
\emph{J. Funct. Anal.} { 50}: 167--187, 1983.

\bibitem{dynkin2} E.B. Dynkin. Gaussian and non-Gaussian random fields associated with Markov processes,
\emph{J.  Funct. Anal.} { 55}: 344--376, 1984.

	\bibitem{duplantier_sheffield}
	B. Duplantier and S. Sheffield.
	\newblock Liouville quantum gravity and KPZ.
	\newblock {\em Inventiones Mathematicae}, 185(2):333--393,  2011.
	
\bibitem{ffs-book} R.~Fern\'andez, J.~Fr\"ohlich, A.D.~Sokal,
\emph{Random Walks, Critical Phenomena, and Triviality in Quantum Field Theory},
%Texts and Monographs in Physics,
Springer-Verlag 1992. %Berlin

	\bibitem{FK}
	B. Freivogel and M. Kleban.
	\newblock A conformal field theory for eternal inflation ?
	\newblock {\em J. High Energy Phys.}, (12):019, 31, 2009.

	\bibitem{GaCa06}
	A. Gamsa and J. Cardy.
	\newblock Correlation functions of twist operators applied to single self-avoiding loops.
	\newblock {\em J. Phys. A}, (39):12983, 2006.
	
%	\bibitem{garban_bourbaki}
%	Christophe Garban.
%	\newblock Quantum gravity and the {KPZ} formula [after
%	{D}uplantier-{S}heffield].
%	\newblock {\em Ast\'{e}risque}, (352):Exp. No. 1052, ix, 315--354, 2013.
%	\newblock S\'{e}minaire Bourbaki. Vol. 2011/2012. Expos\'{e}s 1043--1058.
	
	\bibitem{grieser}
	D. Grieser.
	\newblock Uniform bounds for eigenfunctions of the Laplacian on manifolds with
	boundary.
	\newblock {\em Communications in Partial Differential Equations},
	27(7-8):1283--1299, 2002.

	\bibitem{HaWaZi19}
	Y. Han, Y. Wang, and M. Zinsmeister.
	\newblock On The Brownian Loop Measure.
	\newblock {\em J. Stat. Phys.} (175):987--1005, 2019.
	
	\bibitem{henkel}
	M. Henkel.
	\newblock {\em Conformal invariance and critical phenomena}.
	\newblock Texts and Monographs in Physics. Springer-Verlag, Berlin, 1999.
	
	\bibitem{GMC_imaginary}
	J. Junnila, E. Saksman, and C. Webb.
	\newblock Imaginary multiplicative chaos: Moments, regularity and connections
	to the Ising model.
	\newblock {\em Preprint arXiv:1806.02118}, 2018.
	
	\bibitem{kahane}
	J.P. Kahane.
	\newblock Sur le chaos multiplicatif.
	\newblock {\em Ann. Sci. Math. Qu\'ebec}, 9(2):105--150, 1985.
	
	\bibitem{DOZZ}	
	A. Kupiainen, R. Rhodes, and V. Vargas
	\newblock  Integrability of Liouville theory:  proof of the DOZZ Formula.
	\newblock {\em Preprint arXiv:1707.08785.} 2017.
	
	
	\bibitem{ComplexGMC}
	H. Lacoin, R. Rhodes, and V. Vargas.
	\newblock Complex {G}aussian multiplicative chaos.
	\newblock {\em Comm.  Math. Phys.}, 337(2):569--632, 2015.
	
	\bibitem{GMC_CUE_counting}
	G. Lambert, D. Ostrovsky, and N. Simm.
	\newblock Subcritical multiplicative chaos for regularized counting statistics
	from random matrix theory.
	\newblock {\em Comm.  Math. Phys.}, 360(1):1--54,  2018.
	
	\bibitem{Last}
	G. Last.
	\newblock Stochastic analysis for {P}oisson processes.
	\newblock In {\em Stochastic {A}nalysis for {P}oisson {P}oint {P}rocesses}, volume~7
	of {\em Bocconi \& Springer Ser.}: 1--36. %Bocconi Univ. Press,
	Springer International Publishing, 2016.
	
	\bibitem{Last_Penrose_lectures}
	G. Last and M. Penrose.
	\newblock {\em Lectures on the {P}oisson process}, volume~7 of {\em Institute
		of Mathematical Statistics Textbooks}.
	\newblock Cambridge University Press, Cambridge, 2018.
	
	\bibitem{Last_Penrose_PTRF}
	G. Last and M. Penrose.
	\newblock Poisson process {F}ock space representation, chaos expansion and
	covariance inequalities.
	\newblock {\em Probab. Theory Related Fields}, 150(3-4):663--690, 2011.
	
	\bibitem{Lawler}
	G.F. Lawler.
	\newblock {\em Conformally invariant processes in the plane}, volume 114 of
	{\em Mathematical Surveys and Monographs}.
	\newblock American Mathematical Society, Providence, RI, 2005.
	
	\bibitem{Lawler-Werner}
	G.F. Lawler and W. Werner.
	\newblock The {B}rownian loop soup.
	\newblock {\em Probab. Theory and Related Fields}, 128(4):565--588, 2004.
	
\bibitem{lejan1} Y. Le Jan, Markov loops and renormalization, \emph{Ann. Probab.} { 38}:
1280--1319, 2010.

\bibitem{lejan2} Y. Le Jan. \emph{Markov paths, loops and fields}, in Lecture Notes in Mathematics,
volume 2026, Ecole d'Et\'e de Probabilit\'e de St. Flour, Springer, Berlin, 2012.

		\bibitem{lejan}
		Y. Le Jan.
		\newblock Brownian winding fields.
		\newblock {\em arXiv preprint arXiv:1811.02737}, 2018.
		
		\bibitem{mandelbrot}
			B. Mandelbrot. 
	\newblock Intermittent turbulence in self-similar cascades: divergence of high moments and dimension of the carrier
	\newblock {\em Journal of Fluid Mechanics} 62.2:331--358, 1974.

\bibitem{MS16}
J. Miller and S. Sheffield, Imaginary geometry I: interacting SLEs.
\newblock {\em Probab. Theory and Related Fields} { 164}. no. 3-4, 553--705, 2016.

\bibitem{NW11}
S. Nacu and W. Werner. Random soups, carpets and fractal dimensions.
\newblock \emph{Journal of the London Mathematical Society} { 83}, Issue 3: 789--809, 2011.
	
	\bibitem{GMC_CUE_L1}
	M. Nikula, E. Saksman, and C. Webb.
	\newblock Multiplicative chaos and the characteristic polynomial of the cue:
	the {$L^1$}-phase.
	\newblock {\em Preprint arXiv:1806.01831}, 2018.
	
	\bibitem{Gio_ivan}
	I. Nourdin and G. Peccati.
	\newblock {\em Normal approximations with {M}alliavin calculus: From Stein's method to universality.}, volume 192 of
	{\em Cambridge Tracts in Mathematics}.
	\newblock Cambridge University Press, Cambridge, 2012.
	 
	
	\bibitem{Gio_murad_taqqu}
	G. Peccati and M.S. Taqqu.
	\newblock {\em Wiener chaos: moments, cumulants and diagrams: A survey with computer implementation, (supplementary material
		available online).}, Volume~1 of {\em
		Bocconi \& Springer Series}.
	\newblock Springer, Milan; Bocconi University Press, Milan, 2011.
	
	
	\bibitem{Polyakov}
	A.M. Polyakov.
	\newblock {Conformal symmetry of critical fluctuations}.
	\newblock {\em JETP Lett.}, 12:381--383, 1970.
%	\newblock [Pisma Zh. Eksp. Teor. Fiz.12,538(1970)].
	
	\bibitem{pommerenke}
	Ch. Pommerenke.
	\newblock {\em Boundary behaviour of conformal maps}, Volume 299 of {\em
		Grundlehren der Mathematischen Wissenschaften [Fundamental Principles of
		Mathematical Sciences]}.
	\newblock Springer-Verlag, Berlin, 1992.
	
	\bibitem{GMC_LQG}
	R. Rhodes and V. Vargas.
	\newblock Gaussian multiplicative chaos and {L}iouville quantum gravity.
	\newblock In {\em Stochastic processes and random matrices}, pages 548--577.
	Oxford Univ. Press, Oxford, 2017.
	
	\bibitem{GMC_revisit}
	R. Robert and V. Vargas.
	\newblock Gaussian multiplicative chaos revisited.
	\newblock {\em Ann. Probab.}, 38(2):605--631, 2010.
	
	\bibitem{GMC_zeta}
	E. Saksman and C. Webb.
	\newblock Multiplicative chaos measures for a random model of the Riemann zeta
	function.
	\newblock {\em arXiv preprint arXiv:1604.08378}, 2016.
	
	\bibitem{GMC_shamov}
	A. Shamov.
	\newblock On {G}aussian multiplicative chaos.
	\newblock {\em J. Funct. Anal.}, 270(9):3224--3261, 2016.

\bibitem{sheffield-cle} S. Sheffield, Exploration trees and conformal loop ensembles,
\emph{Duke Math. J.} { 147}: 79--129, 2009. 
	
	\bibitem{sheffield}
	S. Sheffield
	  \newblock Conformal weldings of random surfaces:  SLE and the quantum gravity zipper.
	 \newblock   {\em Ann.Probab.} 44(5), 3474–3545: 2016.

\bibitem{sw} S. Sheffield and W. Werner, Conformal Loop Ensembles: the Markovian characterization
and the loop-soup construction, \emph{Ann. Math.} { 176}: 1827-1917, 2012.
	
	\bibitem{stroock_chaos}
	D.W. Stroock.
	\newblock Homogeneous chaos revisited.
	\newblock In {\em S\'{e}minaire de {P}robabilit\'{e}s, {XXI}}, volume 1247 of
	{\em Lecture Notes in Math.}, pages 1--7. Springer, Berlin, 1987.

	\bibitem{S}
	D. Surgailis.
	\newblock Zones of attraction of self-similar multiple integrals.
	\newblock {\em Lithuanian Math. Journal}, 22(3):327--340, 1982.

\bibitem{symanzik} K.~Symanzik. Euclidean quantum field theory.
%In: Rendiconti della scuola di fisica Enrico Fermi, R. Jost (Ed.), XLV Corso, Academic Press (1969).
%Rendiconti della Scuola internazionale di Fisica Enrico Fermi", XLV Corso (R.~Jost ed.), pp. 152-223, Academic Press, New York, 1969.
In \emph{Local quantum theory}, Proceedings of the International School of Physics ``Enrico Fermi,'' course 45 (R.~Jost editor), pp. 152--223,
Academic Press, New York, 1969.

\bibitem{sznitman-notes} A.S. Sznitman. \emph{Topics in Occupation Times and Gaussian
Free Field}, Z\"urich Lectures in Advanced Mathematics, European Mathematical Society Publishing House, Z\"urich, 2012.
%Notes of the course ``Special topics in probability'' ETH Zurich (Spring term 2011).

	\bibitem{GMC_CUE_char_poly}
	C. Webb.
	\newblock The characteristic polynomial of a random unitary matrix and
	{G}aussian multiplicative chaos---the {$L^2$}-phase.
	\newblock {\em Electron. J. Probab.}, 20:no. 104, 21, 2015.

\bibitem{werner1} W. Werner. SLEs as boundaries of clusters of Brownian loops,
\emph{C. R. Acad. Sci.--Ser. I--Math.} { 337}: 481--486, 2003.

\bibitem{werner2} W.~Werner.
Some recent aspects of random conformally invariant systems,
in \emph{Les Houches Scool Proceedings: Session LXXXII, Mathematical Statistical Physics}
(A. Bovier, F. Dunlop, A. van Enter, J. Dalibard editors), pp. 57--98, Elsevier, 2006. %Amsterdam
%arXiv:math.PR/0511268 (2005).
	
	\bibitem{werner}
	W. Werner.
	\newblock The conformally invariant measure on self-avoiding loops.
	\newblock {\em J. Amer. Math. Soc.}, 21(1):137--169, 2008.
	
	\bibitem{weyl}
	H. Weyl.
	\newblock {\"U}ber die asymptotische Verteilung der Eigenwerte.
	\newblock {\em Nachrichten von der Gesellschaft der Wissenschaften zu
		G{\"o}ttingen, Mathematisch-Physikalische Klasse}, 1911:110--117, 1911.
	
\end{thebibliography}
\end{document}